\documentclass[11pt]{article}

\usepackage{amssymb,amsmath,amsthm}
\usepackage{mathrsfs}
\usepackage{marginnote}
\usepackage{graphicx}

\numberwithin{equation}{section}
\newtheorem{thm}{Theorem}[section]
\newtheorem{df}[thm]{Definition}
\newtheorem{prop}[thm]{Proposition}
\newtheorem{lem}[thm]{Lemma}
\newtheorem{cor}[thm]{Corollary}
\theoremstyle{definition}
\newtheorem{rem}[thm]{Remark}

\let\oldproofname=\proofname
\renewcommand{\proofname}{\rm\bf{\oldproofname}}

\newcommand{\N}{\mathbb{N}}
\newcommand{\Z}{\mathbb{Z}}

\newcommand{\R}{\mathbb{R}}
\newcommand{\C}{\mathbb{C}}
\newcommand{\T}{\mathbb{T}}

\newcommand{\cA}{\mathcal{A}}
\newcommand{\cB}{\mathcal{B}}
\newcommand{\cC}{\mathcal{C}}
\newcommand{\cD}{\mathcal{D}}
\newcommand{\cE}{\mathcal{E}}
\newcommand{\cF}{\mathcal{F}}
\newcommand{\cG}{\mathcal{G}}
\newcommand{\cH}{\mathcal{H}}
\newcommand{\cI}{\mathcal{I}}
\newcommand{\cJ}{\mathcal{J}}
\newcommand{\cK}{\mathcal{K}}
\newcommand{\cL}{\mathcal{L}}

\newcommand{\cN}{\mathcal{N}}
\newcommand{\cO}{\mathcal{O}}
\newcommand{\cR}{\mathcal{R}}

\newcommand{\WW}{\mathscr{W}}
\newcommand{\WWA}{\mathscr{W}^\mathscr{\omega}}
\newcommand{\LL}{\mathscr{L}}
\newcommand{\QQ}{\mathscr{Q}}

\renewcommand{\Re}{\mathop{\mathrm{Re}}}
\renewcommand{\Im}{\mathop{\mathrm{Im}}}
\newcommand{\dd}{\,{\rm d}}
\newcommand{\D}{{\rm d}}
\renewcommand{\div}{\mathop{\mathrm{div}}\nolimits}
\newcommand{\curl}{\mathop{\mathrm{curl}}}

\newcommand{\sgn}{{\mathrm{sgn}}}
\newcommand{\per}{{\mathrm{per}}}
\newcommand{\loc}{\mathrm{loc}}

\newcommand{\QED}{\mbox{}\hfill$\Box$}
\renewcommand{\:}{\thinspace :}
\newcommand{\ggamma}{\gamma_\star}
\newcommand{\bb}{\overline{b}}  
\newcommand{\rb}{\overline{r}}  
\newcommand{\Rp}{\R_+}
\newcommand{\Rpb}{\overline{\R}_+}
\newcommand{\eps}{\varepsilon}
\newcommand{\bbD}{\mathbb{D}}
\newcommand{\cP}{\mathcal{P}}
\newcommand{\cQ}{\mathcal{Q}}

\newdimen\texpscorrection
\newdimen\figcenter
\def\figurewithtex #1 #2 #3 #4 #5\cr{\null
  {\goodbreak\figcenter=\hsize\relax
  \advance\figcenter by -#4truecm
  \divide\figcenter by 2
  \begin{figure}[hbt]
  \vskip #3truecm\noindent\hskip\figcenter
  \includegraphics{#1}{\hskip\texpscorrection\input #2 }
  \vskip 0.8truecm{\baselineskip=0.8\baselineskip
  \noindent \vbox{\noindent {\footnotesize #5}}\par}
  \end{figure}}}
\def\point#1 #2 #3 {\rlap{\kern #1 truecm
\raise #2 truecm \hbox{#3}}}

\begin{document}

\title{Spectral stability of inviscid columnar vortices}

\author{
{\bf Thierry Gallay}\\
Institut Fourier\\
Universit\'e Grenoble Alpes, CNRS\\
100 rue des Maths\\
38610 Gi\`eres, France\\
{\small\tt Thierry.Gallay@univ-grenoble-alpes.fr}
\and
{\bf Didier Smets}\\
Laboratoire Jacques-Louis Lions\\
Sorbonne Universit\'e\\
4, place Jussieu\\
75005 Paris, France\\
{\small\tt Didier.Smets@sorbonne-universite.fr}}

\date{June 6, 2019}
\maketitle

\begin{center}
{\it Dedicated to the memory of Louis N. Howard}
\end{center}

\begin{abstract}
Columnar vortices are stationary solutions of the three-dimensional
Euler equations with axial symmetry, where the velocity field only
depends on the distance to the axis and has no component in the 
axial direction. Stability of such flows was first investigated by Lord 
Kelvin in 1880, but despite a long history the only analytical results 
available so far provide necessary conditions for instability under 
either planar or axisymmetric perturbations. The purpose of this paper 
is to show that columnar vortices are spectrally stable with respect to 
three-dimensional perturbations with no particular symmetry. Our
result applies to a large family of velocity profiles, including the 
most common models in atmospheric flows and engineering 
applications. The proof is based on a homotopy argument, which 
allows us to concentrate in the spectral analysis of the linearized operator
to a small neighborhood of the imaginary axis, where unstable 
eigenvalues can be excluded using integral identities and a careful 
study of the so-called critical layers. 
\end{abstract}

\section{Introduction}\label{sec1}

\noindent
An interesting open question in hydrodynamic stability theory is
whether the balance between the centrifugal force and the pressure 
gradient in axisymmetric vortex flows may lead to an instability even 
if the vorticity profile is monotone and the velocity field has no axial 
component. For incompressible perfect fluids, partial answers have 
been obtained under additional symmetry assumptions. For instance,  
in the restricted framework of two-dimensional flows, radially symmetric
vortices are known to be stable if the vorticity distribution is a monotone
function of the distance to the vortex center \cite{Ra1,MP}, but even in 
that idealized situation no sharp stability criterion seems to be available. 
In the three-dimensional case, the simplest vortex-like equilibria are 
{\em columnar vortices}, namely axisymmetric flows with no vertical 
velocity and no dependence upon the vertical coordinate. In such 
flows, all streamlines are horizontal circles centered on the vertical
symmetry axis. According to a celebrated result of Rayleigh \cite{Ra2}, 
columnar vortices are stable with respect to axisymmetric perturbations if 
the square of the velocity circulation along the streamlines is a
nondecreasing function of the distance to the symmetry axis, and that
condition is actually sharp \cite{Sy}.

A natural question arises from these centennial results: {\em When the
  vorticity profile is monotone and Rayleigh's condition is satisfied,
  are columnar vortices stable against three-dimensional perturbations
  with no particular symmetry\thinspace?} Although instabilities have
never been observed experimentally or numerically for such vortices in
the absence of axial flow, we could not find in the literature even a
plausible formal argument supporting the affirmative answer, see
Section~\ref{subsec13} below for a short historical discussion.  In
the present paper, we give a rigorous proof of spectral stability for
a large family of inviscid columnar vortices without imposing any
symmetry assumption on the class of allowed perturbations. We thus
provide an answer to an important question that dates back to the
pioneering work of Lord Kelvin \cite{Ke}, who was the first to
investigate the three-dimensional stability of vortex columns.

Before stating our results, we first describe the precise framework.
We start from the incompressible Euler equation in the whole space $\R^3$\:
\begin{equation}\label{eq:Euler3d}
  \partial_t u + (u\cdot\nabla)u \,=\, -\nabla p\,, \qquad
  \div u \,=\, 0\,,
\end{equation}
where $u = u(x,t) \in \R^3$ denotes the fluid velocity and
$p = p(x,t) \in \R$ the internal pressure. We mainly consider the
vorticity $\omega(x,t) = \curl u(x,t)$, which describes the local
rotation of the fluid particles. Since we are interested in the
stability of axially symmetric flows, it is convenient to use
cylindrical coordinates $(r,\theta,z)$ defined by $x_1 = r\cos\theta$,
$x_2 = r\sin\theta$, and $x_3 = z$. The velocity and vorticity fields
are then decomposed as follows\:
\begin{align*}
  u \,&=\, u_r(r,\theta,z,t) e_r + u_\theta(r,\theta,z,t) e_\theta 
  + u_z(r,\theta,z,t) e_z\,, \\
  \omega \,&=\, \omega_r(r,\theta,z,t) e_r + \omega_\theta(r,\theta,z,t) 
  e_\theta + \omega_z(r,\theta,z,t) e_z\,,
\end{align*}
where $e_r$, $e_\theta$, $e_z$ are unit vectors in the radial, azimuthal, 
and vertical directions, respectively. In these coordinates, the vorticity 
equation $\partial_t \omega + (u\cdot\nabla)\omega -(\omega\cdot\nabla)u = 0$ 
becomes
\begin{equation}\label{eq:vortcyl}
  \begin{split}
  \partial_t \omega_r + (u\cdot\nabla)\omega_r - (\omega\cdot\nabla)u_r 
  \,&=\, 0\,, \\
  \partial_t \omega_\theta + (u\cdot\nabla)\omega_\theta - (\omega\cdot
  \nabla)u_\theta \,&=\, \textstyle{\frac{1}{r}}\bigl(u_r \omega_\theta 
  - u_\theta\omega_r\bigr)\,, \\
  \partial_t \omega_z + (u\cdot\nabla)\omega_z - (\omega\cdot\nabla)u_z 
  \,&=\, 0\,,
  \end{split}
\end{equation}
where $u \cdot \nabla = u_r \partial_r + \frac1r u_\theta \partial_\theta 
+ u_z \partial_z$ and $\omega \cdot \nabla = \omega_r \partial_r + 
\frac1r \omega_\theta \partial_\theta + \omega_z \partial_z$. The velocity 
field satisfies the incompressibility condition
\begin{equation}\label{eq:div_cyl}
  \frac1r \partial_r (ru_r) + \frac1r \partial_\theta u_\theta
  + \partial_z u_z \,=\, 0\,,
\end{equation}
and can be expressed in terms of the vorticity by solving the 
linear elliptic system
\begin{equation}\label{eq:curl_cyl}
  \frac1r \partial_\theta u_z - \partial_z u_\theta \,=\, \omega_r\,, \qquad
  \partial_z u_r - \partial_r u_z \,=\, \omega_\theta\,, \qquad
  \frac1r \partial_r (r u_\theta) - \frac1r \partial_\theta u_r 
  \,=\, \omega_z\,.
\end{equation}

\subsection{Columnar vortices}\label{subsec11}

Columnar vortices are stationary solutions of 
\eqref{eq:Euler3d}--\eqref{eq:curl_cyl} of the particular form
\begin{equation}\label{eq:column}
  u \,=\, V(r) \,e_\theta\,, \qquad \omega \,=\, W(r) \,e_z\,,
  \qquad p \,=\, P(r)\,,
\end{equation}
where $V$ is the velocity profile and $W$ the vorticity distribution. 
The pressure $P$ inside the vortex is determined, up to an irrelevant 
additive constant, by the centrifugal balance $rP'(r) = V(r)^2$. 
Instead of $V$, we prefer using the angular velocity $\Omega(r) = 
V(r)/r$, which has the same physical dimension as the vorticity $W$.
As a consequence of \eqref{eq:curl_cyl}, we have
\begin{equation}\label{eq:velovort}
  W(r) \,=\, \frac1r \partial_r \bigl(rV(r)\bigr) \,=\, r\Omega'(r) 
  + 2\Omega(r)\,.
\end{equation}
Here are typical examples that are often considered in the literature\:

\smallskip\noindent{\bf 1.} {\em The Rankine vortex\:} 
\begin{equation}\label{eq:Rankine}
  \Omega(r) \,=\, \begin{cases} 1 & \hbox{if} \quad r \le 1\,, 
  \\ r^{-2} & \hbox{if} \quad r \ge 1\,, \end{cases} \qquad
  W(r) \,=\, \begin{cases} 2 & \hbox{if} \quad r < 1\,, 
  \\ 0 & \hbox{if} \quad r > 1\,. \end{cases} \qquad
\end{equation}
As is clear from \eqref{eq:Rankine}, the flow of Rankine's vortex 
corresponds to a rigid rotation for $r < 1$ and an irrotational 
motion for $r > 1$. Although non-physical because of the singularity 
at $r = 1$, this flow is relatively easy to analyze mathematically
due to the very simple form of the vorticity distribution $W$, which is 
a piecewise constant function. The dynamical stability of Rankine's
vortex was first investigated by L.~Kelvin as early as 1880
\cite{Ke}. 

\smallskip\noindent{\bf 2.} {\em The Kaufmann-Scully vortex\:} 
\begin{equation}\label{eq:KSvortex}
  \Omega(r) \,=\, \frac{1}{1+r^2}\,, \qquad
  W(r) \,=\, \frac{2}{(1+r^2)^2}\,, \qquad r > 0\,.
\end{equation}
This smooth vortex is characterized by a relatively slow decay of the 
vorticity distribution as $r \to \infty$. It has also a very simple 
analytical form, and is often used as a model for vortices that appear 
in atmospheric flows or in laboratory experiments, see e.g. 
\cite[Section 3.3.4]{AKO}. 

\smallskip\noindent{\bf 3.} {\em The Lamb-Oseen vortex\:} 
\begin{equation}\label{eq:LOvortex}
  \Omega(r) \,=\, \frac{1}{r^2}\Bigl(1 - e^{-r^2}\Bigr)\,,
  \qquad W(r) \,=\, 2\,e^{-r^2}\,, \qquad r > 0\,.
\end{equation}
Among all solutions of the form \eqref{eq:column}, the Lamb-Oseen
vortex plays a distinguished role in connection with the long-time
asymptotics of viscous planar flows. Indeed, if viscosity is taken
into account, it is known that all localized distributions of
vorticity evolve toward a Gaussian vorticity profile as $t\to +\infty$, 
see \cite{GW}. In particular, the Lamb-Oseen vortex is the only one in 
the above family which corresponds to a self-similar solution of the 
Navier-Stokes equations. 

\begin{rem}\label{rem:scaling}
Throughout this paper, it is understood that all independent and  dependent 
variables in the Euler equations \eqref{eq:Euler3d} are {\em dimensionless}.
Examples \eqref{eq:Rankine}--\eqref{eq:LOvortex} are normalized so that
the vortex core has a diameter of size $\cO(1)$, but that choice can 
be modified by a simple rescaling. Also, we assume without loss of generality
that all vortices are normalized so that $\Omega(0) = 1$, which implies 
$W(0) = 2$. 
\end{rem}

To study the dynamical stability of the columnar vortex \eqref{eq:column}, 
we look for solutions of \eqref{eq:vortcyl}, \eqref{eq:div_cyl} 
of the form 
\[
  u(r,\theta,z,t) \,=\, r \Omega(r) \,e_\theta + \tilde u(r,\theta,z,t)\,, 
  \qquad \omega(r,\theta,z,t) \,=\, W(r) \,e_z + \tilde \omega(r,\theta,z,t)\,,
\]
where $\Omega = V/r$ is the angular velocity of the vortex and 
$W$ the vorticity distribution given by \eqref{eq:velovort}. 
Inserting this Ansatz into \eqref{eq:vortcyl}, neglecting the quadratic 
terms in $\tilde u$ and $\tilde \omega$, and finally dropping the tildes for
notational simplicity, we arrive at the linearized evolution equations
\begin{align}\nonumber
  \partial_t \omega_r + \Omega(r) \partial_\theta \omega_r \,&=\, 
  W(r) \partial_z u_r\,, \\ \label{eq:omlin}
  \partial_t \omega_\theta + \Omega(r) \partial_\theta \omega_\theta 
  \,&=\, W(r) \partial_z u_\theta + r\Omega'(r)\omega_r\,, \\ \nonumber 
  \partial_t \omega_z + \Omega(r) \partial_\theta \omega_z \,&=\, 
  W(r) \partial_z u_z - W'(r) u_r\,,
\end{align}
which are the starting point of our analysis. Of course, the linear 
relations \eqref{eq:div_cyl}, \eqref{eq:curl_cyl} still hold for the 
perturbed velocity and vorticity. 

It is a classical observation that equations \eqref{eq:omlin} can be
considered as a self-contained evolution system for the vorticity
$\omega$, provided the velocity $u$ is expressed in terms of $\omega$
by solving the linear elliptic system \eqref{eq:div_cyl}, 
\eqref{eq:curl_cyl}. Once this is done, we can rewrite \eqref{eq:omlin} 
in the compact form
\begin{equation}\label{eq:omlin2}
  \partial_t \omega \,=\, L\omega\,,
\end{equation}
where $L$ is a vector-valued, nonlocal, first order differential 
operator. Our purpose is to study the spectral properties of 
that operator, and to show that $L$ has no spectrum outside the 
imaginary axis under general assumptions on the angular velocity 
$\Omega$ or the vorticity distribution $W$. 
 
Another fundamental remark is that system
\eqref{eq:vortcyl}--\eqref{eq:curl_cyl} is invariant under rotations
about the vertical axis, and under translations along that axis. Using
a Fourier series expansion with respect to the angular variable
$\theta$ and a Fourier transform in the vertical variable
$z$, we are led to consider velocities and vorticities of the
following particular form
\begin{equation}\label{eq:uomFour}
  u(r,\theta,z,t) \,=\, u_{m,k}(r,t)\,e^{im\theta}\,e^{ikz}\,,
  \qquad 
  \omega(r,\theta,z,t) \,=\, \omega_{m,k}(r,t)\,e^{im\theta}\,e^{ikz}\,,
\end{equation}
where $m \in \Z$ is the angular Fourier mode and $k \in \R$ is 
the vertical wave number. Here $u,\omega$ are complex-valued 
functions, but we impose that $\overline{u_{m,k}} = u_{-m,-k}$ and
$\overline{\omega_{m,k}} = \omega_{-m,-k}$ so as to obtain real 
functions after summing over all possible values of $m,k$. 
Dropping the subscripts $m,k$ for notational simplicity, we see 
that the perturbation equations \eqref{eq:omlin} translate into
\begin{align}\nonumber
 (\partial_t + im\Omega(r))\,\omega_r \,&=\, W(r) iku_r\,, \\ 
 \label{eq:vortFour}   
 (\partial_t + im\Omega(r))\,\omega_\theta \,&=\, W(r) iku_\theta
 + r\Omega'(r) \omega_r\,, \\ \nonumber
 (\partial_t + im\Omega(r))\,\omega_z \,&=\, W(r) iku_z - W'(r)u_r\,.
\end{align}
In addition, the following relations hold\:
\begin{equation}\label{eq:vort_cyl}
  \begin{array}{l} \omega_r \,=\, \frac{im}r u_z - ik u_\theta\,, \\[1mm] 
  \omega_\theta \,=\, ik u_r - \partial_r u_z\,, \\[1mm]
  \omega_z \,=\, \frac1r \partial_r (r u_\theta) - \frac{im}r u_r\,,
  \end{array} \qquad \hbox{and}\qquad 
  \frac1r\partial_r(ru_r) + \frac{im}r u_\theta + 
  ik u_z \,=\, 0\,. 
\end{equation}
As before, we can rewrite \eqref{eq:vortFour} in the compact form
\begin{equation}\label{eq:omlinFour}
  \partial_t \omega \,=\, L_{m,k}\omega\,,
\end{equation}
assuming that the velocity $u = u_{m,k}$ in \eqref{eq:vortFour} is
expressed in terms of the vorticity $\omega = \omega_{m,k}$ by solving
the linear relations \eqref{eq:vort_cyl} with appropriate boundary
conditions. The main properties of the Biot-Savart map
$\omega_{m,k} \mapsto u_{m,k}$ obtained in this way will be recalled
in Section~\ref{subsec61}.  Being an integral operator acting on
functions of the sole variable $r$, the generator $L_{m,k}$ in
\eqref{eq:omlinFour} is of course easier to study than the
original three-dimensional differential operator $L$ in
\eqref{eq:omlin2}.

\subsection{Statement of the results}\label{subsec12}

To state our results in a precise way, we first specify our hypotheses
on the unperturbed columnar vortex. We find it convenient to formulate 
these assumptions at the level of the {\em vorticity} profile $W$. Note
that, in view of \eqref{eq:velovort}, the angular velocity $\Omega$ can 
be expressed in terms of $W$ by the formula 
\begin{equation}\label{eq:Omrep}  
  \Omega(r) \,=\, \frac{1}{r^2}\int_0^r W(s) s\dd s\,, \qquad r > 0\,,
\end{equation}
and the derivative of $\Omega$ is in turn given by
\begin{equation}\label{eq:Omrep2}  
  \Omega'(r) \,=\, \frac{W(r)-2\Omega(r)}{r} \,=\, \frac{1}{r^3}
  \int_0^r W'(s)s^2\dd s\,, \qquad r > 0\,.
\end{equation}
In what follows, we denote $\Rp = (0,\infty)$ and $\Rpb=[0,\infty)$. 

\smallskip\noindent{\bf Assumption H1:} {\em The vorticity profile 
$W : \Rpb \to \Rp$ is a $\cC^1$ function satisfying $W'(0) = 0$,   
$W'(r) < 0$ for all $r > 0$, and the total circulation
\begin{equation}\label{eq:Winteg}
  2\pi \Gamma \,=\, 2\pi \int_0^\infty W(r) r\dd r
\end{equation}
of the columnar vortex is finite.} 

Under assumption H1 the angular velocity profile $\Omega \in \cC^1(\Rpb) 
\cap \cC^2(\Rp)$ given by \eqref{eq:Omrep} is positive and satisfies 
$\Omega(0) = W(0)/2$, $\Omega'(0) = 0$,  $\Omega'(r) < 0$ for all $r > 0$, 
and $\Omega(r) \sim \Gamma/r^2$ as $r \to \infty$. In particular, the 
{\em Rayleigh function} $\Phi : [0,\infty) \to \R$ defined by  
\begin{equation}\label{eq:Phidef}
  \Phi(r) \,=\, 2 \Omega(r) W(r)\,, \qquad r \ge 0\,,
\end{equation}
is positive everywhere. As a matter of fact, in our framework assumption H1 
corresponds exactly to the combination of Rayleigh's condition \cite{Ra2}
and of the two-dimensional stability criterion \cite{Ra1,MP}. We supplement 
it with the following:    

\smallskip\noindent{\bf Assumption H2:} {\em The $\cC^1$ function 
$J : \Rp \to \Rp$ defined by
\begin{equation}\label{eq:Jdef}
  J(r) \,=\, \frac{\Phi(r)}{\Omega'(r)^2}\,, \qquad  r > 0\,,
\end{equation}
satisfies $J'(r) < 0$ for all $r > 0$ and $rJ'(r)\to 0$ as $r\to \infty$.} 

This second assumption is more technical in nature, and certainly more
difficult to justify. We first observe that it is satisfied for the
Kaufmann-Scully vortex \eqref{eq:KSvortex}, because $J(r) = 1 + 1/r^2$
in that case, and a direct calculation that can be found in
Section~\ref{subsec67} below reveals that assumption H2 also holds for
the Lamb-Oseen vortex \eqref{eq:LOvortex}. A quantity corresponding to
\eqref{eq:Jdef} appears in the work of G.I. Taylor \cite{Tay} on the
stability of stratified shear flows; in that context it is called the
{\it local Richardson number} (see e.g. \cite[Chapter 6]{DR}). Its
relevance for stability was confirmed by Miles \cite{Mil} and Howard
\cite{How}. The ideas of Howard were translated into the columnar
vortex framework by Howard and Gupta \cite{HG}, where the quantity
\eqref{eq:Jdef} is also shown to play an important role in the
stability analysis for perturbations with nonzero angular Fourier mode
$m$ and nonzero vertical wave number $k$. Indeed, it is proved in
\cite{HG} that the linear operator $L_{m,k}$ in \eqref{eq:omlinFour}
has no unstable eigenvalue if
\begin{equation}\label{eq:HGmk}
	\frac{k^2}{m^2}\,J(r) \ge\, \frac14\,,
  \qquad \hbox{for all } r > 0\,,
\end{equation}
see also Proposition~\ref{prop:HG} below. Note that, in the case of
the Lamb-Oseen vortex, inequality \eqref{eq:HGmk} is always violated
for large $r > 0$ because $J(r) \to 0$ as $r \to \infty$, whereas
\eqref{eq:HGmk} holds for the Kaufmann-Scully vortex if and only if
$m^2 \le 4k^2$.  Although Howard and Gupta's result alone is not
sufficient, it plays a crucial role in our stability analysis in
Section~\ref{sec4}, where we have to distinguish two spatial regions
according to whether the local Richardson number $(k^2/m^2)J(r)$ is
greater or smaller than $1/4$. It turns out to be important for our
approach that inequality \eqref{eq:HGmk} either holds for all
$r \ge 0$, or is satisfied if and only if $r \le r_*$ for some
$r_* > 0$. The only way to enforce that property for all possible
values of $m$ and $k$ is to assume that the function $J$ in
\eqref{eq:Jdef} is decreasing. However, there is no evidence that
assumption H2 is more than a technical limitation, and we hope that
this question will be clarified in the future.

\begin{rem}\label{rem:limits}
Although this is not immediately obvious, assumption H2 implies 
the existence of a nonnegative number $\ell_\infty \ge 0$ such that
\begin{equation}\label{eq:limsdef}   
  \lim_{r \to \infty}r^4 W(r) \,=\, \ell_\infty\,, \qquad 
  \lim_{r \to \infty}r^5 W'(r) \,=\, -4\ell_\infty\,,
\end{equation}
see Section~\ref{subsec64} below. 
\end{rem}

Next, we specify the function space in which we study the linearized
operator $L_{m,k}$ defined in \eqref{eq:vortFour}, \eqref{eq:omlinFour}.
Since we used a Fourier decomposition to reduce our analysis to 
functions of the form \eqref{eq:uomFour}, it is natural to work 
in $L^2$-based function spaces. Given $m \in \Z$ and $k \in \R$, 
we thus define the enstrophy space
\begin{equation}\label{eq:Xmkdef}
  X_{m,k} \,=\, \Bigl\{\omega \in L^2(\R_+,r\dd r)^3 \,\Big|\,
  \frac1r\partial_r(r\omega_r) + \frac{im}r \omega_\theta + 
  ik \omega_z = 0\Bigr\}\,,
\end{equation}
equipped with the norm
\[
  \|\omega\|_{L^2}^2 \,=\, \int_0^\infty |\omega(r)|^2 \,r\dd r
  \,, \qquad \hbox{where}\quad |\omega|^2 \,=\, |\omega_r|^2 
  + |\omega_\theta|^2 + |\omega_z|^2\,.
\]
It is not difficult to verify that the generator $L_{m,k}$ of the 
linearized evolution equation \eqref{eq:omlinFour} defines 
a bounded linear operator in the space $X_{m,k}$ if $k \neq 0$, 
see Proposition~\ref{prop:AB} below. With this observation in 
mind, we can formulate our first main result\:

\begin{thm}\label{thm:main1} 
Consider a columnar vortex whose vorticity profile $W$ satisfies 
assumptions H1, H2 above. Given $m \in \Z$ and $k \neq 0$, let 
$L_{m,k}$ be the generator of the linearized evolution 
\eqref{eq:omlinFour}. Then the spectrum of $L_{m,k}$ in the enstrophy 
space $X_{m,k}$ satisfies
\begin{equation}\label{eq:spectrum}
  \sigma(L_{m,k}) \,\subset\, i \R\,. 
\end{equation}
\end{thm}

\begin{rem}\label{rem:morespec}
The proof actually shows that, under the normalization
condition $W(0) = 2$, $\sigma(L_{m,k})$ consists of essential spectrum 
filling the closed interval 
$\{-imb\,|\, 0 \le b \le 1\} \subset i\R$, and of a countable family
of simple, purely imaginary eigenvalues that accumulate only at
$-im \in i\R$. These eigenvalues are well studied in the physical
literature (a brief account is given in Section \ref{subsec13} below),
and the corresponding eigenfunctions are referred to as {\em Kelvin
vibration modes}. The main contribution of the present paper is to
show that the operator $L_{m,k}$ has no eigenvalue outside the
imaginary axis, if the vorticity profile $W$ satisfies assumptions 
H1, H2. It is interesting to note that this result remains valid for 
the Rankine vortex \eqref{eq:Rankine} which does not satisfy our 
hypotheses, see Section~\ref{subsec62} below. 
\end{rem}

\begin{rem}\label{rem:2Dexcluded}
The particular case $k = 0$, which corresponds to two-dimensional 
perturbations, is excluded in Theorem~\ref{thm:main1} because 
the function space $X_{m,k}$ is not appropriate in that situation. 
This is essentially due to the fact that the two-dimensional 
Biot-Savart law is ill-defined for vorticities in the enstrophy 
space. The problem can be eliminated by introducing
a radial weight that ensures a faster decay of $\omega(r)$ as 
$r \to \infty$, or alternatively by working in the energy space
as mentioned in Remark~\ref{rem:nonper} below. However, since the 
two-dimensional stability of radially symmetric vortices is already 
well documented, we chose to ignore these technical issues and to 
concentrate here on the genuinely three-dimensional case $k \neq 0$, 
which was essentially unexplored until now. 
\end{rem}

According to Theorem~\ref{thm:main1}, for any $s \in \C$ with
$\Re(s) \neq 0$, the resolvent operator $(s-L_{m,k})^{-1}$ is well
defined and bounded in the space $X_{m,k}$ if $m \in \Z$ and
$k \neq 0$. Actually, one can prove that the resolvent is uniformly
bounded for all $m \in \Z$ and for all nonzero $k$ in the
one-dimensional lattice $\Z k_0$, where $k_0 > 0$ is
arbitrary. Returning to the full linearized evolution
\eqref{eq:omlin2}, this proves spectral stability of the 
generator $L$ in the space
\begin{equation}\label{eq:L2per}
  \dot L^2_{\sigma,\per,h} \,=\, \Bigl\{\omega \in L^2(\R^2 \times 
  \T_h)^3 \,\Big|\, \div \omega = 0\,,~\int_0^h \omega(x_1,x_2,x_3) 
  \dd x_3 = 0\,\Bigr\}\,,
\end{equation}
where $\T_h = \R/(\Z h)$ and $h = 2\pi/k_0$ is the vertical period. 
We can thus state our second main result\:

\begin{thm}\label{thm:main2}
Under the assumptions of Theorem~\ref{thm:main1}, let $L$ denote 
the full linearized operator in \eqref{eq:omlin2}. Then, for any 
$h > 0$, the spectrum of $L$ in the space $\dot L^2_{\sigma,\per,h}$ 
satisfies
\begin{equation}\label{eq:spectrum2}
  \sigma(L) \,=\, i \R\,. 
\end{equation}
\end{thm}

\begin{rem}\label{rem:nonper}
The reason for restricting ourselves to functions with zero average in
the vertical direction was explained in Remark~\ref{rem:2Dexcluded}.
The same technical limitation prevents us from considering
perturbations in the enstrophy space $L^2_\sigma(\R^3)$, without
assuming periodicity in the vertical direction, because in that case
all values of the vertical wave number $k \in \R$ have to be taken 
into account. In a subsequent work \cite{GaSm2}, we use 
Theorem~\ref{thm:main1} to obtain the equivalent of Theorem \ref{thm:main2}
for the Euler equation in velocity formulation. There we consider 
perturbations in the energy space, and we also obtain semigroup 
estimates for the linearized operator at a columnar vortex. 
\end{rem}

In the proof of Theorems~\ref{thm:main1} and \ref{thm:main2}, 
we find it convenient to normalize our velocity and vorticity 
profiles so that $\Omega(0) = 1$ and $W(0) = 2$. This leads to the 
following definition\:

\begin{df}\label{def:Hclass}
We denote by $\WW$ the class of all vorticity profiles $W : \Rpb 
\to \Rp$ satisfying the assumptions H1, H2 above, as well as the 
normalizing condition $W(0) = 2$. 
\end{df}

It is worth emphasizing here that assumption H2 involves the  
function $J$ defined in \eqref{eq:Jdef}, which depends nonlinearly on 
the vorticity profile $W$. As a consequence, our family of admissible
profiles is not a vector space, and the class $\WW$ introduced 
in Definition~\ref{def:Hclass} is not even a convex set. However, 
we shall prove in Section~\ref{subsec64} that any profile $W \in \WW$ 
is entirely determined by the auxiliary function
\begin{equation}\label{eq:Qdef0}
  Q(r) \,=\, \frac{1}{\sqrt{1+J(r)}}\,, \qquad r > 0\,,
\end{equation}
and that the class $\WW$ can be described by simple linear constraints 
at the level of the function $Q$. This makes it possible to perform
continuous interpolation and approximation within the class $\WW$, and
such tools will play a crucial role in the proof of Theorem~\ref{thm:main1}.

\begin{rem}\label{rem:topo}
If we equip the class $\WW$ with the topology of $\cC^1_b(\Rpb)$, the
Banach space of all bounded continuously differentiable functions on
$\Rpb$ with bounded derivative, it is easily verified that the linear
operator operator $L_{m,k} \in \cL(X_{m,k})$ depends continuously on
the vorticity profile $W \in \WW$, see Lemma~\ref{lem:contiW2L}
below. In particular, isolated eigenvalues of $L_{m,k}$ outside the
imaginary axis (if they are any) vary continuously when $W$ is
perturbed in that topology. This implies that the conclusion
\eqref{eq:spectrum} of Theorem~\ref{thm:main1} remains valid for any
vorticity profile that belongs to the closure of the class $\WW$ in
$\cC^1_b(\Rpb)$. This larger class contains vorticities $W$ that are
not strictly decreasing functions of the radius $r$, and may even be
compactly supported.
\end{rem}

\subsection{Previous results and perspectives}\label{subsec13}

The first historical contribution regarding the stability of columnar
vortices in incompressible fluids is of course the seminal work
\cite{Ke} by Kelvin. In that study, the focus is put on neutral modes,
namely eigenmodes of the linearized Euler equation that correspond to
purely imaginary eigenvalues; these were later termed ``Kelvin
vibration modes". As Kelvin expresses it: ``{\it The problem thus
  solved is the finding of the periodic disturbance in the motion of
  rotating liquid} [...]". The computations in \cite{Ke} are performed
in situations where the underlying axisymmetric flow has piecewise
constant vorticity; this exactly corresponds to what was called the
Rankine vortex in Section \ref{subsec11} above. However, Kelvin waves
are observed to play an important role in the dynamics of the Euler
equation for a much wider variety of profiles, and were actively
studied in the literature since then (in most cases numerically, or
using asymptotic expansions combined with physical arguments).  In the
case of the Lamb-Oseen vortex, important contributions were made in
particular by Le Diz\`es and Lacaze \cite{LL} and Fabre, Sipp and
Jacquin \cite{FSJ}, both in the inviscid case and in the vanishing
viscosity limit. Unlike Kelvin (who had no computer account!), the
authors of \cite{LL,FSJ} also consider the possibility of eigenvalues
off the imaginary axis. One of the conclusions of \cite{FSJ} based on
their numerical findings is that ``[...]{\it no amplified modes were
found, a result which demonstrates the stability of the Lamb-Oseen
vortex.''}

In a different direction, Rayleigh \cite{Ra1,Ra2} initiated the study
of necessary conditions for columnar vortex instability\footnote{Or
  equivalently sufficient conditions for their stability; 
  in the present work stability is only understood in the
  spectral sense, meaning the absence of eigenvalues with positive
  real part.}. Although it may certainly be found physically
convincing, the original argument \cite{Ra2} leading to Rayleigh's
criterion cannot be easily transposed into rigorous mathematical
terms. Instead, the approach followed by Howard and Gupta \cite{HG},
which we consider one of the most interesting and important
contributions so far, is both rigorous and elementary.  This
remarkable work contains most importantly a non-conclusive but
enlightening section called ``Remarks on the non-axisymmetric case'',
in which the partial stability criterion \eqref{eq:HGmk} can be found.
The authors write: ``{\it The overall conclusion of this consideration
  of the non-axisymmetric case is thus essentially negative: the
  methods used to derive the Richardson number and semicircle results
  in the axisymmetric case reproduce the known results of Rayleigh for
  two-dimensional perturbations and pure axial flow, but seem to give
  very little more. In fact the present situation with regard to
  non-axisymmetric perturbations seems to be very unsatisfactory from
  a theoretical point of view.}"

Attempts have been made to derive necessary conditions for instability
extending Rayleigh's criterion to non-axisymmetric perturbations. One
such criterion was proposed by Billant and Gallaire \cite{BG},
following earlier work by Leibovich and Stewartson \cite{LS}, and
applies in a given Fourier sector. It is relatively simple to state
but requires a number of a posteriori checks which could be more
difficult to perform. As the authors mention, in all the situations
they tested the most unstable modes were always the axisymmetric ones
(this is reminiscent of Squire's theorem in the context of viscous
shear flows), and therefore, in practice, Rayleigh's criterion appears
to be sufficient to detect potential instabilities. Yet, a priori
estimates on the possible growth in a given Fourier sector are certainly
interesting per se.

Spectral stability of course does not imply stability of the flow for
a Hamiltonian system such as Eq.~\eqref{eq:Euler3d}. In a celebrated
paper \cite{Arn1,Arn2}, Arnold derived a nonlinear stability criterion
for stationary solutions of the Euler equations, which are viewed as
critical points of the kinetic energy functional over the manifold of
isovortical vector fields, and he treated in detail the case of 2D
flows. His approach was subsequently extended by Szeri and Holmes
\cite{SzHo} and applied to axisymmetric perturbations of columnar
vortices. A few years later, Rouchon \cite{Rou} proved that the conditions 
in Arnold's criterion are never satisfied if one considers genuinely 3D 
perturbations of nontrivial stationary flows. An intermediate step between 
spectral and nonlinear stability is linear stability, which consists in 
controling the growth of the semigroup generated by the linearized operator in 
Theorem~\ref{thm:main2}. Preliminary results in that 
direction can be found in the subsequent work \cite{GaSm2}.

We close this section mentioning that a number of interesting
phenomena are known to arise, as far as instabilities are concerned,
when the base flow possesses an additional axial component. Some of
the works already quoted, and many others, do consider that situation
as well. Since we did not investigate it at all in this work, we
keep that discussion for another occasion.

\subsection{Organization of the paper}\label{subsec14}

Our strategy to prove Theorems~\ref{thm:main1} and \ref{thm:main2} can
be explained as follows. In a first step, we show in
Section~\ref{sec2} that the essential spectrum of the operator
$L_{m,k}$ is purely imaginary. The rest of the spectrum consists of
isolated eigenvalues with finite multiplicity, and the corresponding
eigenfunctions are solutions of a second order differential equation
involving a complex potential that depends on $m$, $k$, and the
spectral parameter $s$.  The eigenvalue equation is difficult to study
in general, but using techniques that date back to Rayleigh
\cite{Ra1,Ra2} it is easy to verify that it has no nontrivial solution
with $\Re(s) \neq 0$ when the perturbations are either axisymmetric
($m = 0$) or two-dimensional ($k = 0$). In Section~\ref{sec3}, we
establish a few preliminary results in the case were $m \neq 0$ and
$k \neq 0$. In particular, we derive useful identities satisfied by
any nontrivial eigenfunction, and we recover the stability criterion
\eqref{eq:HGmk} of Howard and Gupta. The core of the proof of
Theorem~\ref{thm:main1} is Section~\ref{sec4}.  We construct a
suitable homotopy between the vorticity profile $W \in \WW$ and a
reference profile for which stability in the corresponding Fourier
sector $X_{m,k}$ is known by Howard and Gupta's criterion.  By a
continuity argument, this strategy allows us to reduce the problem to
proving the absence of unstable eigenvalues {\em arbitrarily close to
  the imaginary axis}, for a one-parameter family of profiles in the
class $\WW$. A delicate combination of integral identities and
comparison arguments relying on assumption H2 are then used to perform
such a ``critical layer analysis'' and hence to preclude the existence
of unstable eigenvalues in the large. Finally, in Section~\ref{sec5},
we prove uniform resolvent estimates for the linear operator $L_{m,k}$
outside the imaginary axis, which imply that the full linearization
$L$ has indeed no spectrum in that region when acting on the space
$\dot L^2_{\sigma,\per,h}$ for any $h > 0$. This is precisely the
conclusion of Theorem~\ref{thm:main2}. The last section is an appendix
were several auxiliary results are established. In particular, we give
useful estimates for the Biot-Savart law in the Fourier sector indexed
by $m,k$, we prove the stability of Rankine's vortex
\eqref{eq:Rankine} which is not covered by Theorem~\ref{thm:main1},
and we explain how to perform continuous interpolation and
approximation in the nonlinear class $\WW$.

\medskip\noindent{\bf Acknowledgements.}
The authors were partially supported by grants ANR-13-BS01-0003-01 (Th.G.)
and ANR-14-CE25-0009-01 (D.S.) from the ``Agence Nationale de la Recherche''.
They benefited from discussions with S. Le Diz\`es, in particular 
during the meeting ``Vortex et solitons pour les fluides classiques et
quantiques" (CIRM, Marseille, 2012) where this work was initiated,  
and also from insightful remarks from an anonymous referee. 

\section{Formulation of the spectral problem}
\label{sec2}

Let $W$ be a vorticity profile in the class $\WW$, and let $\Omega$ be
the corresponding angular velocity defined by \eqref{eq:Omrep}.  For a
fixed value of the angular Fourier mode $m \in \Z$ and of the vertical
wave number $k \in \R$, we consider the linear operator $L_{m,k}$
introduced in \eqref{eq:omlinFour}. In view of \eqref{eq:vortFour}, we
have the natural decomposition
\begin{equation}\label{eq:Lmkdef}
  L_{m,k} \,=\, A_m + B_{m,k}\,, 
\end{equation}
where $A_m$ is the multiplication operator defined by 
\begin{equation}\label{eq:Adef}
  A_m \omega \,=\, -im \Omega(r)\omega + r\Omega'(r)\omega_r 
  \,e_\theta\,,
\end{equation}
and $B_{m,k}$ is the following nonlocal perturbation\:
\begin{equation}\label{eq:Bdef}
  B_{m,k} \omega \,=\, ik W(r)u - W'(r)u_r \,e_z\,.
\end{equation}
Here $u = (u_r,u_\theta,u_z)$ denotes the velocity obtained from 
the vorticity $\omega = (\omega_r,\omega_\theta,\omega_z)$ by solving 
the linear PDE system \eqref{eq:vort_cyl} with appropriate 
boundary conditions. We refer the reader to Section~\ref{subsec61} 
below for a discussion of the map $\omega \mapsto u$, which 
we call the Biot-Savart law in the Fourier subspace indexed by 
$m$ and $k$. Our main goal in this paper is to study the spectral 
properties of the operator $L_{m,k}$ acting on the  enstrophy 
space $X_{m,k}$ defined by \eqref{eq:Xmkdef}. 

The following simple result is the starting point of our analysis. 

\begin{prop}\label{prop:AB} Fix $m \in \Z$ and $k \in \R\setminus\{0\}$.
\\[1mm]
1) The linear operator $A_m$ defined by \eqref{eq:Adef} is bounded 
in $X_{m,k}$ with spectrum given by
\begin{equation}\label{eq:spAm}
  \sigma(A_m) \,=\, \Bigl\{z \in \C\,\Big|\, z = -imb 
  \hbox{ for some } b \in [0,1]\Bigr\}\,.
\end{equation}
This spectrum is purely continuous if $m \neq 0$, and reduces to
a single eigenvalue if $m = 0$. \\[2mm] 
2) The linear operator $B_{m,k}$ defined by \eqref{eq:Bdef} is compact in 
$X_{m,k}$. 
\end{prop}

\begin{proof}
Given $s \in \C$ and $f = (f_r,f_\theta,f_z) \in X_{m,k}$, the resolvent 
equation $(s-A_m)\omega = f$ is equivalent to the linear system
\begin{equation}\label{eq:Amres}
  (s + im\Omega(r))\omega_r \,=\, f_r\,, \quad
  (s + im\Omega(r))\omega_\theta \,=\, f_\theta + r\Omega'(r)\omega_r\,, \quad
  (s + im\Omega(r))\omega_z \,=\, f_z\,.
\end{equation}
As $W \in \WW$, we know that $\Omega : [0,\infty) \to \Rp$ is strictly
decreasing with $\Omega(0) = 1$ and $\Omega(r) \to 0$ as $r \to \infty$. 
Thus, if $s \neq -imb$ for all $b \in[0,1]$, 
 the quantity $|s + im\Omega(r)|$
is bounded away from zero, and it follows that system \eqref{eq:Amres}
has a unique solution $\omega \in X_{m,k}$ satisfying $\|\omega\|_{L^2} 
\le C(s)\|f\|_{L^2}$. On the other hand, if $m \neq 0$ and $s = -imb$ for 
some $b \in [0,1]$, it is easy to verify that the operator $s-A_m$ is
one-to-one but not onto (its range is dense but strictly contained in
$X_{m,k}$), so that $s$ belongs to the continuous spectrum of
$A_m$. Finally, if $m = 0$, it is clear that $s = 0$ is an eigenvalue
of $A_m$, with infinite multiplicity. This proves the first part.

We next consider the operator $B_{m,k}$. If $\omega \in X_{m,k}$ and 
$\|\omega\|_{L^2} \le 1$, Proposition~\ref{prop:estBS} shows that the 
associated velocity field $u$ satisfies $\|\partial_r u\|_{L^2} + 
\|k u\|_{L^2} \le C$ for some universal constant $C > 0$. This 
gives a uniform bound on $u$ in $H^1(\Rp,r\dd r)$ since we assume 
that $k \neq 0$. By the Fr\'echet-Kolmogorov theorem, we deduce 
that the map $\omega \mapsto B_{m,k}\omega = ik W(r)u - W'(r)u_r \,e_z$ 
is compact in $X_{m,k}$, because the functions $W$ and $W'$ are
bounded and converge to zero as $r \to \infty$.
\end{proof}

Proposition~\ref{prop:AB} shows in particular that, for any $m \in \Z$
and any $k \in \R\setminus\{0\}$, the linearization $L_{m,k} = 
A_m + B_{m,k}$ defines a bounded operator in the space $X_{m,k}$. 
Moreover, as $B_{m,k}$ is compact, the {\em essential spectrum} of 
$L_{m,k}$ is the same as the (essential) spectrum of $A_m$, namely 
the closed interval $I_m = \{-imb\,|\, 0 \le b \le 1\} \subset i\R$, see 
\cite[Theorem~I.4.1]{EE}. Note that, in the present case, the various 
definitions of the essential spectrum listed in \cite[Section~I.4]{EE} 
all coincide. This implies that the spectrum of $L_{m,k}$ outside the
interval $I_m$ entirely consists of isolated eigenvalues with finite 
multiplicities, which can accumulate only on the essential spectrum. 
The proof of Theorem~\ref{thm:main1} is thus reduced to showing that 
all isolated eigenvalues of $L_{m,k}$ actually lie on the imaginary axis.

\begin{rem}\label{rem:symmetries}
As the functions $\Omega$, $W$ are real-valued, it is not difficult to
verify, using the definitions \eqref{eq:Adef}, \eqref{eq:Bdef} and the
relations \eqref{eq:vort_cyl} between $u$ and $\omega$, that the
spectrum of $L_{m,k}$ in $X_{m,k}$ has the following symmetries\:
\begin{equation}\label{eq:symetriespectre}
  \sigma(L_{m,k}) \,=\, \sigma(L_{m,-k}) \,=\, -\sigma(L_{-m,k})\,,
  \qquad\hbox{and}\quad \sigma(L_{m,k}) \,=\, -\overline{\sigma(L_{m,k})}\,.
\end{equation}
The corresponding mappings between eigenspaces are also easy to 
establish. In particular, the last relation in \eqref{eq:symetriespectre}
means that the spectrum of $\sigma(L_{m,k})$ is symmetric with respect
to the imaginary axis, a property that will be used later on. 
\end{rem}

As a first step in the proof of Theorem~\ref{thm:main1}, we derive 
an equation for the eigenfunctions of the operator $L_{m,k}$ 
corresponding to eigenvalues outside the essential spectrum. In what 
follows, we thus assume that $s \in \C$ is an isolated eigenvalue of
$L_{m,k}$ with eigenfunction $\omega = (\omega_r,\omega_\theta,\omega_z) 
\in X_{m,k}$, and we denote by $u = (u_r,u_\theta,u_z)$ the velocity field
associated with $\omega$ via the Biot-Savart law, see 
Section~\ref{subsec61}. As in \cite{DR}, we define
\begin{equation}\label{eq:gammadef}
  \gamma(r) \,=\, s + im\Omega(r)\,, \qquad r > 0\,.
\end{equation}
Since $s$ does not belong to the essential spectrum of $L_{m,k}$
by assumption, it follows from Proposition~\ref{prop:AB} that
$\gamma(r) \neq 0$ for all $r > 0$. 

In view of \eqref{eq:vortFour}, the eigenvalue equation reads
\begin{align}\nonumber
  \gamma(r) \omega_r \,&=\, ik W(r) u_r\,, \\ \label{eq:eigeq}
  \gamma(r) \omega_\theta \,&=\, ik W(r)u_\theta + r\Omega'(r) 
    \omega_r\,, \\ \nonumber
  \gamma(r) \omega_z \,&=\, ik W(r) u_z - W'(r) u_r \,,
\end{align}
where $r\Omega'(r) = W(r) - 2\Omega(r)$ by \eqref{eq:velovort}. 
If we express the vorticity $\omega$ in terms of $u$ using the 
relations \eqref{eq:vort_cyl}, we obtain the equivalent system
\begin{align}\label{eq:sys1}
  ikW(r)u_r + ik \gamma(r)u_\theta - \frac{im\gamma(r)}{r}u_z 
    \,&=\, 0\,, \\ \label{eq:sys2}
  ik\gamma(r)u_r -2ik\Omega(r)u_\theta -\partial_r(\gamma(r)u_z) 
    \,&=\, 0\,,\\ \label{eq:sys3}
  \Bigl(W'(r) - \frac{im\gamma(r)}{r}\Bigr)u_r + \gamma(r)
  \frac{1}r \partial_r (ru_\theta) - ik W(r) u_z \,&=\, 0\,.
\end{align}
Assuming for the moment that $k \neq 0$, it is straightforward 
to verify that the relations \eqref{eq:sys1}--\eqref{eq:sys3} 
together imply the incompressibility condition
\begin{equation}\label{eq:divu2}
  \frac1r\partial_r(ru_r) + \frac{im}r u_\theta + ik u_z \,=\, 0\,.
\end{equation}

To reduce system \eqref{eq:sys1}--\eqref{eq:divu2} to a single 
equation, we first express the azimuthal velocity $u_\theta$
in terms of $u_r, u_z$ using \eqref{eq:sys1}, and replace it
into \eqref{eq:sys2}, \eqref{eq:divu2} to obtain the $2 \times 2$ 
system
\begin{align}\label{eq:sys4}
  \Bigl(\partial_r^* - \frac{imW(r)}{r\gamma(r)}\Bigr)u_r + ik 
  \Bigl(1 + \frac{m^2}{k^2 r^2}\Bigr)u_z \,&=\, 0\,, \\ \label{eq:sys5}
  \Bigl(\partial_r + \frac{imW(r)}{r\gamma(r)}\Bigr)u_z - ik 
  \Bigl(1 + \frac{\Phi(r)}{\gamma(r)^2}\Bigr)u_r \,&=\, 0\,,
\end{align}
where $\Phi = 2\Omega W$ is the Rayleigh function and $\partial_r^* = 
\partial_r + \frac1r$. Next, observing that the coefficient of 
$u_z$ in \eqref{eq:sys4} does note vanish, we can divide 
\eqref{eq:sys4} by that coefficient and apply the differential
operator $\partial_r + \frac{imW}{r\gamma}$ to obtain, with the 
help of \eqref{eq:sys5}, the following second-order differential 
equation for the radial velocity\:
\begin{equation}\label{eq:sys6}
  \Bigl(\partial_r + \frac{imW(r)}{r\gamma(r)}\Bigr) \frac{r^2}{m^2 
  + k^2 r^2}\Bigl(\partial_r^* - \frac{imW(r)}{r\gamma(r)}
  \Bigr)\,u_r  \,=\, \Bigl(1 + \frac{\Phi(r)}{\gamma(r)^2}\Bigr)u_r\,.
\end{equation}
If we expand the product in the left-hand side, we find after 
straightforward calculations 
\begin{equation}\label{eq:main}
   -\partial_r \biggl(\frac{r^2 \partial_r^* u_r}{m^2 + k^2 r^2}\biggr) 
   + \biggl\{1 + \frac{1}{\gamma(r)^2}\frac{k^2 r^2 
  \Phi(r)}{m^2 + k^2r^2} + \frac{imr}{\gamma(r)}\partial_r
  \Bigl(\frac{W(r)}{m^2+k^2r^2}\Bigr)\biggr\}u_r \,=\, 0\,,
\end{equation}
see also \cite[Eq.~(15.26)]{DR}. This is the desired eigenvalue
equation, which will be our main concern in the rest of this paper. It
is formulated in terms of the radial velocity $u_r$, which satisfies
$u_r \in H^1(\Rp,r\dd r)$ according to Proposition~\ref{prop:estBS}.
In fact, we also have $u_r \in H^2_\loc(\Rp)$ in view of the 
divergence-free condition \eqref{eq:divu2}. 

\begin{rem}\label{rem:00}
In the case where $k = 0$, a much simpler calculation shows
that the eigenvalue equation is still given by \eqref{eq:main} if 
$m \neq 0$, although the derivation above is not correct. 
If $k = m = 0$, equation \eqref{eq:main} is of course meaningless, 
but in that case it is obvious that system \eqref{eq:eigeq} 
has no nontrivial solution for $s \neq 0$. 
\end{rem}

Summarizing the arguments developed so far, the proof of
Theorem~\ref{thm:main1} can be reduced to showing that, for all
$m \in \Z$ and all $k \in \R\setminus\{0\}$, the eigenvalue equation
\eqref{eq:main} has no nontrivial solution $u_r \in H^1(\R_+,r\dd r)$
if the spectral parameter $s \in \C$ satisfies $\Re(s) \neq 0$. This
is a difficult task in general, which we postpone to Sections~\ref{sec3}
and \ref{sec4}. For the time being, we just mention two important 
particular cases which are relatively easy to handle.  

\subsection{The axisymmetric case}\label{subsec21}

In the axisymmetric case $m = 0$, Proposition
\ref{prop:AB} asserts that the essential spectrum of 
$L_{0,k}$ is reduced to zero, and therefore away from the origin there
may only exist eigenvalues with finite multiplicity. The spectral function
\eqref{eq:gammadef} is constant in that case, and the stability equation
\eqref{eq:main} reduces to
\begin{equation}\label{eq:main_axi}
  -\partial_r \partial_r^* u_r + k^2 \Bigl(1 + \frac{\Phi(r)}{s^2}
  \Bigr)u_r \,=\, 0\,.
\end{equation}
The following classical result dates back to the work of 
L. Rayleigh \cite{Ra2}, and is reproduced here for the 
reader's convenience. 

\begin{prop}\label{prop:axi} Assume that the Rayleigh function $\Phi$ 
is nonnegative. Then the eigenvalue equation \eqref{eq:main_axi} 
has no nontrivial solution $u_r \in H^1(\R_+,r\dd r)$ if $\Re(s) \neq 0$. 
\end{prop}

\begin{proof}
According to Remark~\ref{rem:00}, we can suppose that $k \neq 0$. 
Assume that $u_r \in H^1(\R_+,r\dd r)$ is a nontrivial solution of 
\eqref{eq:main_axi} for some $s \in \C\setminus\{0\}$. Multiplying 
both sides of \eqref{eq:main_axi} by $r \bar u_r$ and 
integrating the resulting expression over $\R_+$, we obtain the 
useful relation
\begin{equation}\label{eq:axirel}
  \int_0^\infty \Bigl\{|\partial_r^* u_r|^2 + k^2 \Bigl(1 + 
  \frac{\Phi(r)}{s^2}\Bigr)|u_r|^2 \Bigr\}r \dd r \,=\, 0\,.  
\end{equation}
By assumption we have $\int_0^\infty \Phi |u_r|^2 r\dd r > 0$, because
$u_r$ is a nontrivial solution of \eqref{eq:main_axi} and $\Phi$ is a
nonnegative function with $\Phi(0) > 0$.  Thus taking the imaginary
part of \eqref{eq:axirel} we deduce that $\Im(s^2) = 0$, hence
$s \in \R$ or $s \in i\R$. The first possibility is excluded 
by taking the real part of \eqref{eq:axirel}, hence we conclude
that $s \in i\R$. 
\end{proof}

\begin{rem}\label{rem:Synge}
Actually it was observed by Synge \cite{Sy} that the Rayleigh 
stability criterion $\Phi \ge 0$ is not only sufficient, but also 
necessary in the axisymmetric case. Indeed, we know that 
$\Phi(0) = W(0)^2 > 0$, and for localized vortices we always
have $\Phi(r) \to 0$ as $r \to \infty$. Now, assume that 
$\Phi(\bar r) < 0$ for some $\bar r > 0$, and consider the 
Schr\"odinger equation
\begin{equation}\label{eq:semic}
  -s^2 \partial_r \partial_r^* u_r + k^2 \Bigl(s^2 + \Phi(r)
  \Bigr)u_r \,=\, E u_r, \qquad r > 0\,,
\end{equation}
in the semiclassical limit where $0 < s \ll 1$. As the potential term
$s^2 + \Phi(r)$ takes negative values near $r = \bar r$, it is well
known that the operator in \eqref{eq:semic} has negative eigenvalues
$E$ if $s > 0$ is sufficiently small, see e.g. \cite{Si,HS}. In fact,
the number of negative eigenvalues increases unboundedly as $s \to 0$,
and this implies by continuity that Eq.~\eqref{eq:semic} with $E = 0$,
or equivalently Eq.~\eqref{eq:main_axi}, has a nontrivial solution
$u_r \in H^1(\R_+,r\dd r)$ for a sequence of values of $s > 0$ that
converges to zero. 
\end{rem}

 We also note that the equivalent of Synge's observation, 
but used for $s \in i\R$ instead of $s\in \R$, implies in contrast 
that, when the Rayleigh function is nonnegative,
the linearized operator $L_{0,k}$ {\it does} possess nonzero eigenvalues 
on the imaginary axis, which correspond to Kelvin modes.

\subsection{The two-dimensional case}\label{subsec22}

Although it is not included in Theorem~\ref{thm:main1},
the two-dimensional case $k = 0$ is worth mentioning too.  
When $m \neq 0$, the eigenvalue equation \eqref{eq:main} 
reduces to 
\begin{equation}\label{eq:main2D}
  -\partial_r (r^2 \partial_r^* u_r) + \Bigl(m^2 + \frac{imr W'(r)}{
  \gamma(r)}\Bigr)\,u_r \,=\, 0\,.
\end{equation}
A well-known sufficient condition for stability is that the vorticity
profile $W$ be a monotone function, see e.g. \cite{MP}, but 
unlike in the axisymmetric case no sharp criterion has been 
established so far. Again, for the reader's convenience, we reproduce 
here the easy argument showing spectral stability if $W'$ has
a constant sign.  

\begin{prop}\label{prop:2D} Assume that the vorticity profile $W$ 
is monotone. Then the eigenvalue equation 
\eqref{eq:main2D} has no nontrivial solution $u_r \in H^1(\R_+,r\dd r)$ 
if $\Re(s) \neq 0$. 
\end{prop}

\begin{proof}
  Assume that $u_r \in H^1(\R_+,r\dd r)$ is a nontrivial solution of
  \eqref{eq:main2D} for some $s \in \C$ with $\Re(s) \neq
  0$.
  Multiplying both members of \eqref{eq:main2D} by $r\bar u_r$ and
  integrating over $\R_+$, we obtain the relation
\begin{equation}\label{eq:juillet}
  \int_0^\infty \biggl\{|\partial_r(ru_r)|^2 + \Bigl(m^2  + 
  \frac{imrW'(r)}{\gamma(r)}\Bigr) |u_r|^2\biggr\}r \dd r \,=\, 0\,.
\end{equation}
In particular, taking the imaginary part and using  
\eqref{eq:gammadef}, we find
\[
  m\Re(s) \int_0^\infty \frac{W'(r)}{|\gamma(r)|^2} |u_r|^2 r^2 \dd r 
  \,=\, 0\,,
\]
and since $W$ is monotone we conclude that $u_r$ is supported in the
set where $W'$ vanishes.  This is clearly impossible if $W$ is not
identically constant, because $u_r$ is a nontrivial solution of the
second order ODE \eqref{eq:main2D}. But if $W$ is a constant, equation
\eqref{eq:juillet} immediately gives the desired contradiction.
\end{proof}

\section{The eigenvalue equation for $m \neq 0$ and 
$k \neq 0$}\label{sec3}

In this section we begin our study of the eigenvalue equation 
\eqref{eq:main} in the general case where $m \neq 0$ and $k \neq 0$. 
In view of the symmetries \eqref{eq:symetriespectre}, we can assume 
without loss of generality that $m \ge 1$ and $k > 0$. We write 
the spectral parameter as $s = m(a-ib)$, where $a,b \in \R$, and 
we decompose 
\begin{equation}\label{eq:gamma1def}
  \gamma(r) \,=\, s + im\Omega(r) \,=\, im\ggamma(r)\,, 
  \qquad \hbox{where}\quad \ggamma(r) \,=\, \Omega(r) - b -ia\,.
\end{equation}
According to Proposition~\ref{prop:AB}, the essential spectrum 
of the operator $L_{m,k}$ is the set of all $s = m(a-ib)$ 
such that $a = 0$ and $b \in [0,1]$. Outside that set, the 
function $\ggamma$ is bounded away from zero for all $r > 0$ and 
the eigenvalue equation \eqref{eq:main} becomes
\begin{equation}\label{eq:main2}
   -\partial_r\bigl(\cA(r)\partial_r^* u_r\bigr) + \cB(r) u_r \,=\, 0\,,
\end{equation}
where  $\partial_r^* = \partial_r + \frac1r$ and
\begin{equation}\label{eq:ABdef}
  \cA(r) \,=\, \frac{r^2}{m^2 + k^2 r^2}\,, \qquad 
  \cB(r) \,=\, 1 - \frac{k^2}{m^2}\,\frac{\cA(r)\Phi(r)}{\ggamma(r)^2}
  + \frac{r}{\ggamma(r)}\partial_r\biggl(\frac{W(r)}{m^2+k^2 r^2}
  \biggr)\,.
\end{equation}

\subsection{Asymptotic behavior at the origin and at infinity}
\label{subsec31}

Our first goal is to determine the asymptotic behavior of the
solutions of the complex ODE \eqref{eq:main2} as $r \to 0$ and
$r \to \infty$, assuming that $a \neq 0$ or $b \notin [0,1]$.  We
start with the behavior at the origin. If $u_r$ is a solution of
\eqref{eq:main2}, we set
\[
  u_r(r) \,=\, \frac{1}{r}\,v\Bigl(\log\frac{1}{r}\Bigr)\,,
  \qquad r > 0\,,
\]
or equivalently $v(x) = e^{-x} u_r(e^{-x})$ for $x=\log(1/r) \in \R$. 
The new function $v : \R \to \R$ satisfies the equation
\begin{equation}\label{eq:veq0}
  v''(x) + 2 k^2 \cA(e^{-x}) v'(x) - \cC(x) v(x) \,=\, 0\,, \qquad
  \hbox{where}\quad \cC(x) \,=\, e^{-2x}\,\frac{\cB(e^{-x})}{
  \cA(e^{-x})}\,.
\end{equation}
In view of \eqref{eq:ABdef} we have $\cA(e^{-x}) = \cO(e^{-2x})$ and
$\cC(x) = m^2 + \cO(e^{-2x}) + \cO(e^{-x}|W'(e^{-x})|)$ as
$x \to +\infty$. Thus applying e.g. \cite[Theorem~3.8.1]{CoLe}, we
deduce that equation \eqref{eq:veq0} as a unique solution $v$ such
that $e^{mx}v(x) \to 1$ as $x \to + \infty$. Returning to the original
variables, we conclude that equation \eqref{eq:main2} has a unique
solution $u_r$ such that $r^{1-m} u_r(r) \to 1$ as $r \to 0$. This
solution $u_r$ and its first derivative $u_r'$ depend continuously on
the various parameters in \eqref{eq:main2}, including the vorticity
profile $W \in \mathcal{C}^1_b(\Rpb)$ and the spectral parameter
$s = m(a-ib) \in \C$, uniformly in $r$ on any bounded interval of the
form $(0,R)$. Any linearly independent solution of \eqref{eq:main2}
blows up like $r^{-1-m}$ as $r \to 0$, and is therefore not square
integrable near the origin.

We next study the behavior at infinity. If $u_r$ is a solution 
of \eqref{eq:main2}, we define $w(r) = r^{1/2} u_r(r)$ and obtain
for $w$ the equation
\begin{equation}\label{eq:weq0}
  w''(r) + \frac{\cA'(r)}{\cA(r)}w'(r) - \cD(r) w(r) \,=\, 0\,,
  \quad \hbox{where}\quad \cD(r) \,=\, \frac{\cB(r)}{\cA(r)}
  + \frac{3}{4r^2} - \frac{1}{2r}\frac{\cA'(r)}{\cA(r)}\,.
\end{equation}
We have $\cA'(r)/\cA(r) = \cO(r^{-3})$ and
$\cD(r) = k^2 + \cO(r^{-2})$ as $r \to \infty$, because
Remark~\ref{rem:limits} implies that $W(r) = \cO(r^{-4})$,
$W'(r) = \cO(r^{-5})$, and $\Phi(r) = \cO(r^{-6})$ in that
limit. Invoking again \cite[Theorem~3.8.1]{CoLe}, we deduce that
\eqref{eq:weq0} has a unique solution $w$ such that
$e^{kr} w(r) \to 1$ as $r \to \infty$, hence \eqref{eq:main2} has a
unique solution satisfying $r^{1/2} e^{kr} u_r(r) \to 1$ as
$r \to \infty$.  This solution and its first derivative depend
continuously on the parameters in \eqref{eq:main2}, uniformly on the
interval $(R,\infty)$ for any $R > 0$. Any linearly independent
solution of \eqref{eq:main2} grows like $r^{-1/2}e^{kr}$ as
$r \to \infty$, and is therefore not square integrable.

Summarizing, we have shown: 

\begin{lem}\label{lem:simple}
If $m \neq 0$ and $k \neq 0$, any eigenvalue of the linear operator 
$L_{m,k} \in \cL(X_{m,k})$ outside the essential spectrum \eqref{eq:spAm}
is necessarily simple. Moreover, if $u_r$ is the radial velocity profile 
of the corresponding eigenfunction, there exist $\alpha, \beta \in \C$ 
such that
\[
  \lim_{r \to 0} r^{1-|m|} u_r(r) \,=\, \alpha\,, \qquad
  \hbox{and}\quad 
  \lim_{r \to \infty} r^{1/2} e^{|k|r} u_r(r) \,=\, \beta\,.
\]
\end{lem}

\subsection{Eigenvalues on the imaginary axis: Kelvin waves}
\label{subsec32}

In a second step, we consider the eigenvalues of the linearized
operator $L_{m,k}$ on the imaginary axis. The corresponding
eigenfunctions describe ``vibration modes'' of the columnar vortex,
and were first studied by Kelvin \cite{Ke} in the particular case of
Rankine's vortex.  Strictly speaking, this subsection is not part of
the proof of Theorem~\ref{thm:main1}, but in view of the physical
relevance of the Kelvin waves it is worth mentioning a few results
that can be rigorously established.

In what follows, we thus assume that $a = 0$ and $b \notin (0,1)$, so
that $\ggamma(r) \neq 0$ for all $r > 0$. In that case equation
\eqref{eq:main2} has real coefficients, and its solutions can be
studied using standard ODE techniques. For simplicity we suppose here
that the vorticity profile $W \in \WW$ is the restriction to $\Rp$ 
of a smooth even function on $\R$ satisfying $W''(0) < 0$, as it 
is the case for the Kaufmann-Scully vortex \eqref{eq:KSvortex} or 
the Lamb-Oseen vortex \eqref{eq:LOvortex}. We consider separately 
the regimes where $b \ge 1$ and $b \le 0$.

\begin{lem}\label{lem:b>1}
For any $m \neq 0$ and $k \neq 0$, the set of all $b > 1$ such that 
Eq.~\eqref{eq:main2} with $a = 0$ has a nontrivial solution in 
$H^1(\Rp,r \dd r)$ is a countable family which accumulates only
at $1$. Moreover, Eq.~\eqref{eq:main2} has no nontrivial 
solution in $H^1(\Rp,r \dd r)$ if $a = 0$ and $b = 1$. 
\end{lem}

\begin{proof}
When $b > 1$, we apply to Eq.~\eqref{eq:main2} the change of variables 
$u_r = r^m \cA(r)^{-1/2} v$, where $\cA(r)$ is as in \eqref{eq:ABdef}. 
A direct calculation shows that the new function $v$ satisfies
\begin{equation}\label{eq:veq1}
  -\partial_r^2 v - \frac{2m{+}1}{r}\,\partial_r v + 
  \Bigl(k^2 + \cF(r) + \cG(r)\Bigr)v \,=\, 0\,, \qquad r > 0\,,
\end{equation}
where 
\[
  \cF(r) \,=\, \frac{k^2 \cA(r)}{r^2}\Bigl(-2+3k^2 \cA(r)\Bigr)\,,  \qquad
  \cG(r) \,=\, - \frac{k^2}{m^2}\,\frac{\Phi(r)}{\ggamma(r)^2}
  + \frac{r}{\cA(r)\ggamma(r)}\partial_r\biggl(\frac{W(r)}{m^2+k^2 r^2}
  \biggr)\,.
\]
We assume that $b = 1 + h^2$ for some small $h > 0$, and we expand
\[
  -\ggamma(r) \,=\, 1 + h^2 - \Omega(r) \,=\, h^2 + \rho r^2 + 
  \cO(r^4)\,, \quad \hbox{as } r \to 0\,,
\]
where $\rho = -\Omega''(0)/2 = -W''(0)/8 > 0$. If $r = hs$, 
it is straightforward to verify that
\[
  h^4\Bigl(k^2 + \cF(hs) + \cG(hs)\Bigr) \,=\, -\frac{4k^2}{m^2}\,
  \frac{1}{(1+\rho s^2)^2} + \cO(h^2)\,, \quad \hbox{as } h \to 0\,,
\]
uniformly for all $s > 0$. Thus the new function $w$ defined 
by setting $w(s) = v(hs)$ satisfies the semi-classical 
Schr\"odinger equation
\begin{equation}\label{eq:weq1}
  \LL_h w \,:=\, 
  -h^2\Bigl(\partial_s^2 w + \frac{2m{+}1}{s}\,\partial_s w\Bigr) 
  -\frac{4k^2}{m^2}\,\frac{w}{(1+\rho s^2)^2} + U(s,h)w \,=\, 0\,,
\end{equation}
for all $s > 0$, where $U(s,h) = \cO(h^2)$ as $h \to 0$, uniformly in
$s$. Since the principal part of the potential term in \eqref{eq:weq1}
is negative, standard results in semiclassical analysis \cite{Si,HS}
show that the operator $\LL_h$ has negative eigenvalues if $h$ is
sufficiently small, and that the number of these bound states is
$\cO(h^{-1})$ as $h \to 0$. Moreover, as $\cF(r) + \cG(r) \to 0$ as
$r \to \infty$, the bottom of the essential spectrum of $\LL_h$ is
$k^2 h^4 > 0$ for any $h > 0$. These two observations together imply
that $\LL_h$ has a zero eigenvalue for a countable sequence
$h_n \to 0$, and returning to the original variables we conclude that
Eq.~\eqref{eq:main2} with $a = 0$ has a nontrivial solution in
$H^1(\Rp,r \dd r)$ for a sequence $b_n = 1 + h_n^2 \to 1$.

When $b = 1$, namely $h = 0$, the leading term in the function 
$\cB(r)/\cA(r)$ satisfies
\[
  \frac{k^2}{m^2}\,\frac{\Phi(r)}{\ggamma(r)^2} \,=\, 
  \frac{\Theta^2}{r^4}\,\bigl(1+\cO(r^2)\bigr) \quad
  \hbox{as } r \to 0\,, \qquad \hbox{where}\quad
  \Theta^2 \,=\, \frac{4k^2}{m^2\rho^2}\,.
\]
To investigate the behavior of the solutions of \eqref{eq:main2} near 
$r = 0$ in that case, it is useful make the change of variables 
$u_r(r) = r^{-1/2}U(1/r)$. Setting $x = 1/r$, this leads to an equation 
of the form
\begin{equation}\label{eq:Ueq1}
  U''(x) + \tilde \cC(1/x)U'(x) + \tilde \cD(1/x)U(x) \,=\, 0\,,
  \qquad x > 0\,,
\end{equation}
where $\tilde \cC(r) = \cO(r^3)$ and $\tilde \cD(r) = \Theta^2 
+ \cO(r^2)$ as $r \to 0$. Using \cite[Theorem~3.8.1]{CoLe}, 
we deduce that Eq.~\eqref{eq:Ueq1} has two linearly independent 
solutions satisfying $U_\pm(x) \,=\, e^{\pm i\Theta x}\bigl(1+\cO(1/x)\bigr)$ 
as $x \to +\infty$. If we now return to the original variables, we 
conclude that Eq.~\eqref{eq:main2} has two linearly independent solutions
$\phi_\pm$ such that
\begin{equation}\label{eq:devsingsing}
  \phi_\pm(r) \,=\, \frac{1}{\sqrt{r}}\,e^{\pm i\Theta/r}
  \Bigl(1 + \cO(r)\Bigr)\,, \qquad \hbox{as } r \to 0\,.
\end{equation}
As is easily verified, no nontrivial linear combination of $\phi_+$ 
and $\phi_-$ can belong to $H^1(\Rp,r\dd r)$, which means that 
Eq.~\eqref{eq:main2} has no nontrivial solution if $a = 0$ and
$b = 1$. 
\end{proof}

The situation is completely different when $b \le 0$. 

\begin{lem}\label{lem:b<0}
For any $m \neq 0$ and $k \neq 0$, the set of all $b \le 0$ such 
that Eq.~\eqref{eq:main2} with $a = 0$ has a nontrivial solution 
in $H^1(\Rp,r \dd r)$ is finite. Moreover\:\\[1mm]
1) This set is nonempty for a finite number of values of $m$ 
only;\\[1mm]
2) For both the Kaufmann-Scully vortex \eqref{eq:KSvortex} and 
the Lamb-Oseen vortex \eqref{eq:LOvortex}, Eq.~\eqref{eq:main2} has\\
\null\hspace{12pt} no nontrivial solution when $a = 0$ and $b \le 0$ 
if $|m| \ge 2$. 
\end{lem}

\begin{proof}
If $a = 0$ and $b \le 0$, then $\ggamma(r) = \Omega(r) - b 
= \Omega(r) + |b| > 0$. In this region, it is easy to verify that 
the coefficient $\cB(r) < 1$ defined in \eqref{eq:ABdef} is an 
increasing function of both parameters $|m|$ and $|b|$. Moreover, 
using the bounds on $\Omega$, $W$, and $\Phi$ which follow from 
assumptions H1, H2, see Remark~\ref{rem:limits}, we obtain the following 
estimate\:
\[
  \sup_{r > 0} \,\Bigl(1 - \cB(r)\Bigr) \,\le\, \frac{C}{m^2}
  \,\sup_{r > 0}\biggl(\frac{\Phi(r)}{(\Omega(r) + |b|)^2} + \frac{
  W(r) + r|W'(r)|}{\Omega(r) + |b|}\biggr) \,\le\, \frac{C}{m^2}\,
  \frac{1}{1+|b|}\,,
\]
where the constant $C$ depends only on the vorticity profile. As a 
consequence, we see that $\cB(r) \ge 0$ when $|m|$ or $|b|$ is large
enough, and this implies that Eq.~\eqref{eq:main2} has no nontrivial
solution, see \eqref{eq:HG0} below. It follows that the linearized 
operator $L_{m,k}$ can have eigenvalues $s = m(a-ib)$ with $a = 0$ and 
$b \le 0$ only for a finite number of values of $m \in \Z$, and 
using Sturm-Liouville theory we also conclude that, for any $m \in \Z$, 
there exist only finitely many eigenvalues with $a = 0$ and $b \le 0$. 
Interestingly enough, for both the Kaufmann-Scully vortex \eqref{eq:KSvortex} 
and the Lamb-Oseen vortex \eqref{eq:LOvortex}, an explicit calculation 
which is reproduced in Section~\ref{subsec67} shows that $\cB(r) 
\ge 1 - 4/m^2$, so that there are no eigenvalues in this region when 
$|m| \ge 2$.
\end{proof}

As a final comment, we mention that, when $m = \pm 1$, there are 
always eigenvalues with $a = 0$ and $b \le 0$. Indeed, due to 
translation invariance, the operator $L_{m,0}$ has a zero eigenvalue
with eigenfunction
\[
  u \,=\, -im\Omega\,e_r + (W{-}\Omega)\,e_\theta\,, \qquad
  \omega \,=\, W'\,e_z\,.
\]
That eigenvalue bifurcates out of the essential spectrum as the 
parameter $k$ varies, so that $L_{m,k}$ has at least one eigenvalue 
$s = -imb$ with $b < 0$ if $|m| = 1$ and $|k|$ is small enough.

\subsection{Eigenvalues outside the imaginary axis: 
Howard identities}\label{subsec33}

For our next step in the study of the eigenvalue equation
\eqref{eq:main2}, we use a classical method originally due to Rayleigh
\cite{Ra1} to show that the linearized operator $L_{m,k}$ has no
spectrum in large regions of the complex plane, which are depicted in
Fig.~1. The idea is to derive integral identities satisfied by the
hypothetical eigenfunctions, which eventually lead to a contradiction.

Assume thus that the eigenvalue equation \eqref{eq:main2} has a
nontrivial solution $u_r \in H^1(\R_+,r\dd r)$, for some
$s = m(a-ib) \in \C$, where $a \neq 0$. Multiplying both sides of
\eqref{eq:main2} by $r\bar u_r$ and integrating over $\R_+$, we easily
obtain, using the results of Section~\ref{subsec31}\:
\begin{equation}\label{eq:HG0} 
  \int_0^\infty \Bigl(\cA(r)|\partial_r^* u_r|^2 + \cB(r) |u_r|^2 
  \Bigr)r \dd r \,=\, 0\,. 
\end{equation}
Note that the function $\cB$ is complex-valued if $a \neq 0$, 
so that \eqref{eq:HG0} gives two integral relations for the 
radial velocity $u_r$. For instance, taking the imaginary part of 
\eqref{eq:HG0} and using the expression \eqref{eq:ABdef} of 
$\cB$, we obtain the identity
\begin{equation}\label{eq:HG0im}
  a\int_0^\infty \biggl\{\frac{2(b-\Omega(r))}{(a^2 + (\Omega{-}
  b)^2)^2}\,\frac{k^2}{m^2}\,\cA(r)\Phi(r) + \frac{r}{a^2 + 
  (\Omega{-}b)^2}\partial_r \Bigl(\frac{W(r)}{m^2+k^2r^2}
  \Bigr)\biggr\}|u_r|^2 r\dd r  \,=\, 0\,.
\end{equation}
This relation is identically satisfied if $a = 0$, but gives useful
information if $a \neq 0$. For instance, if $b \le 0$, then 
$b - \Omega(r) < 0$ for all $r > 0$, and assumption H1 
implies that 
\[
   \Phi(r) \,>\, 0 \quad \hbox{and} \quad \partial_r \Bigl(\frac{ 
  W(r)}{m^2+k^2r^2}\Bigr) \,<\, 0\,, \qquad \hbox{for all }
  r > 0\,.
\]
Thus the integrand in \eqref{eq:HG0im} is nonpositive and not
identically zero, hence equality \eqref{eq:HG0im} cannot hold. We
conclude that the operator $L_{m,k}$ has no eigenvalue $s = m(a-ib)$
with $a \neq 0$ and $b \le 0$, see Fig.~1. Unfortunately, we do not
know how to use the relation \eqref{eq:HG0} to preclude the existence
of eigenvalues of $L_{m,k}$ in other regions of the complex plane.

The following approach, due to Howard \cite{How, HG}, 
provides other identities similar to \eqref{eq:HG0}, which give
further information on the possible eigenvalues. Define
$u_r = q(r) v_r$, where $q$ is a (real or complex valued)
weight function satisfying $q(r) \neq 0$ for all $r > 0$. 
Then $v_r$ is a solution to
\begin{equation}\label{eq:main_vr}
  - \partial_r \Bigl(q(r)^2 \cA(r) \partial_r^* v_r\Bigr) + 
  \cE(r) v_r \,=\, 0\,,
\end{equation}
where
\[
  \cE(r) \,=\, q(r)^2 \cB(r) - q(r)q'(r)\Bigl(\cA'(r) - 
  \frac{\cA(r)}{r}\Bigr) -q(r)q''(r)\cA(r)\,.
\]
Multiplying both sides of \eqref{eq:main_vr} by $r\bar v_r$ and 
integrating over $\R_+$, we deduce
\begin{equation}\label{eq:weighted}
  \int_0^\infty \Bigl(q(r)^2 \cA(r) |\partial_r^* v_r|^2 +  
  \cE(r) |v_r|^2\Bigr) r\dd r \,=\, 0\,.
\end{equation}
If $q$ is real-valued, then $q^2 |v_r|^2 = |u_r|^2$ and taking 
the imaginary part of \eqref{eq:weighted} we recover \eqref{eq:HG0im}, 
but the real part gives new information. If $q$ is complex, 
both the real and the imaginary parts of \eqref{eq:weighted} 
provide new information. 

Following \cite{HG}, we now consider in more detail some interesting 
particular cases of \eqref{eq:weighted}.

\medskip\noindent{\bf Choice 1\:} $q(r) = \ggamma(r)$. We then have
\begin{align*}
  \cE(r) \,&=\, \ggamma(r)^2 - \frac{k^2}{m^2}\,\cA(r)\Phi(r) 
  + r\ggamma(r)\partial_r\Bigl(\frac{W(r)}{m^2+k^2r^2}\Bigr) \\
  & \quad -\ggamma(r)\ggamma'(r)\Bigl(\cA'(r) - \frac{\cA(r)}{r}\Bigr)
  -\ggamma(r)\ggamma''(r)\cA(r)\,.
\end{align*}
Since $r\ggamma'(r) = r\Omega'(r) = W(r) - 2\Omega(r)$, we 
observe that
\begin{equation}\label{eq:curious}
  \ggamma'(r)\Bigl(\cA'(r) - \frac{\cA(r)}{r}\Bigr) + 
  \ggamma''(r)\cA(r) \,=\,  r\partial_r\Bigl(\frac{\ggamma'(r)
  \cA(r)}{r}\Bigr) \,=\, r\partial_r\Bigl(\frac{W(r) -2\Omega(r)}{
  m^2+k^2 r^2}\Bigr)\,,
\end{equation}
and we deduce the following simpler expression of $\cE(r)$\:
\[
  \cE(r) \,=\, \ggamma(r)^2 - \frac{k^2}{m^2}\,\cA(r)\Phi(r)
  + 2r\ggamma(r)\partial_r\Bigl(\frac{\Omega(r)}{m^2+k^2r^2}\Bigr)\,.  
\]
In particular, taking the imaginary part of \eqref{eq:weighted}, we 
obtain the identity
\begin{equation}\label{eq:HG1}
  2a \int_0^\infty \biggl\{(b-\Omega(r))\Bigl(\cA(r)|\partial_r^* v_r|^2 +  
  |v_r|^2\Bigr)  - r\partial_r\Bigl(\frac{\Omega(r)}{m^2+k^2r^2}\Bigr)
  |v_r|^2 \biggr\}r\dd r \,=\, 0\,.
\end{equation}
If we now assume that $b \ge 1$, so that $b - \Omega(r) > 0$ 
for all $r > 0$, we see that all terms in the integrand of 
\eqref{eq:HG1} are nonnegative, which leads to a contradiction 
if $a \neq 0$. We conclude that the linear operator $L_{m,k}$ has 
no eigenvalue $s = m(a-ib)$ if $a \neq 0$ and $b \ge 1$, see Fig.~1. 

\figurewithtex 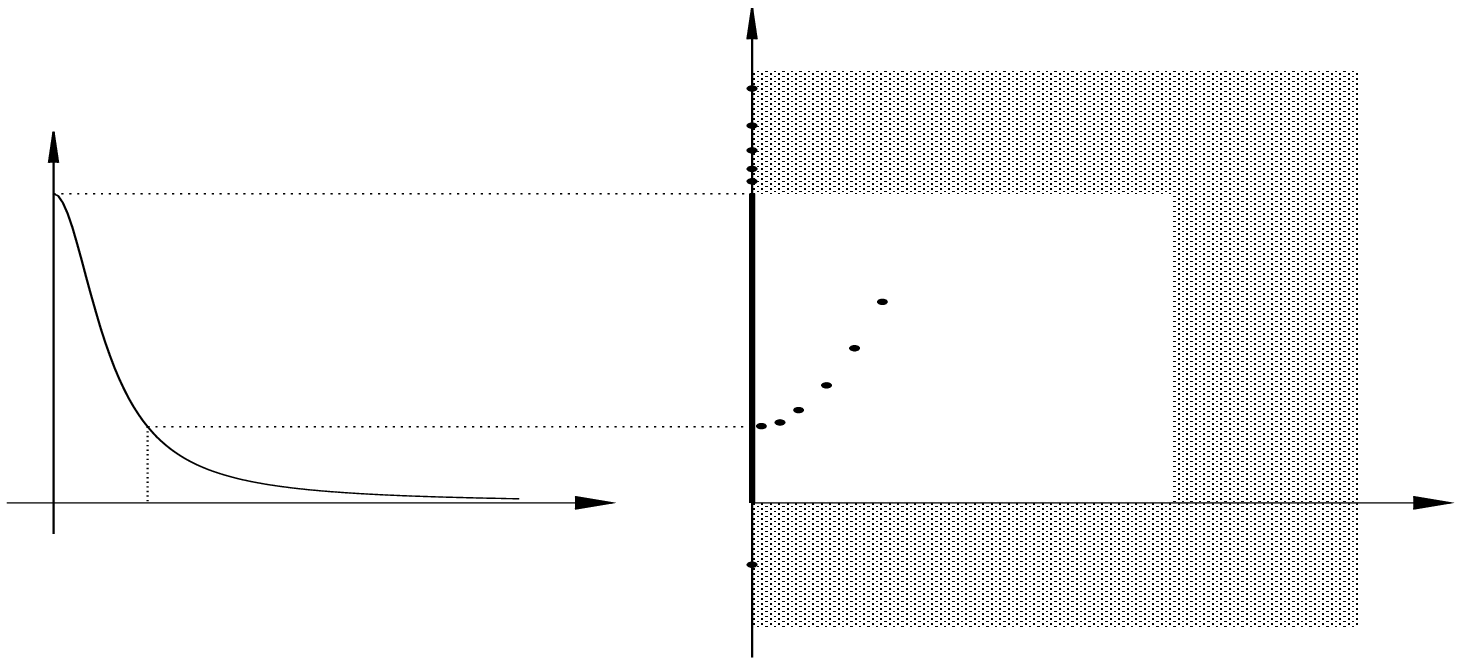 fig1.tex 6.50 15.00 {\bf Fig.~1:} (right)
The information obtained so far on the spectrum of the linearized
operator \eqref{eq:Lmkdef}, using the parametrization $s=m(a-ib)$.  
Kelvin modes are located on the imaginary
axis $a = 0$, and accumulate only at the upper edge of the essential
spectrum, which fills the segment $a = 0$, $b \in [0,1]$. The rest of
the spectrum, if any, consists of isolated eigenvalues which can
accumulate only on the essential spectrum, and are contained in 
 a region of the form $|a| \le M$, $b \in [0,1]$ 
according to Proposition~\ref{prop:HG}. (left) The angular velocity 
profile $\Omega$ and the critical radius $\rb$ associated with a 
spectral point $a = 0$, $b = \bb$ inside the essential spectrum. \cr

\medskip\noindent{\bf Choice 2\:} $q(r) = \ggamma(r)^{1/2}$. 
Proceeding as above, we find
\begin{align*}
  \cE(r) \,&=\, \ggamma(r) - \frac{k^2}{m^2}\, \frac{\cA(r)
  \Phi(r)}{\ggamma(r)} + r\partial_r\Bigl(\frac{W(r)}{m^2+k^2r^2}\Bigr) \\
  & \quad -\frac12 \ggamma'(r)\Bigl(\cA'(r) - \frac{\cA(r)}{r}\Bigr)
  -\frac12 \ggamma''(r)\cA(r) + \frac14 \frac{\ggamma'(r)^2}{\ggamma(r)}
  \,\cA(r)\,.
\end{align*}
Using again \eqref{eq:curious}, we deduce that
\[
  \cE(r) \,=\, \ggamma(r) - \frac{k^2}{m^2}\,\frac{\cA(r) 
  \Phi(r)}{\ggamma(r)} + \frac{r}{2}\partial_r\Bigl(\frac{W(r) + 2\Omega(r)}
  {m^2+k^2r^2}\Bigr) + \frac14 \frac{\Omega'(r)^2}{\ggamma(r)} \,\cA(r)\,.
\]
In particular, taking the imaginary part of \eqref{eq:weighted}, we 
obtain the identity
\begin{equation}\label{eq:HG1/2}
  -a \int_0^\infty \biggl\{\cA(r)|\partial_r^* v_r|^2 + |v_r|^2 +
  \frac{\cA(r)}{a^2 + (\Omega-b)^2}\Bigl(\frac{k^2 \Phi(r)}{m^2} 
  - \frac{\Omega'(r)^2}{4}\Bigr)|v_r|^2\biggr\}r\dd r \,=\, 0\,.
\end{equation}
As a consequence, if we assume that 
\begin{equation}\label{eq:HGcond}
  J(r) \,\equiv\, \frac{\Phi(r)}{\Omega'(r)^2} \,\ge\, \frac{m^2}{4k^2}\,,
  \qquad \hbox{for all } r > 0\,,
\end{equation}
we see that all terms in the integrand of \eqref{eq:HG1/2} are
nonnegative, which leads to a contradiction if $a \neq 0$. We conclude
that \eqref{eq:HGcond} is a {\em sufficient condition for spectral
  stability}.  Unfortunately, condition \eqref{eq:HGcond} is never met
for the Lamb-Oseen vortex, because $J(r) \to 0$ as $r\to +\infty$ in
that case. In the case of the Kaufmann-Scully vortex, it is satisfied
only if $m ^2 \le 4k^2$.

The results obtained by Howard's approach can thus be 
summarized as follows. 

\begin{prop}\label{prop:HG} Assume that the vorticity profile $W$
satisfies assumption H1 in Section~\ref{subsec12}. Then for any 
$m \neq 0$ and $k \neq 0$ the following holds: \\[1mm]
i) The linearized operator $L_{m,k}$ has no eigenvalue $s = m(a-ib)$ 
   with $a \neq 0$ and $b(1-b) \le 0$. \\[1mm]
ii) If condition \eqref{eq:HGcond} is satisfied, then $L_{m,k}$ has no 
   eigenvalue outside the imaginary axis. 
\end{prop}

\subsection{Critical layers and embedded eigenvalues}
\label{subsec34}

We assume here that $a = 0$ and $0 < b < 1$, which means that the 
spectral parameter $s=m(a-ib)$ is contained in the essential spectrum
of the linearized operator $L_{m,k}$, and does not coincide with 
one of its endpoints. The natural extension of the eigenvalue equation
\eqref{eq:main2} to this situation\footnote{We emphasize that the
derivation of \eqref{eq:main2} from the spectral problem was performed
in Section~\ref{sec2} under the assumption that $s = m(a-ib)$ does 
not belong to the essential spectrum.} is
\begin{equation}\label{eq:main2sing}
  -\partial_r\bigl(\cA(r)\partial_r^* u_r\bigr) +\biggl[1 -
  \frac{k^2}{m^2}\,\frac{\cA(r)\Phi(r)}{(\Omega(r)-b)^2} +
  \frac{r}{(\Omega(r)-b)}\partial_r\biggl(\frac{W(r)}{m^2+k^2 r^2}
  \biggr)\biggr]\
  u_r \,=\, 0\,.
\end{equation}
Since the value of $b$ belongs to the range of the angular velocity
$\Omega$, both denominators in \eqref{eq:main2sing} vanish at exactly
one point $\rb > 0$, characterized by $\Omega(\rb)=b$, so that
equation \eqref{eq:main2sing} becomes singular at that point. In the
physical literature, singularities of the eigenvalue equation are
usually avoided by allowing the variable $r$ to take slightly complex
values, a procedure that is referred to as ``critical layer
analysis'' in this context \cite{DR}. 

To perform such an analysis, we restrict our attention in the rest of
this section to vorticity profiles $W$ which satisfy assumption H1
and, in addition, are real-analytic on $(0,\infty)$, so that the
angular velocity $\Omega$ and the Rayleigh function $\Phi$ are
analytic too. According to the usual terminology, the point $\rb$ is 
then a {\it regular singular point} of equation \eqref{eq:main2sing}, see
e.g. \cite[Chapter~4]{CoLe} or Section \ref{subsec63} below. Extending
the range of the variable $r$ to a neighbourhood of $(0,\infty)$ in
$\C$ allows us to make a connection between solutions of
\eqref{eq:main2sing} defined on the interval $(0,\rb)$ and others
defined on $(\rb,\infty)$. In a neighborhood of $\rb$, the behavior 
of the solutions of \eqref{eq:main2sing} is determined by the roots 
 $d$ of the {\it indicial equation}
\begin{equation}\label{eq:determinant}
  d(d-1) + \frac{k^2}{m^2}\,J(\rb) \,=\, 0\,.
\end{equation}
We distinguish three cases.

\smallskip\noindent
{\bf Case 1: $0<J(\rb)< m^2/(4k^2)$}. The roots of \eqref{eq:determinant} 
are real and simple\: 
\[
  d_\pm \,:=\, \frac{1}{2} \pm \Big( \frac{1}{4} - \frac{k^2}{m^2}J(\rb)
  \Big)^\frac12\,. 
\]
In particular, we have $0 < d_- < 1/2 < d_+ < 1$. The Frobenius method 
\cite[Section~4.8]{CoLe} can be used to construct two real-valued 
analytic functions $V_\pm$ on $(0,\infty)$ such that $V_\pm(\rb)=1$ and 
such that the functions $\phi_\pm$ defined by
\begin{equation}\label{eq:devsing1}  
  \phi_\pm(r) \,=\, |b-\Omega(r)|^{d_\pm} \,e^{i\frac{\pi}{2}(1-{\sgn(b-\Omega(r)))
  d_\pm}} \,V_\pm(r)\,, \quad r > 0\,, 
\end{equation}
are independent solutions of \eqref{eq:main2sing} on both intervals
$(0,\rb)$ and $(\rb,\infty)$. Note that $\phi_\pm$ are real-valued on
$(\rb,\infty)$, but complex-valued (although with a constant phase) on 
the interval $(0,\rb)$.

\smallskip\noindent 
{\bf Case 2: $J(\rb)> m^2/(4k^2)$}. The roots of \eqref{eq:determinant} 
are complex conjugate\: 
\[
  d_\pm \,:=\, \frac{1}{2} \pm i\delta\,, \qquad \hbox{where}\quad 
  \delta \,=\, \Big(\frac{k^2}{m^2}J(\rb)-\frac14\Big)^{\frac12}\,.
\]
Similarly, the Frobenius method yields the existence of two
independent solutions $\phi_\pm$ which we write in the form
\begin{equation}\label{eq:devsing2}  
  \phi_\pm(r) \,=\, |b-\Omega(r)|^{1/2}\,e^{\pm i\delta \log |b-\Omega(r)|}
  \,e^{i\frac{\pi}{2}(1-\sgn(b-\Omega(r)))d_\pm}\, V_\pm(r)\,. 
\end{equation}

\smallskip\noindent
{\bf Case 3: $J(\rb) = m^2/(4k^2)$}. Equation \eqref{eq:determinant} 
possesses the unique root $1/2$ with multiplicity two, and two independent 
solutions of \eqref{eq:main2sing} can be constructed such that
\begin{equation}\label{eq:devsing3}\begin{split}  
  &\phi_+(r) = |b-\Omega(r)|^{\frac12} \,e^{i\frac{\pi}{4}(1-\sgn(b-\Omega(r)))}\ V_+(r), \\
  &\phi_-(r) = |b-\Omega(r)|^{\frac12} \,e^{i\frac{\pi}{4}(1-\sgn(b-\Omega(r)))}\bigl[ 
  \bigl({\scriptstyle \log|b-\Omega(r)| +  i\frac{\pi}{2}(1-\sgn(b-\Omega(r)))}
  \bigr)\ V_+(r) + V_-(r)\bigr]\,. 
\end{split}
\end{equation}

The following technical lemma emphasizes the relevance
of the singular functions $\phi_\pm$ for the approximation of
solutions of \eqref{eq:main2sing} by non-singular solutions of
\eqref{eq:main2}. In the statement, the vorticity profile $W$, the
spectral parameter $s=-imb$ and the corresponding singular radius
$\rb$ are defined as above. However, we consider a sequence
$(u_n)_{n \in \N}$ of smooth solutions of the eigenvalue equation
\eqref{eq:main2} where the spectral parameter $s$ is replaced by some
complex number $s_n$ with nonzero real part (so that $s_n$ does not
belong to the essential spectrum), and where also the vorticity
profile $W$ is replaced by some function $W_n$ that is allowed to
depend on $n$.\footnote{The reason for the latter will become clear in
  Section~\ref{sec4}.} We assume that $s_n \to s$ and $W_n \to W$ as
$n \to \infty$. In what follows, for $\rho > 0$ we denote by
$\bbD(\rb,\rho) \subset \C$ the open disc of radius $\rho$ centered at
$\rb+0i$. When no confusion is possible, we also use the same symbols
for functions of the real variable $r \in (0,\infty)$ and their
analytic extensions into (part of) the complex plane.

\begin{lem}\label{lem:approxsing}
Let $(u_n)_{n \in \N}$ be a sequence of solutions of \eqref{eq:main2} 
corresponding to a sequence of spectral parameters $s_n = m(a_n-ib_n)$ 
and of real-valued analytic profiles $W_n$. Suppose that \\[1mm]
\null\hskip 6pt i) $a_n > 0$ for all $n$, and $a_n \to 0$, $b_n \to b \in (0,1)$ 
as $n\to \infty$; \\[1mm]
\null\hskip 3pt ii) $W_n \to W$ in $\cC^1_b(\Rpb)$ as $n\to \infty$;\\[1mm]
iii) there exists $\rho>0$ such that, for all $n \in \N$, the radius of 
  analyticity of $W_n$ at $\rb$ is at least\\
  \null\hskip 18pt  equal to $\rho$, and $W_n \to W$ uniformly 
  in $\bbD(\rb,\rho)$. \\[1mm]
If $u_n(r)$ and $u_n'(r)$ have a limit as $n \to \infty$ for some $r \in 
(0,\infty) \setminus \{\rb\}$, then there exist $\alpha_\pm \in \C$ such that 
$u_n \to \alpha_+ \phi_+ + \alpha_- \phi_-$  in the 
$\mathcal{C}^1$ topology on compact subsets of $(0,\infty) \setminus 
\{\rb\}$, where $\phi_\pm$ are given by \eqref{eq:devsing1}, 
\eqref{eq:devsing2} or \eqref{eq:devsing3} depending on the roots of 
the indicial equation. 
\end{lem}

The proof of Lemma~\ref{lem:approxsing} is postponed to
Section~\ref{subsec63} below where we also establish the main
properties of $\phi_\pm$, in particular the analyticity of $V_\pm$
across the singularity $\rb$ and the fact these functions are
real-valued. For the moment, we observe that the implicit
determination of logarithms we opted for in constructing the solutions
$\phi_\pm$ is directly related to the assumption that $a_n > 0$ in
Lemma~\ref{lem:approxsing}.  An approximation procedure valid for
negative values of $a_n$ would involve the complex conjugates of the
functions $\phi_\pm$ defined in \eqref{eq:devsing1}--\eqref{eq:devsing3}.

\begin{rem}\label{rem:noembed}
The expressions \eqref{eq:devsing1}--\eqref{eq:devsing3} show in particular 
that no nontrivial solution of \eqref{eq:main2sing} lies in $H^2(\rb-\epsilon,
\rb+\epsilon)$ if $\epsilon > 0$. As we know that the radial velocity $u_r$ 
associated with any vorticity vector $\omega \in X_{m,k}$ belongs to 
$H^1(\Rp,r\dd r) \cap H^2_\loc(\Rp)$, we deduce from the observation 
above that the linear operator $L_ {m,k}$ acting on $X_{m,k}$ has 
no nonzero eigenvalue $s$ embedded in the continuous spectrum 
\eqref{eq:spAm}. 
\end{rem}

Finally, repeating the proof of Lemma~\ref{lem:simple} for 
Eq.~\eqref{eq:main2sing}, we easily obtain

\begin{lem}\label{lem:psi0psiinf}
If $m \neq 0$, $k \neq 0$, and $0 < b < 1$, there exist a unique 
solution $\psi_0$ of \eqref{eq:main2sing} on $(0,\rb)$ and a unique 
solution $\psi_\infty$ of \eqref{eq:main2sing} on $(\rb,\infty)$ such that 
\[
  \lim_{r \to 0} r^{1-|m|} \psi_0(r) \,=\, 1, \qquad
  \hbox{and}\quad 
  \lim_{r \to \infty} r^{1/2} e^{|k|r} \psi_\infty(r) \,=\, 1\,.
\]
Moreover, both $\psi_0$ and $\psi_\infty$ are real-valued.  
\end{lem}

Since \eqref{eq:main2sing} is a linear equation, we infer the existence 
of constants $\alpha^0_\pm, \alpha^\infty_\pm \in \C$ such that 
\begin{equation}\label{eq:decompsi}
  \psi_0 = \alpha^0_- \phi_- + \alpha^0_+ \phi_+ \quad\hbox{on } (0,\rb)\,,
  \qquad
  \psi_\infty = \alpha^\infty_- \phi_- + \alpha^\infty_+ \phi_+ \quad\hbox{on } 
  (\rb,\infty)\,,
\end{equation}
where $\phi_\pm$ are defined in \eqref{eq:devsing1}--\eqref{eq:devsing3}. 

\section{The homotopy argument}\label{sec4}

This section is the core of the proof of Theorem~\ref{thm:main1}.  We
concentrate on the situation where the angular Fourier mode $m$ and
the vertical wave number $k$ are both nonzero, because the cases
$m = 0$ and $k = 0$ have already been treated in Sections~\ref{subsec21} 
and \ref{subsec22}, respectively. In view of the symmetry properties
\eqref{eq:symetriespectre}, we can assume without loss of generality
that $m \ge 1$ and $k > 0$.

The argument is by contradiction: given a vorticity profile $W$
satisfying assumptions H1, H2 in Section~\ref{subsec12}, we assume
that there exist an integer $m \ge 1$ and a real number $k > 0$ such
that the linearized operator $L_{m,k}$ has at least one eigenvalue
outside the imaginary axis. The strategy is then to perform a homotopy
between the vorticity profile $W_0 := W$ and a reference profile $W_1$ for
which we know {\it a priori} that the corresponding linearized
operator has no eigenvalue with nonzero real part. Since
eigenvalues outside the imaginary axis depend continuously on the
vorticity profile, in an appropriate topology, this implies in our
situation that all eigenvalues necessarily merge into the essential
spectrum as the homotopy parameter varies from zero to one. We
eventually reach a contradiction by showing that such a merger is
impossible. This is achieved by a careful asymptotic analysis of the
solutions of the complex ODE \eqref{eq:main2} in the limit where the
real part of the eigenvalue $s = m(a-ib)$ vanishes. Our approach
combines the results of Section~\ref{subsec34} on critical layers, 
the integral identities obtain by Howard's method in 
Section~\ref{subsec33}, and new ingredients which rely on the 
monotonicity assumption H2. 

Since we have to consider various vorticity profiles in the course 
the proof, the linearized operator \eqref{eq:Lmkdef} will sometimes 
be denoted by $L_{m,k}^W$ instead of $L_{m,k}$, to avoid any ambiguity. 
The following continuity property plays an essential role in our 
argument. 

\begin{lem}\label{lem:contiW2L} 
The (linear) mapping $W \mapsto L_{m,k}^W$ is continuous from $\cC^1_b(\Rpb)$
into $\cL(X_{m,k})$.
\end{lem}

\begin{proof}
As can be seen from definitions \eqref{eq:Lmkdef}-\eqref{eq:Bdef}, 
the linearized operator $L_{m,k}^W$ has variable coefficients depending 
(linearly) on the functions $\Omega$, $r\Omega'$, $W$, and $W'$. 
Now, we have the estimate
\[
  \|\Omega\|_{L^\infty(\Rp)} + \|r\Omega'\|_{L^\infty(\Rp)} \,\le\, 
  C\|W\|_{L^\infty(\Rp)}\,, 
\]
which follows from the representation formula \eqref{eq:Omrep} and 
the identity $r\Omega' = W-2\Omega$. Thus all coefficients of 
$L_{m,k}^W$ are $L^\infty$ functions that depend continuously on $W$ in the topology of $\cC^1_b(\Rpb)$, 
and since the Biot-Savart map $\omega \mapsto u$ is bounded in 
$X_{m,k}$ by Proposition~\ref{prop:estBS} below, we obtain the 
desired result. 
\end{proof}

\subsection{Reduction to a real analytic vorticity profile}
\label{subsec41}

We now present the contradiction argument in detail. We fix $m \ge 1$, 
$k > 0$, and we assume that there exists a vorticity profile $W_0 \in 
\WW$ such that the associated linear operator $L_{m,k}^{W_0} \in \cL(X_{m,k})$
has at least one (isolated) eigenvalue outside the imaginary axis. 
Our goal is to prove that this is impossible, which is exactly the 
conclusion of Theorem~\ref{thm:main1}. 

In a first step, we show that one can assume without loss of
generality that the profile $W_0 \in \WW$ is {\em real analytic} on
$\Rpb$. By this we mean more precisely that $W_0$ is the restriction
to $\Rpb$ of a real analytic even function defined on the whole real
line. Indeed, we know from Proposition~\ref{prop:AB} that, for any
$W \in \WW$, the spectrum of $L_{m,k}^W$ outside the imaginary axis 
consists of isolated eigenvalues with finite multiplicity,
which are in fact simple as asserted by Lemma~\ref{lem:simple}.
Invoking Lemma~\ref{lem:contiW2L} and classical perturbation 
theory \cite[IV-\S 3.5]{Ka}, we observe that these (hypothetical) 
eigenvalues depend continuously on the vorticity profile $W$ in 
the topology of $\cC^1_b(\Rpb)$. In particular, if $W$ is close 
enough to $W_0$ in that topology, we are sure that the operator 
$L_{m,k}^{W}$ has at least one eigenvalue with nonzero real part.

We next invoke a density result that will be established in 
Section~\ref{subsec64} below. 

\begin{lem}\label{lem:dense}
The subset $\WWA$ of $\WW$ consisting of vorticity profiles which are 
also real analytic on $\Rpb$ is dense in $\WW$ for the topology of 
$\cC^1_b(\Rpb)$.
\end{lem}

The proof of Lemma~\ref{lem:dense} is not straightforward because the
definition of the class $\WW$ involves the quantity $J$, introduced in
\eqref{eq:Jdef}, which depends in a nonlinear way on the vorticity
profile $W$.  Thus, given $W \in \WW$, we cannot construct an
approximation $W_\epsilon \in \WWA$ just by taking the convolution of
$W$ with a real analytic mollifier. To avoid this difficulty, we prove
in Section~\ref{subsec64} that all quantities $\Omega$, $W$, $\Phi$
are entirely determined by the auxiliary function $J$, and we even
provide explicit reconstruction formulas. Then, at the level of $J$,
we use a nonlinear approximation scheme of the form
\[
	\frac{1}{\sqrt{1+J_\epsilon}} \,=\, G_\epsilon * \frac{1}{
  \sqrt{1+J}}\,, \quad \epsilon > 0\,,
\]
where $G_\epsilon$ denotes the heat kernel on the half-line $\Rp$ with
Dirichlet boundary condition at $r = 0$. This provides an 
approximation procedure within the class $\WW$ which allows us 
to prove Lemma~\ref{lem:dense}, see Section~\ref{subsec64} for 
details.  

Taking advantage of Lemma~\ref{lem:dense} we assume from now on that
the initial vorticity profile $W_0$ in our contradiction argument is
real analytic, namely $W_0 \in \WWA$.

\subsection{Construction of the homotopy}
\label{subsec42}

In the particular example of the Kaufmann-Scully vortex
\eqref{eq:KSvortex}, the function \eqref{eq:Jdef} reduces to   
$J(r) = 1+1/r^2 \ge 1$. By a simple rescaling we deduce that, for the
vorticity profile $W_1 \in \WWA$ defined by
\begin{equation}\label{eq:W1def}
  W_1(r) \,=\, \frac{2}{(1+4k^2r^2/m^2)^2} \,=\, \frac{2 m^4}{(m^2
  +4k^2r^2)^2}\,, \qquad r \ge 0\,, 
\end{equation}
the stability condition \eqref{eq:HGcond} is satisfied, so that the
linear operator $L_{m,k}^{W_1}$ has no eigenvalue outside the
imaginary axis as a consequence of Proposition~\ref{prop:HG}. To
interpolate in the class $\WW$ between the initial profile
$W_0 \in \WWA$ and the reference profile \eqref{eq:W1def}, we use the
following result, whose proof is also postponed to
Section~\ref{subsec64}.

\begin{lem}\label{lem:homotopie} 
If $W_0, W_1 \in \WW$, there exists a Lipschitz function $\cH \: \ [0,1] \to 
\cC^1_b(\Rpb)$ such that $\cH(0)=W_0$, $\cH(1)=W_1$, and $W_t:=\cH(t) \in \WW$ 
for any $t \in [0,1]$. Moreover, if $W_0, W_1 \in \WWA$, then $W_t \in 
\WWA$ for all $t \in [0,1]$. In that case, if $W_1''(0) < 0$, then 
$W_t''(0) < 0$ for all $t \in (0,1]$. 
\end{lem}

Since the class $\WW$ is not convex, the linear interpolation $H(t) = 
(1-t)W_0 + t W_1$ is not appropriate here. Instead, we use again
the fact that a vorticity profile $W \in \WW$ is entirely determined
by the auxiliary function \eqref{eq:Jdef}, and at the level of $J$ 
we define the homotopy $\cH$ by the following nonlinear interpolation 
procedure
\begin{equation}\label{eq:homotopieexplicite}
  \frac{1}{\sqrt{1+J_t}} \,=\, \frac{1-t}{\sqrt{1+J_0}} + 
  \frac{t}{\sqrt{1+J_1}}, \qquad \forall t \in [0,1]\,.
\end{equation}
If $J_0$ and $J_1$ are real analytic, so is $J_t$ for all
$t \in [0,1]$, and it follows that $W_t \in \WWA$ for all
$t \in [0,1]$. We refer to Section~\ref{subsec64} for details.

\subsection{The bifurcation point}
\label{subsec43}

For any $t \in [0,1]$, we denote by $W_t \in \WWA$ the vorticity profile 
obtained from Lemma~\ref{lem:homotopie}, where $W_0 \in \WWA$ is the 
initial vorticity defined in Section~\ref{subsec41} and $W_1$ is given by
\eqref{eq:W1def}. We also introduce the associated angular velocity
\[
  \Omega_t(r) \,=\, \frac{1}{r^2}\int_0^r W_t(s)s\dd s\,, 
  \qquad r > 0\,,
\]
and we define $\Phi_t = 2 \Omega_t W_t$ and $J_t = \Phi_t/(\Omega_t')^2$
as in \eqref{eq:Phidef}, \eqref{eq:Jdef}. We
consider the family of linear operators $L_{m,k}^{W_t}$, indexed by
the homotopy parameter $t \in [0,1]$, which is uniformly bounded in
$\cL(X_{m,k})$ by Lemma~\ref{lem:contiW2L}. For each $t \in [0,1]$, it
follows from Proposition~\ref{prop:AB} and Lemma~\ref{lem:simple} that
the spectrum of $L_{m,k}^{W_t}$ outside the imaginary axis consists of
simple isolated eigenvalues. If $s = m(a-ib)$ is such an eigenvalue,
we know from Proposition~\ref{prop:HG} that $0 < b < 1$, and by
uniform boundedness there exists a constant $M > 0$ (independent of
$t$) such that $0 < |a| \le M$.

As the homotopy parameter $t$ varies, the isolated eigenvalues of
$L_{m,k}^{W_t}$ move continuously in the complex plane, as described
e.g. in \cite[IV-\S 3.5]{Ka}, and we chose our reference profile $W_1$
so that the associated linearized operator has no eigenvalue with
nonzero real part. This implies that, when $t$ increases from zero to
one, all isolated eigenvalues of $L_{m,k}^{W_t}$ eventually merge into
the essential spectrum on the imaginary axis. In particular, we can
define the bifurcation point
\[
  t_* \,=\, \inf\Bigl\{t \in (0,1]\,\Big|\,\sigma(L_{m,k}^{W_\tau}) 
  \subset i\R \hbox{ for all } \tau \in [t,1]\Bigr\}\,. 
\]
Our assumption on $W_0$ and the continuity of the eigenvalues imply that
$t_* >0$ and $\sigma(L_{m,k}^{W_{t_*}}) \subset i\R$. Moreover, there 
exist an increasing sequence $t_n \to t_*$ and a sequence of isolated 
eigenvalues $s_n=m(a_n-ib_n)$ of $L_{m,k}^{W_{t_n}}$ such that $a_n \neq 0$,
$0 < b_n < 1$, and 
\begin{equation}\label{eq:limiteig}
  a_n + ib_n \,\to\, i\bb\,, \qquad \hbox{as} \quad n \to \infty\,,
\end{equation}
for some $\bb \in [0,1]$. In view of the second identity in 
\eqref{eq:symetriespectre}, we can assume without loss of generality 
that $a_n > 0$ for all $n \in \N$. Associated with $\bb$, we also 
introduce the critical radius
\begin{equation}\label{eq:rbar}
  \rb \,=\, \left\{\begin{array}{ll} 
   \Omega_{t_*}^{-1}(\bb) & \hbox{if } \,\bb > 0\,,\\[1mm]
   +\infty & \hbox{if } \,\bb = 0\,.\\
\end{array}\right.
\end{equation}
As $W_{t_*} \in \WWA$ by construction, we recall that $\Omega_{t_*} : 
\Rpb \to \Rp$ is real analytic, strictly decreasing on $\Rp$, and 
satisfies $\Omega_{t_*}(0) = 1$ and $\Omega_{t_*}(r) \to 0$ as $r \to 
\infty$, so that $\rb \in [0,\infty]$ is well defined, see Fig.~1. 

In the sequel, for notational simplicity, we write $W$ instead of
$W_{t_*}$ and $W_n$ instead of $W_{t_n}$. Note in particular that,
after this redefinition, the symbol $W$ no longer refers to the
vorticity profile that appears in the statement of
Theorem~\ref{thm:main1}\thinspace! Similarly, we denote
\[
  \Omega = \Omega_{t_*}\,, ~\Phi = \Phi_{t_*}\,, ~J=J_{t_*}\,, \quad \hbox{and}
  \quad \Omega_n =  \Omega_{t_n}\,, ~\Phi_n = \Phi_{t_n}\,, ~J_n=J_{t_n}\,. 
\]
Finally, we also set $L_{m,k} = L_{m,k}^W$ and $L_{m,k}^n = L_{m,k}^{W_n}$. 
We observe that $W_n \to W$ in $\cC^1_b(\Rpb)$ as $n \to \infty$, due to the 
continuity properties of the homotopy defined in Lemma~\ref{lem:homotopie}. 

As is recalled at the beginning of Section~\ref{sec3}, for each $n \in \N$ 
we may associate to the eigenvalue $s_n = m(a_n-ib_n)$ of $L_{m,k}^n$ a 
nontrivial solution $u_n \in H^1(\R_+,r\dd r) \cap H^2_\loc(\Rp)$ of the 
complex differential equation
\begin{equation}\label{eq:main2n}
  -\partial_r\bigl(\cA(r)\partial_r^* u_n\bigr) + \Bigl[ 1 -
  \frac{k^2}{m^2}\,\frac{\cA(r)\Phi_n(r)}{\gamma_n(r)^2}
  + \frac{r}{\gamma_n(r)}\partial_r\biggl(\frac{W_n(r)}{m^2+k^2 r^2}
  \biggr)\Bigr]\ u_n \,=\, 0\,,
\end{equation}
where $\gamma_n(r) = \Omega_n(r)-b_n-ia_n$. As $W_n \in \WWA$, it is
clear that $u_n$ is in fact real analytic for all $n \in \N$. According to
Lemma~\ref{lem:simple} there exist nonzero complex numbers $\alpha_n$,
$\beta_n$ such that
\begin{equation}\label{eq:asymptotesany}
  \alpha_n \,=\, \lim_{r\to 0_+} r^{-m+1} u_n(r)\,, \qquad\hbox{and}\qquad
  \beta_n \,=\,\lim_{r\to \infty} r^{1/2} \exp(kr)u_n(r)\,.
\end{equation}
In what follows, we often normalize $u_n$ so that $\beta_n = 1$ for all 
values of $n$. 

As $n \to \infty$, the ODE \eqref{eq:main2n} becomes singular at the
point $r = \rb$, because $\gamma_n(\rb) \to \Omega(\rb)-\bb = 0$ in
view of \eqref{eq:limiteig} and \eqref{eq:rbar}. As is explained in
Section~\ref{subsec32}, the nature of the critical layer near
$r = \rb$ depends upon whether the quantity $J(\rb)$ is larger or
smaller than $m^2/(4k^2)$. This motivates the following definition:
\begin{equation}\label{eq:rbstardef}
  r_* =  \left\{\begin{array}{ll}
  J^{-1}(\tfrac{m^2}{4k^2}) & \hbox{if } \inf J < \tfrac{m^2}{4k^2}\,,\\
  +\infty & \hbox{otherwise}\,.
  \end{array}\right.  
\end{equation}
Note that $J : (0,\infty) \to \Rp$ is strictly decreasing by assumption H2, 
so that $r_*$ is uniquely defined. Moreover,  
\begin{equation}\label{eq:Jmono}
   r_* \,>\, 0\,, \qquad J(r) \,>\, \tfrac{m^2}{4k^2} ~\hbox{ for } r < r_*,
   \quad \hbox{and} \quad J(r) \,<\, \tfrac{m^2}{4k^2} ~\hbox{ for } r > r_*\,.
\end{equation}
Also, since $t_*>0$ we deduce from \eqref{eq:homotopieexplicite} and 
from our choice \eqref{eq:W1def} of $W_1$ that
\begin{equation}\label{eq:Jinfinipositif}
  J(\infty) \,>\, 0\,.
\end{equation}

In the rest of the proof of Theorem~\ref{thm:main1}, to reach the
desired contradiction, we consider various cases according to whether
the critical radius $\rb$ is larger, smaller or equal to $0$, $r_*$,
or~$+\infty$. 

\subsection{The situation $0<\rb < r_*$ is excluded}
\label{subsec44}

In this case, a contradiction is obtained from identity \eqref{eq:HG1/2}, 
or rather from its equivalent for the solutions of \eqref{eq:main2n} 
where the vorticity profile $W_n$ and the spectral parameter $s_n = 
m(a_n-ib_n)$ depend on $n$. In terms of the weighted function 
\begin{equation}\label{eq:vndef}
   v_n(r) \,=\, \frac{u_n(r)}{(b_n + ia_n - \Omega_n(r))^{1/2}} \,=\, 
   \frac{-i u_n(r)}{\gamma_n(r)^{1/2}}\,, \qquad r > 0\,,
\end{equation}
the identity becomes, after dividing by $a_n \neq 0$\:
\begin{equation}\label{eq:HG1/2n}
  \int_0^\infty \biggl\{\cA(r)|\partial_r^* v_n|^2 + |v_n|^2 +
  \frac{\cA(r)\Omega_n'(r)^2}{a_n^2 + (\Omega_n(r)-b_n)^2}\Bigl(
  \frac{k^2}{m^2}J_n(r)-\frac14\Bigr)|v_n|^2\biggr\}r\dd r \,=\, 0\,.
\end{equation}

We choose here to normalize the solutions $u_n$ of \eqref{eq:main2n}
so that $\beta_n = 1$ in \eqref{eq:asymptotesany} for all $n \in \N$.
This implies, in view of the analysis in Sections~\ref{subsec31} and
\ref{subsec34} and of the definitions in Section~\ref{subsec43},
that $u_n(r) \to \psi_\infty(r)$ and $u_n'(r) \to \psi'_\infty(r)$
locally uniformly on $(\rb,\infty]$ as $n \to \infty$, where
$\psi_\infty$ is the solution of the limiting equation
\eqref{eq:main2sing} introduced in Lemma~\ref{lem:psi0psiinf}.
Moreover, for any $\eps > 0$, the sequence $(u_n)$ is uniformly
bounded in $H^1([\rb+\eps,\infty),r\dd r)$, and so is the sequence
$(v_n)$ since $|\gamma_n(r)|$ is bounded away from zero when
$r \ge \rb + \eps$.  This uniform $H^1$ bound means that the
restriction of the integral in \eqref{eq:HG1/2n} to the interval
$[\rb+\eps,\infty)$ is uniformly bounded for all $n \in \N$. As the
integral over $(0,+\infty)$ is equal to zero, we deduce that the
integral over $(0,\rb+\eps)$ is also uniformly bounded, namely
\begin{equation}\label{eq:demiHG1/2n}
  \sup_{n\in\N}\int_0^{\rb+\eps}\biggl\{\cA(r)|\partial_r^* v_n|^2 + |v_n|^2 +
  \frac{\cA(r)\Omega_n'(r)^2}{a_n^2 + (\Omega_n(r)-b_n)^2}\Bigl(\frac{
  k^2}{m^2}J_n(r)-\frac14\Bigr)|v_n|^2\biggr\}r\dd r \,< \infty\,.
\end{equation}

Now comes into play the assumption that $\rb < r_*$. If we choose
$\eps > 0$ small enough so that $\rb+\eps < r_*$, we observe that, 
due to the definition of $r_*$ in \eqref{eq:rbstardef}, the integrand 
in \eqref{eq:demiHG1/2n} is nonnegative when $n$ is sufficiently large. 
Moreover, for all $r > \rb$, we know from \eqref{eq:vndef} that
$v_n(r) \to \psi_\infty(r)(\bb - \Omega(r))^{-1/2}$ as $n \to \infty$. 
So restricting the integral to the interval $(\rb,\rb+\eps)$ and 
invoking Fatou's lemma, we deduce from \eqref{eq:demiHG1/2n} that
\begin{equation}\label{eq:impossible0}
  \int_{\rb}^{r_*} \frac{\cA(r)\Omega'(r)^2}{(\bb - \Omega(r))^3}
  \Bigl(\frac{k^2}{m^2}J(r) - \frac{1}{4}\Bigr)|\psi_\infty(r)|^2
  r\dd r \,<\, \infty\,.
\end{equation}

The inequality $\rb < r_*$ also means that the roots of the indicial
equation \eqref{eq:determinant} are complex conjugate so that,
according to what we called {Case 2} in Section~\ref{subsec34}, we
have the decomposition $\psi_\infty(r) = \alpha^\infty_- \phi_-(r) +
\alpha^\infty_+\phi_+(r)$ for $r > \rb$, where $\phi_\pm$ are given by 
\eqref{eq:devsing2}. From these expressions, it is easy to deduce that
\eqref{eq:impossible0} cannot hold if $\psi_\infty$ is replaced by
either $\phi_+$ or $\phi_-$, because the integrand is positive and
behaves like $(r-\rb)^{-2}$ in a neighborhood of $\rb$. 
In the general case where both coefficients $\alpha^\infty_\pm$ are
nonzero, there may be cancellations between the contributions of 
$\phi_+$ and $\phi_-$, but due to the logarithmic phases in the  
expressions \eqref{eq:devsing2} of $\phi_\pm$ the function $\psi_\infty$ 
cannot vanish on large sets. More precisely, given $\alpha^\infty_\pm
\in \C$, there exist $\epsilon > 0$ and $0 < \rho_0 < r_*-\rb$ such 
that, for any $\rho \in (0,\rho_0)$,
$$
  \frac{1}{\rho}{\rm meas}\left(\left\{r \in (\rb,\rb+\rho) \hbox{ s.t. } 
  |\psi_\infty(r)|^2 \ge \epsilon |r-\rb|\right\}\right) \,\ge\, \frac12\,, 
$$
and the same argument as above shows that \eqref{eq:impossible0} is 
impossible.

\subsection{The situation $\rb=0$ is excluded}
\label{subsec45}

This case is treated by the same argument as in the previous section 
up to and including inequality \eqref{eq:impossible0}. The only
difference at that level is the asymptotic behavior of the functions
$\phi_\pm(r)$ as $r \to \rb$, because $\rb = 0$ is now an irregular 
singular point of the ODE \eqref{eq:main2sing}. According to 
\eqref{eq:devsingsing}, we have the expansion
\begin{equation}\label{eq:frob00}
  \phi_\pm(r) \,=\, \frac{1}{\sqrt{r}}\,\exp\biggl( 
  \frac{\pm 16ik}{m W''(0)r}\biggr)\bigl(1+\cO(r)\bigr)\,, 
  \qquad\hbox{as } \,r \to 0_+\,.
\end{equation}
The contradiction then follows exactly as in Section~\ref{subsec44},
the integrand of \eqref{eq:impossible0} being even more singular here
since it behaves like $r^{-4}$ in a neighborhood of $0$. 

\subsection{The situation $r_*<\rb <\infty$ is excluded}
\label{subsec46}

In that case, we cannot get a contradiction from identity
\eqref{eq:HG1/2n}, because the various terms in the integrand now 
have different signs in a neighborhood of the singular point $\rb$.
Instead, our argument relies on a detailed analysis of the 
solutions of \eqref{eq:main2n} near $\rb$, and on monotonicity 
properties that follow from assumption H2. 

As in the previous section, we normalize the solutions $u_n$ of
\eqref{eq:main2n} so that $\beta_n = 1$ in \eqref{eq:asymptotesany}
for all $n \in \N$. In particular, for any $r > \rb$, we know that 
$u_n(r) \to \psi_\infty(r)$ and $u_n'(r) \to \psi_\infty'(r)$ as $n \to 
\infty$, where $\psi_\infty$ is as in Lemma~\ref{lem:psi0psiinf}. Applying 
Lemma~\ref{lem:approxsing}, whose assumptions are satisfied 
by construction of the homotopy argument, we deduce that
\begin{equation}\label{eq:convunp}
  u_n(r) \,\to \alpha^\infty_- \phi_-(r) + \alpha^\infty_+\phi_+(r)\,, 
  \quad \hbox{and} \quad
  u_n'(r) \,\to \alpha^\infty_- \phi_-'(r) + \alpha^\infty_+\phi_+'(r)\,, 
\end{equation}
for all $r \in (0,\infty) \setminus \rb$, where $\alpha^\infty_\pm 
\in \C$ and $\phi_\pm$ are the solutions of \eqref{eq:main2sing} given by 
\eqref{eq:devsing1}. Note that the roots $d_\pm$ of the indicial equation 
\eqref{eq:determinant} are now real and distinct, so that we are in 
the situation referred to as {Case 1} in Section~\ref{subsec34}.  
The convergence \eqref{eq:convunp} for some $r < \rb$ implies, in 
view of the results in Section~\ref{subsec31} concerning the solutions 
of \eqref{eq:main2} near the origin, that the normalizing constants 
$\alpha_n$ in \eqref{eq:asymptotesany} converge to some limit 
$\alpha_* \in \C$ as $n \to \infty$. We deduce that
\begin{equation}\label{eq:phipmrep}
  \alpha^\infty_-\phi_-(r) + \alpha^\infty_+ \phi_+(r) \,=\, 
  \begin{cases}
  \alpha_* \psi_0(r) & \hbox{if } \,r \in (0,\rb)\,,\\
  \psi_\infty(r) &\hbox{if } \,r \in (\rb,\infty)\,.
   \end{cases}
\end{equation}

Now, the functions $\psi_\infty$, $\phi_-$, $\phi_+$ are all
real-valued on $(\rb,\infty)$, and we know from \eqref{eq:devsing1}
that $\phi_\pm(r) \approx (r-\rb)^{d_\pm}$ as $r \to \rb_+$, where
$0 < d_- < 1/2 < d_+< 1$. These observations imply that both
coefficients $\alpha^\infty_-$ and $\alpha^\infty_+$ are necessarily
real.  On the other hand, we deduce from \eqref{eq:phipmrep} that the
complex function $\alpha^\infty_-\phi_- + \alpha^\infty_+ \phi_+$ must
have a constant phase (modulo $\pi$) on the interval $(0,\rb)$, as it
is equal to the product of the real function $\psi_0$ by the complex
constant $\alpha_*$.  This, however, is impossible if both
coefficients $\alpha^\infty_\pm$ are nonzero, because by
\eqref{eq:devsing1} the complex functions $\phi_\pm$ have different
phases when $r < \rb$ and vanish at different rates as $r \to
\rb_-$. More precisely, since $\phi_+(r)/\phi_-(r) \to 0$ as $r \to \rb_-$, 
it follows from \eqref{eq:phipmrep}, \eqref{eq:devsing1} that
\[
  \alpha^\infty_- \,=\, \alpha_* \,e^{-i\pi d_-} \lambda\,, \qquad
  \hbox{where} \quad \lambda \,=\, \lim_{r \to \rb_-} \frac{\psi_0(r)}{
  |\bb-\Omega(r)|^{d_-}} \,\in\, \R\,.
\]
If $\alpha^\infty_- \neq 0$, then $\lambda \neq 0$ and $\alpha_* =
\alpha^\infty_- \lambda^{-1} \,e^{i\pi d_-}$, so using \eqref{eq:phipmrep}
and \eqref{eq:devsing1} again we obtain
\[
  0 \,=\, \Im\bigl(\alpha^\infty_+ \,e^{-i\pi d_-} \phi_+(r)\bigr) \,=\,
  \alpha^\infty_+ \sin(\pi (d_+ {-} d_-)) V_+(r)\,, \qquad 0 < r < \rb\,.
\]
As $V_+(r) \neq 0$ for $r$ sufficiently close to $\rb$, we conclude
that $\alpha^\infty_+ \sin(\pi (d_+ {-} d_-)) = 0$, and this implies that
$\alpha^\infty_+ = 0$ since $0 < d_+ - d_- < 1$. Therefore, we must
have $\alpha^\infty_-  \alpha^\infty_+ = 0$.

In the rest of this section, using totally different arguments 
which rely on assumption H2, we show that necessarily 
$\alpha^\infty_- \alpha^\infty_+ \neq 0$ in \eqref{eq:phipmrep}, 
and this will give the desired contradiction. To that purpose, 
we introduce the auxiliary functions
\begin{equation}\label{eq:Upm}
  U_\pm(r) \,=\, (\bb-\Omega(r))^{d_\pm}\,, \qquad r > \rb\,,
\end{equation}
and we denote by $\LL = -\partial_r \cA(r)\partial_r^* + \cB(r)$ the linear
operator in \eqref{eq:main2sing}. We claim that: 

\begin{lem}[Upper solutions]\label{lem:uppersol}
There exists $\gamma>0$ such that $\LL(U_\pm) \ge \gamma U_\pm'>0$ on 
$(\rb,\infty)$. 
\end{lem}

\begin{proof}
For notational simplicity we write $U$ instead of $U_\pm$, $d$ instead
of $d_\pm$, and $b$ instead of $\bb$. Computing $\LL(U)$ when $U 
= (b-\Omega)^d$, we obtain after elementary rearrangements of terms
\[
  \LL(U) \,=\, (b-\Omega)^{d-2} \bigl(T_1+T_2+T_3\bigr)\,,
\]
where
\[
  T_1 \,=\, -\cA{\Omega'}^2\Bigl(d(d-1)+\frac{k^2}{m^2}J\Bigr), \qquad
  T_2 \,=\, \biggl(1 - \Bigl(\frac{\cA}{r}\Bigr)'\biggr)(b-\Omega)^2\,,
\]
and
\[
  T_3 \,=\, \cA (b-\Omega)\biggl(d\Omega'' + d\Bigl(\frac{\cA'}{\cA}
  +\frac{1}{r}\Bigr)\Omega' - \frac{r}{\cA}\partial_r\Big(\frac{W}{m^2
  +k^2r^2}\Big)\biggr)\,.
\]
Since $d$ is a solution of the indicial equation \eqref{eq:determinant}, 
we may rewrite
\begin{equation}\label{eq:T1exp}
  T_1(r)  \,=\, \cA(r){\Omega'}^2(r)\frac{k^2}{m^2}\bigl(J(\rb)-J(r)
  \bigr)\,, 
\end{equation}
and our assumption H2 on $J$ implies that $T_1 > 0$ on $(\rb,\infty)$. 
Next, using the definition of $\cA$ and the fact that $m^2 \ge 1$, we 
observe that
\[
  1 - \Bigl(\frac{\cA}{r}\Bigr)' \,=\, 1 - \frac{m^2-k^2r^2}{(m^2+k^2r^2)^2} 
  \,\ge\, 0\,,
\]
so that $T_2 \ge 0$ on $(0,\infty)$. As for $T_3$, we expand
\[
  \frac{\cA'}{\cA}+\frac{1}{r} = \frac{3}{r} - \frac{2k^2r}{m^2+k^2r^2}\,, \qquad 
  \frac{r}{\cA}\,\partial_r\Big(\frac{W}{m^2+k^2r^2}\Big) = \frac{W'}{r} -
  \frac{2k^2 W}{m^2+k^2r^2}\,, 
\]
and we use the identities $W = r\Omega'+2\Omega$ and $W' = r\Omega''+
3\Omega'$ to derive the alternative expression
\[
  T_3 \,=\, \cA(b-\Omega)\Bigl( (d-1) \frac{W'}{r}  + \frac{2k^2}{m^2+k^2r^2}
  (W-dr\Omega')\Bigr)\,.
\]
Since $0<d<1$, and since $W > 0$, $W' < 0$ and $\Omega' < 0$ by 
assumption H1, we deduce that $T_3$ as well is positive on $(\rb,\infty)$. 
Altogether, we have shown that $\LL(U) > 0$ on $(\rb,\infty)$. 

To conclude the proof, we fix $r_0 > \rb$ large enough so that
$|(\cA/r)'| \le 1/2$ for $r \ge r_0$. In that region, we have
$\LL(U) \ge (b-\Omega)^{d-2}T_2 \ge (b-\Omega)^d/2 \ge \gamma_0 U'$ if
$\gamma_0 > 0$ is small enough. On the other hand, using 
\eqref{eq:T1exp} and the fact that $J'(\rb)<0$, $\Omega(\rb)=b$, and 
$\Omega'(\rb)<0$, we can find $\gamma_1 > 0$ small enough so that 
$$
  T_1(r) \,=\, \frac{J(\rb)-J(r)}{b-\Omega(r)}\,\cA(r){\Omega'}^2(r)
  \frac{k^2}{m^2}(b-\Omega(r)) \,\ge\, \gamma_1 d (b-\Omega(r))|\Omega'(r)|\,,
$$ 
for all $r$ in the compact interval $[\rb,r_0]$. This implies that $\LL(U) 
\ge \gamma_1 U'$ on $[\rb,r_0]$, and taking $\gamma = \min(\gamma_0,\gamma_1)$ 
we obtained the desired conclusion. 
\end{proof}

\begin{cor}\label{cor:uppersol}
The solutions $\phi_\pm(r)$ of \eqref{eq:main2sing} given by 
\eqref{eq:devsing1} are unbounded as $r \to \infty$. 
\end{cor}

\begin{proof}
Assume that $\varphi$ is a solution of $\LL(\varphi)=0$ on $(\rb,\infty)$, 
which is decomposed as $\varphi = UV$ where $U$ is one of the functions 
$U_\pm$ defined in \eqref{eq:Upm}. The equation satisfied by $V$ is
\begin{equation}\label{eq:upper2}
  0 \,=\, \LL(UV) \,=\, -\cA U V'' -(2\cA U'+(\cA'+\cA/r)U)V' + 
  \LL(U)V\,, \qquad r \in (\rb,\infty)\,.
\end{equation}
We interpret the right-hand side of \eqref{eq:upper2} as the action
on the function $V$ of a second order differential operator $\LL_U$
whose coefficients depend on $U$. Since $\cA U$ is positive on
$(\rb,\infty)$ by construction, and $\LL(U)$ is positive on
$(\rb,\infty)$ by Lemma \ref{lem:uppersol}, we observe that the
maximum principle holds for the operator $\LL_U$. As a consequence, 
the function $V$ which satisfies $\LL_U(V) = 0$ cannot have a positive
maximum nor a negative minimum on the interval $(\rb,\infty)$.

We first choose $\varphi=\varphi_+$, $U=U_-$, and we claim that
$\varphi_+$ is unbounded on $(\rb,\infty)$. Indeed, in the opposite
case, the function $V(r)=\phi_+(r)/U_-(r)$ would tend to zero both as
$r\to \rb_+$ and as $r\to \infty$, so that $V \equiv 0$ by the maximum
principle, which is clearly absurd. As a second application, we take
$\varphi=\varphi_-$, $U=U_-$, and we claim again that $\varphi_-$ is
unbounded on $(\rb,\infty)$. If not, by the maximum principle the
function $V(r)=\phi_-(r)/U_-(r)$ would be nonincreasing on
$(\rb,\infty)$ with $V(r)\to 1$ as $r\to \rb_+$ and $V(r) \to 0$ as
$r\to \infty$. Note that $V$ coincides with the function $V_-$ in
\eqref{eq:devsing1} and is therefore analytic up to the singular point
$\rb$.  Thus, using equation \eqref{eq:upper2} and
Lemma~\ref{lem:uppersol}, we can compute
\[
  V'(\rb) =\lim_{r\to \rb_+} V'(r) \,=\, \lim_{r\to \rb_+} 
  \frac{\LL(U)V-\cA UV''}{2\cA U'+(\cA'+\cA/r)U} \,=\, 
  \lim_{r\to \rb_+} \frac{\LL(U)}{2\cA U'} 
  \,\ge\, \frac{\gamma}{2\cA(\rb)} \,>\, 0\,,  
\]
and this contradicts the claim that $V$ is nonincreasing. The proof is 
complete. 
\end{proof}

As $\psi_\infty$ is bounded on the interval $(\rb,\infty)$ whereas
both $\phi_+$, $\phi_-$ are unbounded by Corollary~\ref{cor:uppersol}, 
the relation $\psi_\infty = \alpha^\infty_-\phi_- + \alpha^\infty_+ \phi_+$ can
hold only if both coefficients $\alpha^\infty_-$, $\alpha^\infty_+$ are nonzero. 
This gives the desired contradiction in the case where $r_*<\rb <\infty$.

\subsection{The situation $\rb=r_*<\infty$ is excluded}
\label{subsec47}

We proceed here as in Section \ref{subsec46}, the only essential
difference being that the exponents $d_- = d_+ = 1/2$ are no longer
distinct. We are thus in the situation referred to as {Case 3} in
Section~\ref{subsec34}, where the solutions $\phi_\pm$ are given by
the expressions \eqref{eq:devsing3}.  Applying
Corollary~\ref{cor:uppersol}, we obtain as above that $\phi_+$ is
unbounded on the interval $(\rb,\infty)$, but the argument does not
apply to the second solution $\phi_-$ which contains a logarithmic
correction, see \eqref{eq:devsing3}. Nevertheless, we deduce that the
coefficient $\alpha^\infty_-$ in the representation \eqref{eq:phipmrep} 
is necessarily nonzero, and this turns out to be enough to obtain the 
desired contradiction. Indeed, due to the logarithmic term, 
 it is easy to verify that, if $\theta(r) = 
\arg(\alpha^\infty_- \phi_-(r) + \alpha^\infty_+ \phi_+(r))$, then
$\tan(\theta(r)) \sim -\pi^{-1}\log(\rb - r)$ as $r \to \rb$, 
which shows that the left-hand side of \eqref{eq:phipmrep} cannot have
a constant phase for $r \in (0,\rb)$.

\subsection{The situation $r_*<\rb=\infty$ is excluded}
\label{subsec48}

We next consider the case where $\bb = 0$ in \eqref{eq:limiteig}, so
that $\rb = + \infty$ according to \eqref{eq:rbar}. In that situation,
the ``critical layer'' occurs at very large values of $r$,
 in a region where the eigenvalue equation
\eqref{eq:main2n} is already in some asymptotic regime. Here we cannot
use the same arguments as in section~\ref{subsec46} to obtain a
contradiction, because the location of the critical layer changes as
$n$ is increased. However, it is possible to obtain an
accurate representation of the solution of \eqref{eq:main2n} that
decays to zero as $r \to \infty$, by comparing it with the explicit
solution of a model problem, \eqref{eq:weqn2} below, which can be
expressed in terms of modified Bessel functions. This approximation
turns out to be sufficient to derive a contradiction when
combined with the identity \eqref{eq:HG0im}.

Our starting point is the equation \eqref{eq:weq0} for $w_n(r) = r^{1/2} 
u_n(r)$, which reads
\begin{equation}\label{eq:weqn}
  w_n''(r) + \frac{\cA'(r)}{\cA(r)}\,w_n'(r) - \cD_n(r) w_n(r) \,=\, 0\,,
  \qquad r > 0\,,
\end{equation}
where, in view of \eqref{eq:ABdef}, 
\begin{equation}\label{eq:Dndef}
  \cD_n(r) \,=\, k^2 + \frac{m^2+\frac34}{r^2} - \frac{1}{2r}\frac{\cA'(r)}{\cA(r)}
  - \frac{k^2}{m^2} \frac{\Phi_n(r)}{\gamma_n(r)^2} + \frac{r}{\cA(r)\gamma_n(r)}
  \partial_r\biggl(\frac{W_n(r)}{m^2+k^2 r^2}\biggr)\,.
\end{equation}
We recall that $\gamma_n(r) = \Omega_n(r)-b_n-ia_n$, and we observe
that the function $r \mapsto |\gamma_n(r)|$ reaches its minimum at
$r = r_n$, where $r_n = \Omega_n^{-1}(b_n)$. As $b_n \to 0$, it is
clear that $r_n \to \infty$ as $n \to \infty$, and since
$r^2\Omega_n(r)$ converges uniformly on $\Rp$ to $r^2 \Omega(r)$ (by
the results of Section~\ref{subsec64}), we even have
$\lim_{n \to \infty} r_n^2 b_n = \lim_{r \to \infty}r^2 \Omega(r) =
\Gamma > 0$, hence
\begin{equation}\label{eq:Gamman}
  r_n^2 \,=\, \frac{\Gamma_n}{b_n} \quad \forall n \in \N\,, \qquad 
  \hbox{where}\quad \Gamma_n \,\xrightarrow[n \to \infty]{}\, \Gamma\,. 
\end{equation}
Similarly we have $\Omega_n'(r_n) = -d_n b_n^{3/2}$ for all $n \in \N$, 
where $d_n \to 2\Gamma^{-1/2}$ as $n \to \infty$.
 
Eq.~\eqref{eq:weqn} has asymptotically constant coefficients, in the
sense that $\cA'(r)/\cA(r) = \cO(r^{-3})$ and $\cD_n(r) \to k^2$ as
$r \to \infty$. However, in general, the convergence of $\cD_n(r)$
toward its limit $k^2$ is not uniform with respect to $n \in \N$,
because of the ``critical layer'' that may occur at $r = r_n$. Indeed,
if we expand the expression $\gamma_n(r)$ around that point, we obtain
to leading order
\begin{equation}\label{eq:gammaexp}
  \gamma_n(r) \,=\ \Omega_n(r)-\Omega_n(r_n)-ia_n \,\approx\, \Omega_n'(r_n)
  \Bigl(r - r_n + i c_n\Bigr)\,, 
\end{equation}
where
\begin{equation}\label{eq:cndef}
   c_n \,=\, -\frac{a_n}{\Omega_n'(r_n)} \,=\, \frac{1}{d_n}\,
  \frac{a_n}{b_n^{3/2}}\,.
\end{equation}
It follows that, for $r$ close to $r_n$,  
\[
   \frac{\Phi_n(r)}{\gamma_n(r)^2} \,\approx\, \frac{J_n(r_n)}{(r
  -r_n+ic_n)^2}\,, \qquad \hbox{where} \quad 
  J_n(r_n) \,=\, \frac{\Phi_n(r_n)}{\Omega_n'(r_n)^2}\,.  
\]
We know that $J_n(r_n) \to J(\infty) $ as $n \to \infty$, where
$0 < J(\infty) < \frac{m^2}{4k^2}$ (by \eqref{eq:Jinfinipositif} and
because $r_* < \infty$). Thus, the term involving $\gamma_n(r)^{-2}$
in \eqref{eq:Dndef} converges to zero as $r \to \infty$ uniformly in
$n \in \N$ only if $c_n \to \infty$ as $n \to \infty$, which is the
case if $a_n \gg b_n^{3/2}$. Otherwise, that term plays an important
role and has to be taken into account.

Our strategy is thus to compare for large $r$ the solutions 
of \eqref{eq:weqn} with those of the simplified equation
\begin{equation}\label{eq:weqn2}
   w_n''(r) - \biggl(k^2 - \frac{k^2}{m^2}\,\frac{J_n(r_n)}{(r-r_n+ic_n)^2}
   \biggr)w_n(r) \,=\, 0\,,
\end{equation}
which can be solved explicitly in terms of modified Bessel functions. 
In particular, the unique solution of \eqref{eq:weqn2} such that 
$w_n(r) \sim e^{-k(r-r_n)}$ as $r \to \infty$ is given by
\begin{equation}\label{eq:chidef}
  w_n(r) \,=\, \chi_n(r) \,:=\, \sqrt{\frac{2}{\pi}}\,\bigl(r-r_n+ic_n\bigr)^{1/2} 
  K_{\nu_n}\bigl(k(r-r_n+ic_n)\bigr)\,, \qquad r > 0\,,
\end{equation}
where $K_{\nu}$ is the modified Bessel function of 
the second kind \cite[Section~9.6]{AS}, and the parameter $\nu_n 
\in (0,\frac12)$ is determined by the relation
\begin{equation}\label{eq:nundef}
  \nu_n^2 \,=\, \frac14 - \frac{k^2}{m^2}\, J_n(r_n)\,.
\end{equation}

To perform a rigorous analysis, we rewrite \eqref{eq:weqn} in the 
equivalent form
\begin{equation}\label{eq:weqn3}
  w_n''(r) + \frac{\cA'(r)}{\cA(r)}\,w_n'(r) - \biggl(k^2 - \frac{k^2}{m^2}\,
  \frac{J_n(r_n)}{(r-r_n+ic_n)^2} + \cR_n(r)\biggr)w_n(r) \,=\, 0\,,
\end{equation}
where the remainder $\cR_n$ is defined by
\begin{equation}\label{eq:Rndef}
  \cR_n(r) \,=\, \cD_n(r) - k^2 + \frac{k^2}{m^2}\,\frac{J_n(r_n)}{
  (r-r_n+ic_n)^2}\,, \quad r > 0\,.
\end{equation}
The idea is now to look for a solution of \eqref{eq:weqn3} in the form
\[
  w_n(r) \,=\, f_n(r) \chi_n(r)\,, \qquad r > 0\,,
\]
where $\chi_n$ is as in \eqref{eq:chidef}, and $f_n(r) \to 1$ as 
$r \to \infty$. The equation satisfied by $f_n$ is easily found to be
\[
  \frac{\D}{\D r}\Bigl(\cA(r)\chi_n(r)^2f_n'(r)\Bigr) \,=\, 
  \Bigl(\cA(r)\chi_n(r)^2 \cR_n(r) - \cA'(r)\chi_n(r)\chi_n'(r)
  \Bigr)f_n(r)\,, \qquad r > 0\,.
\]
Integrating both sides over $(r,\infty)$, we first obtain
\[
  f_n'(r) \,=\, -\frac{1}{\cA(r)\chi_n(r)^2}\int_r^\infty
  \Bigl(\cA(s)\chi_n(s)^2 \cR_n(s) - \cA'(s)\chi_n(s)\chi_n'(s)
  \Bigr)f_n(s)\dd s\,,
\]
and a second integration, combined with an application of Fubini's theorem, 
gives the representation formula
\begin{equation}\label{eq:fndef}
  f_n(r) \,=\, 1 + \int_r^\infty \cK_n(r,s) f_n(s)\dd s\,,
\end{equation}
where the integral kernel $\cK_n(r,s)$ has the following expression\:
\begin{equation}\label{eq:KKdef}
  \cK_n(r,s) \,=\, \Bigl(\cA(s)\chi_n(s)^2 \cR_n(s) - \cA'(s)\chi_n(s)\chi_n'(s)
  \Bigr)\int_r^s \frac{1}{\cA(t)\chi_n(t)^2}\dd t\,. 
\end{equation}

We now use the following estimate, whose proof is postponed to 
Section~\ref{subsec65}\:

\begin{lem}\label{lem:kernel}
For any $\delta \in (0,1)$, there exists a constant $C > 0$ such that
\begin{equation}\label{eq:Knest}
  \sup_{r \ge \delta r_n} \int_{r}^{\infty} |\cK_n(r,s)| \dd s \,\le\, 
  C\,b_n^{1/2} \Bigl(1 + \log_+\frac{1}{b_n}\Bigr)\,,  
\end{equation}
for all sufficiently large $n \in \N$, where $\log_+(x) = \max(\log(x),0)$.    
\end{lem}

Assuming \eqref{eq:Knest} for the moment, we easily  
deduce that the 
solution of \eqref{eq:fndef} satisfies, 
\begin{equation}\label{eq:fn}
  \sup_{r \ge \delta r_n} |f_n(r) - 1| \,\le\, C\,b_n^{1/2} \Bigl(1 + 
  \log_+\frac{1}{b_n}\Bigr) \,\xrightarrow[n \to \infty]{}\, 0\,.
\end{equation}
Also, differentiating \eqref{eq:fndef} and using similar estimates 
as in the proof of Lemma~\ref{lem:kernel}, we obtain
\begin{equation}\label{eq:fnder}
  \sup_{r \ge \delta r_n} |f'_n(r)| \,\le\, C\,b_n^{1/2} \Bigl(1 + 
  \log_+\frac{b_n}{a_n}\Bigr)\,.
\end{equation}
These estimates imply that the solution $w_n = f_n\chi_n$ of 
\eqref{eq:weqn3} is very close for large $n \in \N$ to its
approximation $\chi_n$ defined by \eqref{eq:chidef}, uniformly on the
interval $[\delta r_n, \infty)$ for any (small) $\delta > 0$. 
In particular, in view of \eqref{eq:chiest}, there exist positive 
constants $C_1, C_2$ such that
\begin{equation}\label{eq:wnbounds}
  \begin{split}
  |w_n(r)| \,&\le\, C_1\,e^{-k(r-r_n)} \quad \hbox{for} \quad r \ge r_n-1\,, \\
  |w_n(r)| \,&\ge\, C_2\,e^{-k(r-r_n)} \quad \hbox{for} \quad \delta r_n \le 
  r \le (1-\delta)r_n\,.
  \end{split}
\end{equation}

To reach the desired contradiction, we now show that these bounds
are {\em incompatible} with identity \eqref{eq:HG0im}, which has to be 
satisfied for all $n \in \N$ by the function $u_n(r) = r^{-1/2}w_n(r)$. 
In terms of $w_n$, identity \eqref{eq:HG0im} becomes
\[
  \int_0^\infty \biggl\{\frac{2\cA(r)(b_n-\Omega_n)}{(a_n^2 + (\Omega_n{-}
  b_n)^2)^2}\,\frac{k^2}{m^2}\,\Phi_n(r) + \frac{r}{a_n^2 + 
  (\Omega_n{-}b_n)^2}\partial_r \Bigl(\frac{W_n(r)}{m^2+k^2r^2}
  \Bigr)\biggr\}|w_n(r)|^2\dd r  \,=\, 0\,.
\] 
The second term in the integrand is obviously negative, because 
$W_n(r)$ is a decreasing function of $r$. It follows that $\cI_{n,1} 
+ \cI_{n,2} \le \cI_{n,3}$, where
\begin{align}\label{eq:In1def}
  \cI_{n,1} \,&=\, \int_0^{r_n-1} \frac{\cA(r)(\Omega_n(r)-b_n)}{(a_n^2 
  + (\Omega_n(r){-}b_n)^2)^2}\,\Phi_n(r) |w_n(r)|^2 \dd r\,, 
   \\[1mm] \label{eq:In2def}
  \cI_{n,2} \,&=\, \int_{r_n-1}^{r_n+1} \frac{\cA(r)(\Omega_n(r)-b_n)}{(a_n^2 
  + (\Omega_n(r){-}b_n)^2)^2}\,\Phi_n(r) |w_n(r)|^2 \dd r\,, 
  \\[1mm] \label{eq:In3def}
  \cI_{n,3} \,&=\, \int_{r_n+1}^\infty \frac{\cA(r)(b_n - \Omega_n(r))}{(a_n^2 
  + (\Omega_n(r){-}b_n)^2)^2}\,\Phi_n(r) |w_n(r)|^2 \dd r\,.
\end{align}
Since $\Omega_n(r)-b_n$ is positive for $r < r_n$ and negative for 
$r > r_n$, the quantities $\cI_{n,1}$ and $\cI_{n,3}$ are positive, 
whereas $\cI_{n,2}$ has no obvious sign. 

We first estimate $\cI_{n,3}$. As $b_n - \Omega_n(r) = \Omega_n(r_n) - 
\Omega_n(r) \ge \Omega_n(r_n) - \Omega_n(r_n{+}1) \ge C |\Omega'(r_n)|$ 
for $r \ge r_n+1$, we have
\[
  \frac{(b_n-\Omega_n(r))\Phi_n(r)}{(a_n^2 + (\Omega_n(r){-} b_n)^2)^2} \,\le\, 
  \frac{C\Phi_n(r)}{|\Omega'(r_n)|^3} \,\le\, \frac{C J_n(r)}{|\Omega'(r_n)|}
  \,\le\, C b_n^{-3/2}\,,\quad r \ge r_n+1\,,
\]
and using \eqref{eq:wnbounds}, \eqref{eq:In3def} we deduce that $\cI_{n,3} \le
C b_n^{-3/2}$. Next, we bound $\cI_{n,1}$ from below by restricting the integral 
in \eqref{eq:In1def} to the region where $\delta r_n \le r \le (1{-}\delta)r_n$,
for some small $\delta > 0$. In that region we have $\epsilon_1 b_n \le 
\Omega_n(r)-b_n \le \epsilon_2 b_n$ and $|\Omega_n'(r)| \ge \epsilon_3 b_n^{3/2}$
for some $\epsilon_1, \epsilon_2, \epsilon_3 > 0$ depending on $\delta$, 
hence 
\[
  \frac{(\Omega_n(r)-b_n)\Phi_n(r)}{(a_n^2 + (\Omega_n(r){-} b_n)^2)^2} 
  \,\ge\, \frac{\epsilon_1 b_n \Omega_n'(r)^2 J_n(r)}{(a_n^2 + (\epsilon_2 
  b_n)^2)^2} \,\ge\, \frac{Cb_n^4}{(a_n^2 + b_n^2)^2}\,, \quad
  \delta r_n \le r \le (1{-}\delta)r_n\,,
\]
and using \eqref{eq:wnbounds} we deduce that
\[
 \cI_{n,1} \,\ge\, \int_{\delta r_n}^{(1{-}\delta)r_n} \frac{\cA(r)(b_n-
 \Omega_n(r))}{(a_n^2 + (\Omega_n(r){-}b_n)^2)^2}\,\Phi_n(r) |w_n(r)|^2\dd r 
  \,\ge\,  \frac{Cb_n^4}{(a_n^2 + b_n^2)^2}\,e^{2k(1{-}\delta)r_n}\,.
\]
Finally, when $|r-r_n| \le 1$, we have $\Omega_n(r)-b_n \approx 
(r{-}r_n)\Omega_n'(r_n)$ hence
\[
  \frac{|\Omega_n(r)-b_n|\Phi_n(r)}{(a_n^2 + (\Omega_n(r){-}b_n)^2)^2}
  \,\le\, \frac{C |\Omega'(r_n)|^3 |r-r_n|}{(a_n^2 + \Omega'(r_n)^2
  (r{-}r_n)^2)^2}\,J_n(r_n)\,\le\, \frac{C}{ |\Omega'(r_n)|}\,
  \frac{|r-r_n|}{((r{-}r_n)^2 + c_n^2)^2}\,,
\]
and using \eqref{eq:wnbounds} we obtain the crude estimate
\[
  |\cI_{n,2}| \,\le\, C \int_{r_n-1}^{r_n+1} \frac{b_n^{-3/2}|r-r_n|}{((r{-}r_n)^2 
  + c_n^2)^2}\dd r \,\le\, C \int_{-\infty}^\infty \frac{b_n^{-3/2}|x|}{
  (x^2+c_n^2)^2}\dd x \,\le\, \frac{C b_n^{-3/2}}{c_n^2} \,\le\, 
  \frac{C b_n^{3/2}}{a_n^2}\,.
\]

As $\cI_{n,1} \le |\cI_{n,2}| + \cI_{n,3}$, the estimates obtained so 
far show that
\begin{equation}\label{eq:firstI}
  \frac{C_3 b_n^4}{(a_n^2 + b_n^2)^2}\,e^{2k(1{-}\delta)r_n}
  \,\le\, C_4 b_n^{-3/2} \Bigl(1 + \frac{b_n^3}{a_n^2}\Bigr)\,,
\end{equation}
for some positive constants $C_3, C_4$. If $a_n \ge b_n$ for a
sequence of integers $n$, then multiplying both sides of \eqref{eq:firstI} 
by $(a_n^2 + b_n^2)^2$ we clearly obtain an inequality that cannot be satisfied
for large $n$ if $a_n \to 0$. Thus we can assume that $a_n \le b_n$ 
for all $n \in \N$, in which case \eqref{eq:firstI} implies that 
\begin{equation}\label{eq:secondI}
  C_3 b_n^{3/2}\,e^{2k(1{-}\delta)r_n}  \,\le\, 4 C_4 \Bigl(1 + 
  \frac{b_n^3}{a_n^2}\Bigr)\,, \qquad \hbox{hence}\quad 
  a_n^2 \,\le\, C b_n^{3/2}\,e^{-2k(1{-}\delta)r_n}\,.
\end{equation}
Since $r_n = \cO(b_n^{-1/2})$, this means that $a_n$ is {\em 
exponentially small} when compared to $b_n$. With this information 
at hand, it is possible to compute explicitly the quantity $\cI_{n,2}$ 
to leading order as $n \to \infty$. Indeed, in this parameter 
regime, the main contribution to the integral \eqref{eq:In2def}
comes from an extremely small neighborhood of the critical 
point $r_n$, so that
\begin{align}\nonumber
  \cI_{n,2} \,&\approx\, -\frac{\cA(r_n)\Phi_n(r_n)}{|\Omega'(r_n)|^3}
  \int_{r_n-1}^{r_n+1} \frac{r-r_n}{(c_n^2 + (r{-}r_n)^2)^2}\,
  |w_n(r)|^2 \dd r \\ \label{eq:thirdI}
  \,&\approx\, -\frac{2\cA(r_n)J_n(r_n)}{\pi|\Omega'(r_n)|}
  \int_{r_n-1}^{r_n+1} \frac{r-r_n}{(c_n^2 + (r{-}r_n)^2)^{3/2}}\,
  |K_{\nu_n}(k(r-r_n+ic_n))|^2 \dd r\,, 
\end{align}
where in the second line we used the fact that the solution $w_n$ of
\eqref{eq:weqn} is well approximated for large $n$ by the function
$\chi_n$ in \eqref{eq:chidef} in view of \eqref{eq:fn} and
\eqref{eq:fnder}. Now, an explicit calculation which is reproduced in
Section~\ref{subsec66} shows that, for any $\nu \in (0,1/2)$ and any
$\epsilon > 0$,
\begin{equation}\label{eq:limBessel}
  \cJ_\nu \,:=\,\lim_{a \to 0_+} \int_{-\epsilon}^\epsilon \frac{-ax}{(a^2+x^2)^{3/2}}
  \,|K_\nu(x+ia)|^2\dd x \,=\, \frac{2\pi \cos(\nu\pi)}{1-4\nu^2}\,.
\end{equation}
Assuming \eqref{eq:limBessel} for the moment, we deduce from 
\eqref{eq:thirdI} that
\[
  \cI_{n,2} \,\approx\, \frac{2\cA(r_n)J_n(r_n)}{\pi|\Omega'(r_n)|}\,
  \frac{\cJ_{\nu_n}}{c_n}\,, \qquad \hbox{as } n \to \infty\,,
\]
which means in particular that $\cI_{n,2} > 0$ when $n$ is 
sufficiently large. Thus, for large $n$ we must have
\[
  C_5 \,e^{2k(1{-}\delta)r_n} \,\le\, \cI_{n,1} \,\le\, \cI_{n,1} + 
  \cI_{n,2} \le \cI_{n,3} \,\le\, C_6 r_n^3\,, 
\]
for some positive constants $C_5, C_6$, which is clearly impossible 
since $r_n \to \infty$ as $n \to \infty$. So we have reached a 
contradiction in that case too. 

\subsection{The situation $r_*=\rb=\infty$ is excluded}
\label{subsec49}

Having exhausted all possibilities for which $r_* < \infty$, we
finally consider the case where $r_* = \infty$. According to
\eqref{eq:rbstardef}, this occurs if and only if
$J(\infty) \ge m^2/(4k^2)$. Of course, we can assume that
$\rb = \infty$, because when $\rb < \infty$ a contradiction has
already been obtained in Section~\ref{subsec44} or \ref{subsec45}. If
$J(\infty) > m^2/(4k^2)$, then $J_n(r) > J_n(\infty) > m^2/(4k^2)$ for
all $r > 0$ when $n$ is sufficiently large, and in that situation we
know from Proposition~\ref{prop:HG} that the operator $L_{m,k}^n$ has
no eigenvalue outside the imaginary axis.  However, if
$J(\infty) = m^2/(4k^2)$, it is possible that
$J_n(\infty) < m^2/(4k^2)$ for all $n \in \N$, in which case we cannot
obtain a contradiction directly from Proposition~\ref{prop:HG}. In
that situation, we must have $J_n(r_n) \to m^2/(4k^2)$ as
$n \to \infty$. Two possibilities can occur\:

{\bf 1.} If $J_n(r_n) < m^2/(4k^2)$ for a sequence of integers $n$,
we can get a contradiction by following exactly the same lines as in
Section~\ref{subsec48}, the only difference being that the indices
$\nu_n$ defined by \eqref{eq:nundef} now converge to zero as
$n \to \infty$. This is harmless because, as is observed for
instance in Remark~\ref{rem:nuzero} below, all estimates we need hold
uniformly in the limit where $\nu \to 0$. We leave the details to the
reader.

{\bf 2.} If instead $J_n(r_n) \ge m^2/(4k^2)$ for all $n \in \N$, we
shall prove that the quadratic form given by the left-hand side of
\eqref{eq:HG1/2n} is positive definite for sufficiently large values
of $n$, so that \eqref{eq:HG1/2n} gives the desired contradiction. To
do that, we use the asymptotic expansions
\begin{equation}\label{eq:OJasym}
  \Omega_n(r) \,=\, \frac{\Gamma_n}{r^2} + \cO\Bigl(\frac{1}{r^4}\Bigr)\,, \quad
  \Omega_n'(r) \,=\, \frac{-2\Gamma_n}{r^3} + \cO\Bigl(\frac{1}{r^5}\Bigr)\,, \quad
  J_n'(r) \,=\, o\Bigl(\frac{1}{r}\Bigr)\,, \quad r \to \infty\,, 
\end{equation}
which are established in Section~\ref{subsec64}, and we observe that 
the integrand in \eqref{eq:HG1/2n} is nonnegative outside an interval of 
the form $[r_n,r_n+\delta_n]$, where $\delta_n \to 0$ as $n \to \infty$. 
Indeed, the integrand is clearly nonnegative for $r \le r_n$, and for 
$r \ge r_n$ we have the lower bound
\begin{equation}\label{eq:HGlow}
  \frac{\cA(r)\Omega_n'(r)^2}{a_n^2 + (b_n - \Omega_n(r))^2}\Bigl(
  \frac{k^2}{m^2}J_n(r)-\frac14\Bigr) \,\ge\, \frac{1}{m^2}\,
  \frac{\Omega_n'(r)^2}{(b_n - \Omega_n(r))^2}\Bigl(J_n(r) - J_n(r_n)\Bigr)\,.
\end{equation}
In view of \eqref{eq:OJasym}, the last member of \eqref{eq:HGlow} is
bounded from below by $-1$ if $r \ge r_n + \cO(|J'(r_n)|)$, which
proves the claim.  If we now restrict the integral in
\eqref{eq:HG1/2n} to the interval $[r_n-1,r_n+1]$, make the change of
variable $r=r_n+s$, and use the lower bound \eqref{eq:HGlow}, we
deduce (after a few obvious simplifications) that it is sufficient for
us to show that the quadratic form
\begin{equation}\label{eq:Qbizarre}
  \mathfrak{Q}_n(v) \,:=\,  \int_{-1}^1 \biggl\{ |\partial_s v|^2 + \Bigl(k^2
   - \frac{\eps_n s}{c_n^2 + s^2}\Bigr)|v|^2\biggr\}\dd s\,, \qquad k > 0\,,
\end{equation}
is positive on $H^1([-1,1],\dd s)$ for all $c_n \neq 0$ and all
sufficiently small $\eps_n > 0$. This in turn is an easy consequence
of the Sobolev embedding theorem. Indeed, decomposing $v=v(0)\chi + w$
where $\chi : [-1,1] \to [0,1]$ is smooth, even, and satisfies
$\chi(0)=1$, we first observe that
\[
  \int_{-1}^1 \frac{\eps_ns}{c_n^2+s^2}|v(0)|^2\chi^2(s)\dd s \,=\, 0\,,
\]
by symmetry. Moreover, as $w(0)=0$ by construction, we have 
$|w(s)|^2 \le C |s| \|v\|_{H^1}^2$. Combining these observations, 
we deduce that
\[
  \Bigl| \int_{-1}^1 \frac{\eps_n s}{c_n^2 + s^2}\,|v|^2 \dd s\Bigr| 
  \,\le\, C \eps_n  \|v\|_{H^1}^2\,,
\]
where the constant $C > 0$ is independent of $n$. The quadratic form
\eqref{eq:Qbizarre} is thus positive if $\eps_n$ is sufficiently
small, and we deduce that \eqref{eq:HG1/2n} cannot be satisfied. This
concludes the contradiction argument initiated in
Section~\ref{subsec41}, hence also the proof of Theorem~\ref{thm:main1}.

\section{Uniform resolvent estimates}\label{sec5} 

This section is devoted to the proof of Theorem~\ref{thm:main2}.
Given any $s \in \C$ with $\Re(s) \neq 0$, we already know from
Theorem~\ref{thm:main1} that the resolvent operator $(s-L_{m,k})^{-1}$ 
is bounded in the space $X_{m,k}$ for all $m \in \Z$ and all 
nonzero $k \in \R$. It remains to show that, for any $k_0 > 0$,
the resolvent norm $\|(s-L_{m,k})^{-1}\|$ is {\em uniformly 
bounded} for all $m \in \Z$ and all nonzero $k \in \Z k_0$.

Given $m \in \Z$, $k \neq 0$, and $\omega, f \in X_{m,k}$, the resolvent 
equation $(s - L_{m,k})\omega = f$ takes the form
\begin{align}\nonumber
  \gamma(r) \omega_r \,&=\, ik W(r) u_r + f_r\,, \\ \label{eq:res1}
  \gamma(r) \omega_\theta \,&=\, ik W(r)u_\theta + r\Omega'(r) 
    \omega_r + f_\theta\,, \\ \nonumber
  \gamma(r) \omega_z \,&=\, ik W(r) u_z - W'(r) u_r + f_z\,,
\end{align}
where $\gamma(r) = s + im\Omega(r)$. Here and in what follows, we
assume that $\Re(s) \neq 0$, which implies that $|\gamma(r)| \ge 
|\Re(s)| > 0$ for all $r > 0$. In \eqref{eq:res1}, it is understood that 
the velocity $u$ is obtained from the vorticity $\omega$ by the
Biot-Savart formula in the Fourier subspace indexed by $m$ and $k$,
see Section~\ref{subsec61} below.

Proceeding as in the derivation of the eigenvalue equation \eqref{eq:main} 
in Section~\ref{sec2}, we can transform the resolvent system~\eqref{eq:res2} 
into a single equation for the radial velocity $u_r$. After some 
calculations, we obtain the differential equation 
\begin{equation}\label{eq:res2}
  -\partial_r\bigl(\cA(r)\partial_r^* u_r\bigr) + 
  \cB(r) u_r \,=\, \cF(r)\,,
\end{equation}
where the coefficients in the left-hand side are as in \eqref{eq:ABdef}\: 
\begin{equation}\label{eq:AB2}
  \cA(r) \,=\, \frac{r^2}{m^2 + k^2 r^2}\,, \qquad  
  \cB(r) \,=\, 1 + \frac{k^2}{\gamma(r)^2}\,\cA(r)\Phi(r)
  + \frac{imr}{\gamma(r)}\partial_r\Bigl(\frac{W(r)}{m^2+k^2r^2}\Bigr)\,,
\end{equation}
and the right-hand side takes the form\:
\begin{equation}\label{eq:cFdef}
  \cF(r) \,=\, -\partial_r \Bigl(\frac{m}{k r\gamma(r)}\,\cA(r)f_r
  \Bigr) - \frac{i m^2}{k\gamma(r)^2}\,\frac{W(r)}{m^2+k^2r^2}\,f_r
  - \frac{i}{k\gamma(r)}\,f_\theta + \frac{2i\Omega(r)}{k\gamma(r)^2}
  \,f_z\,.
\end{equation}
Of course, if $f = 0$, then $\cF = 0$ and \eqref{eq:res2} reduces 
to \eqref{eq:main}. The following result will be useful to estimate
the solutions of \eqref{eq:res2} when $|k|$ is large. 

\begin{lem}\label{lem:res}
For any $m \in \Z$ and any $s \in \C$ with $\Re(s) \neq 0$, there
exists a positive constant $C = C(m,s)$ such that, for any $k \neq 0$
and any $f \in X_{m,k}$, the solution $u_r$ of \eqref{eq:res2}
satisfies
\begin{equation}\label{eq:res3}
 \|\cA^{1/2}\partial_r^* u_r\|_{L^2} + \|u_r\|_{L^2} \,\le\, 
 \frac{C(m,s)}{|k|}\,\bigl(\|u_r\|_{L^2} + \|f\|_{L^2}\bigr)\,. 
\end{equation}
\end{lem}

\begin{proof}
As in Section~\ref{subsec33}, we set $u_r(r) = \gamma(r)^{1/2}v(r)$. 
The new function $v$ satisfies the equation
\begin{equation}\label{eq:res4}
  -\partial_r\bigl(\cA(r)\gamma(r)\partial_r^* v\bigr) + 
  \cE(r) v \,=\, \gamma(r)^{1/2} \cF(r)\,,
\end{equation}
where
\begin{align}\nonumber
  \cE(r) \,&=\, \gamma(r)\cB(r) - \frac{\gamma'(r)}{2}\bigl(\cA'(r) - 
  \frac{\cA(r)}{r}\Bigr) -\frac12 \gamma''(r) \cA(r) + 
  \frac{\gamma'(r)^2\cA(r)}{4\gamma(r)} \\ \label{eq:Edefres}
  \,&=\, \gamma(r) + \frac{k^2}{\gamma(r)}\,\cA(r)\Phi(r)
  + \frac{imr}{2}\partial_r\Bigl(\frac{W(r)+2\Omega(r)}{m^2+k^2r^2}\Bigr)
  - \frac{m^2\Omega'(r)^2}{4\gamma(r)}\,\cA(r)\,.
\end{align}
We also observe that
\begin{equation}\label{eq:gammacF}
  \gamma^{1/2}\cF \,=\, -\partial_r \Bigl(\frac{m}{k r\gamma^{1/2}}
  \,\cA f_r \Bigr) - \frac{i m^2}{2k\gamma^{3/2}}\,\frac{W+2\Omega}{m^2+k^2r^2}
  \,f_r - \frac{i}{k\gamma^{1/2}}\,f_\theta + \frac{2i\Omega}{k\gamma^{3/2}}\,f_z\,.
\end{equation}

Without loss of generality, we assume that $a := \Re(s) > 0$. If we
multiply both sides of \eqref{eq:res4} by $r \bar v$, integrate the
resulting equality over $\Rp$ and take the real part, we obtain the
identity
\begin{align*}
  a \int_0^\infty \biggl\{\cA |\partial_r^* v|^2 + |v|^2 &+ \frac{\cA}{
  |\gamma|^2}\Bigl(k^2\Phi - \frac{m^2\Omega'^2}{4}\Bigr)\,|v|^2
  \biggr\}r\dd r \,=\, \Re\int_0^\infty (\partial_r^* \bar v)\frac{m\cA}{
  kr\gamma^{1/2}}\,f_r r\dd r \\ \label{eq:res5}
  &+ \Re\int_0^\infty \bar v\Bigl(- \frac{i m^2}{2k\gamma^{3/2}}
  \,\frac{W+2\Omega}{m^2+k^2r^2}\,f_r - \frac{i}{k\gamma^{1/2}}\,f_\theta 
  + \frac{2i\Omega}{k\gamma^{3/2}}\,f_z\Bigr)r \dd r\,.
\end{align*}
Keeping in mind that $\Phi(r) \ge 0$, $|\gamma(r)| \ge a$, and $0 < \cA(r) 
\le \min(1/k^2,r^2/m^2)$, we can estimate the various terms in a 
straightforward way, and we arrive at the inequality  
\begin{align*}
  a\Bigl(\|\cA^{1/2}\partial_r^* v\|_{L^2}^2 + \|v\|_{L^2}^2\Bigr) 
  \,&\le\, \frac{Cm^2}{a k^2}\|v\|_{L^2}^2 + \frac{C}{a^{1/2}|k|} \|\cA^{1/2}
  \partial_r^* v\|_{L^2}\|f_r\|_{L^2} \\
  &\quad + \frac{C}{a^{3/2}|k|}\|v\|_{L^2}\bigl(\|f_r\|_{L^2} + 
  a \|f_\theta\|_{L^2} + \|f_z\|_{L^2}\bigr)\,,
\end{align*}
where $C > 0$ is a universal constant. Applying now Young's inequality, 
we conclude that
\begin{equation}\label{eq:res5}
  \|\cA^{1/2}\partial_r^* v\|_{L^2} + \|v\|_{L^2} \,\le\, 
  \frac{C(m,a)}{|k|}\Bigl(\|v\|_{L^2} + \|f\|_{L^2}\Bigr)\,,
\end{equation}
where the constant depends only on $m$ and $a$. 

Finally, we return to the original function $u_r(r) = \gamma(r)^{1/2}v(r)$.
As $|\gamma(r)| \le |s| + |m|$ and 
\[
  \cA^{1/2} \partial_r^* u_r \,=\, \cA^{1/2}\gamma^{1/2} \Bigl(\partial_r^* v 
  + \frac{imr\Omega'}{2r\gamma}\,v\Bigr)\,,
\]
we have $|\cA^{1/2}\partial_r^* u_r| + |u_r| \le C(m,s)\bigl(|\cA^{1/2}
\partial_r^* v| + |v|\bigr)$. Thus the desired inequality \eqref{eq:res3} 
follows immediately from \eqref{eq:res5}. 
\end{proof}

Equipped with this lemma, we now establish the main result of 
this section. 

\begin{prop}\label{prop:periodic}
Fix any $k_0 > 0$. For any $s \in \C$ with $\Re(s) \neq 0$, there
exists a constant $C = C(s,k_0)$ such that, for all $m \in \Z$ 
and all nonzero $k \in \Z k_0$, the following estimate holds 
for all $f \in X_{m,k}$\:
\begin{equation}\label{eq:unifmk}
  \|(s -L_{m,k})^{-1}f\|_{L^2} \,\le\, C \|f\|_{L^2}\,.  
\end{equation}
\end{prop}

\begin{proof}
We proceed by contradiction. If \eqref{eq:unifmk} does not hold, 
there exist sequences $(m_n)$ in $\Z$, $(k_n)$ in $\Z^* k_0$, and 
$\omega^{(n)}, f^{(n)}$ in $X_{m,k}$ such that $\|\omega^{(n)}\|_{L^2} = 1$ 
for all $n \in \N$,  $\|f^{(n)}\|_{L^2} \to 0$ as $n \to \infty$, and
$(s - L_{m,k})\omega^{(n)} = f^{(n)}$ for all $n \in \N$, namely
\begin{align}\nonumber
  \bigl(s+im_n \Omega(r)\bigr) \omega^{(n)}_r \,&=\, ik_n W(r) u^{(n)}_r + 
     f^{(n)}_r\,, \\ \label{eq:resn}
  \bigl(s+im_n \Omega(r)\bigr) \omega^{(n)}_\theta \,&=\, ik_n W(r)u^{(n)}_\theta 
    + r\Omega'(r) \omega^{(n)}_r + f^{(n)}_\theta\,, \\ \nonumber
  \bigl(s+im_n \Omega(r)\bigr) \omega^{(n)}_z \,&=\, ik_n W(r) u^{(n)}_z 
  - W'(r) u^{(n)}_r + f^{(n)}_z\,.
\end{align}

\noindent{\bf Step 1.} We first show that the sequence $(m_n)$ 
is bounded. Indeed, if this is not the case, we can assume 
(after extracting a subsequence) that $|m_n| \to \infty$ as 
$n \to \infty$. In view of the first equation in \eqref{eq:resn}, 
this implies that 
\begin{equation}\label{eq:res6}
  \|\omega^{(n)}_r\|_{L^2} \,\le\,  \Bigl\|\frac{i k_n W u^{(n)}_r}{s+im_n 
  \Omega} \Bigr\|_{L^2} +  \Bigl\|\frac{f^{(n)}_r}{s+im_n\Omega}\Bigr\|_{L^2}
   \,\xrightarrow[n \to \infty]{}\, 0\,. 
\end{equation}
Indeed, we know from Proposition~\ref{prop:estBS} that $\|k_n u^{(n)}\|_{L^2}
\le C\|\omega^{(n)}\|_{L^2} \le C$ for all $n \in \N$, so that
\[
  \Bigl\|\frac{i k_n W u^{(n)}_r }{s+im_n \Omega}\Bigr\|_{L^2}
  \,\le\, \Bigl\|\frac{W}{s+im_n \Omega}\Bigr\|_{L^\infty}
  \|k_n u^{(n)}_r\|_{L^2}  \,\le\, C\,\Bigl\|\frac{W}{s+im_n \Omega}
  \Bigr\|_{L^\infty}\,\xrightarrow[n \to \infty]{}\, 0\,,
\]
and the last term in \eqref{eq:res6} is bounded by $|\Re(s)|^{-1} 
\|f^{(n)}_r\|_{L^2}$, a quantity that converges to zero as $n \to \infty$ 
by assumption. Once \eqref{eq:res6} is known, the same argument applied 
to the second equation in \eqref{eq:resn} shows that $\|\omega^{(n)}_\theta
\|_{L^2} \to 0$. Finally, we have $\|u^{(n)}_r\|_{L^2} \le k_0^{-1} \|k_n 
u^{(n)}_r\|_{L^2} \le C k_0^{-1}$, because $|k_n| \ge k_0$ for all $n \in \N$. 
Applying thus the same argument again to the third equation in
\eqref{eq:resn}, we conclude that $\|\omega^{(n)}_z\|_{L^2} \to 0$,
which of course contradicts the hypothesis that $\|\omega^{(n)}\|_{L^2} = 1$ 
for all $n \in \N$. This means that sequence $(m_n)$ must be bounded, and 
after extracting a subsequence we can therefore assume that there exists 
an integer $m \in \Z$ such that $m_n = m$ for all $n \in \N$.

\medskip\noindent{\bf Step 2.} We next show that the sequence $(k_n)$
is bounded. Again, if this is not the case, we can assume after
extracting a subsequence that $|k_n| \to \infty$ as $n \to \infty$.
In that situation, we infer from estimate \eqref{eq:res3} that, for
$n$ sufficiently large,
\begin{equation}\label{eq:urconv}
  |k_n|\Bigl(\|\cA_n^{1/2}\partial_r^* u_r^{(n)}\|_{L^2} + \|u_r^{(n)}\|_{L^2}
  \Bigr) \,\le\, 2 C(m,s) \|f^{(n)}\|_{L^2} \,\xrightarrow[n \to \infty]{}\, 
  0\,.
\end{equation}
Next, we use the relation
\[
  k_n^2 \cA_n \Bigl(\partial_r^* - \frac{imW}{r\gamma}\Bigr)u_r^{(n)} 
  + i k_n u^{(n)}_z \,=\, \frac{m k_n}{r\gamma}\,\cA_n f_r^{(n)}\,,
\]
which reduces to \eqref{eq:sys4} when $k_n = k$ and $f_r^{(n)} = 0$. Invoking 
\eqref{eq:urconv} and using the elementary bounds $0 < \cA_n(r) 
\le \min(1/k_n^2,r^2/m^2)$, we deduce that
\begin{equation}\label{eq:uzconv}
  |k_n|\,\|u_z^{(n)}\|_{L^2} + |m|\,\Bigl\|\frac{u^{(n)}_z}{r}\Bigr\|_{L^2}
  \,\xrightarrow[n \to \infty]{}\, 0\,.
\end{equation}
Finally, with the help of the additional relation
\[
  \frac{ik_nW}{\gamma}\,u^{(n)}_r + ik_n u^{(n)}_\theta - \frac{im}{r}\,u^{(n)}_z
  \,=\, -\frac{f^{(n)}_r}{\gamma}\,,
\]
which reduces to \eqref{eq:sys1} when $k_n = k$ and $f_r^{(n)} = 0$, we find that 
$|k_n|\|u^{(n)}_\theta\|_{L^2} \to 0$ as $n \to \infty$ in view of \eqref{eq:urconv}, 
\eqref{eq:uzconv}. 

Thus, we have shown that $|k_n|\|u^{(n)}\|_{L^2} \to 0$ as $n \to \infty$, 
and considering successively all three lines in \eqref{eq:resn} we easily
deduce that 
\[
  \|\omega^{(n)}_r\|_{L^2} \,\xrightarrow[n \to \infty]{}\, 0\,, \qquad
  \|\omega^{(n)}_\theta\|_{L^2} \,\xrightarrow[n \to \infty]{}\, 0\,, \qquad
  \|\omega^{(n)}_z\|_{L^2} \,\xrightarrow[n \to \infty]{}\, 0\,. 
\]
This of course contradicts the assumption that $\|\omega^{(n)}\|_{L^2} 
= 1$ for all $n \in \N$. The sequence $(k_n)$ must therefore 
be bounded, and after extracting a subsequence we can assume 
that $k_n = k$ for some fixed $k \in \Z^* k_0$. 

\medskip\noindent{\bf Step 3.} Assuming that estimate \eqref{eq:unifmk}
does not hold for some $s \in \C$ with $\Re(s) \neq 0$, we have reached 
the conclusion that, for some $m \in \Z$ and some $k \neq 0$, the operator 
$s - L_{m,k}$ has no bounded inverse in $X_{m,k}$, in contradiction
with Theorem~\ref{thm:main1}. Thus estimate  \eqref{eq:unifmk} must 
hold, and the proof of Proposition~\ref{prop:periodic} is complete.  
\end{proof}

\section{Appendix}\label{appendix}

\subsection{The Biot-Savart law in cylindrical coordinates}\label{subsec61}

The Biot-Savart law defines the velocity field $u = (u_r,u_\theta,u_z)$ 
in terms of the vorticity vector $\omega = (\omega_r,\omega_\theta,\omega_z)$, 
for a fixed value of the angular Fourier mode $m \in \Z$ and of the 
vertical wave number $k \in \R$. The velocity is determined by
the linear relations
\begin{equation}\label{eq:BS}
  \omega_r \,=\, \frac{im}r u_z - ik u_\theta\,, \qquad
  \omega_\theta \,=\, ik u_r - \partial_r u_z\,, \qquad
  \omega_z \,=\, \frac1r \partial_r (r u_\theta) - \frac{im}r u_r\,,
\end{equation}
together with the divergence-free condition 
\begin{equation}\label{eq:divu}
  \frac1r\partial_r(ru_r) + \frac{im}r u_\theta + ik u_z \,=\, 0\,.
\end{equation}
These equations have to be solved on the half-line $r > 0$, and we
require that the velocity field $u_re_r+u_\theta e_\theta+u_z e_z$ be
regular at the origin $r = 0$ and decay to zero as $r \to
\infty$. More precisely, if the vorticity $\omega$ is (for instance) 
compactly supported in $\Rp = (0,\infty)$, the following boundary conditions
hold for the associated velocity $u$\:

\begin{itemize}
\item The horizontal velocities $u_r$, $u_\theta$ satisfy the 
homogeneous Dirichlet condition at $r = 0$ if $m = 0$ or 
$|m| \ge 2$, and the homogeneous Neumann condition if $|m| = 1$ 
(or $|m| \ge 3$).
\item The vertical velocity $u_z$ satisfies the homogeneous Dirichlet 
condition at $r = 0$ if $|m| \ge 1$, and the homogeneous Neumann 
condition if $m = 0$ (or $|m| \ge 2$). 
\end{itemize}

It is possible to give explicit formulas for the velocity $u$ in terms
of the vorticity $\omega$, but the bounds we need in this paper are
more conveniently obtained by standard energy estimates. We recall
that $\|\cdot\|_{L^2}$ denotes the usual norm in the Lebesgue space
$L^2(\R_+,r\dd r)$.

\begin{prop}\label{prop:estBS}
There exists a constant $C > 0$ such that, for any $m \in \Z$ and any 
$k \in \R$, the following inequality holds
\begin{equation}\label{eq:unifBS}
\begin{split}
  &\|\partial_r u_r\|_{L^2}^2 +  \|\partial_r u_\theta\|_{L^2}^2 + 
  \|\partial_r u_z\|_{L^2}^2 + k^2\Bigl(\|u_r\|_{L^2}^2 + 
  \|u_\theta\|_{L^2}^2 + \|u_z\|_{L^2}^2\Bigr) \\
  & + |m^2 - 1| \Bigl(\Bigl\|\frac{u_r}{r}\Bigr\|_{L^2}^2 + 
  \Bigl\|\frac{u_\theta}{r}\Bigr\|_{L^2}^2\Bigr) + m^2 \Bigl\|\frac{u_z}{r}
  \Bigr\|_{L^2}^2 \,\le\, C\Bigl(\|\omega_r\|_{L^2}^2 +  \|\omega_\theta\|_{L^2}^2 + 
  \|\omega_z\|_{L^2}^2\Bigr)\,.
\end{split}
\end{equation}
\end{prop}

\begin{proof} 
We assume here for definiteness that $k \neq 0$, but the proof is 
similar (and in fact simpler) when $k = 0$. Without loss of generality, 
we also suppose that $\omega$ is continuous and compactly supported 
in $\Rp$. We first observe that the vertical velocity $u_z$ satisfies the 
linear elliptic equation
\begin{equation}\label{eq:diffequz}
  \Bigl(-\partial_r^2 - \frac1r \partial_r + \frac{m^2}{r^2} + k^2
  \Bigr) u_z \,=\, \frac1r \partial_r (r\omega_\theta) - \frac{im}{r}
  \,\omega_r\,.
\end{equation}
We multiply both sides of \eqref{eq:diffequz} by $r\bar u_z$ 
and integrate the resulting expression over $\Rp$. After elementary 
calculations, we obtain the estimate
\begin{equation}\label{eq:uzbound}
  \|\partial_r u_z\|_{L^2}^2 + \|\frac{m}{r}u_z\|_{L^2}^2 +  
  \|k u_z\|_{L^2}^2 \,\le\, C\Bigl( \|\omega_r\|_{L^2}^2 + 
  \|\omega_\theta\|_{L^2}^2\Bigr)\,,
\end{equation}
where $C > 0$ is a universal constant. As $ik u_r = \partial_r u_z 
+ \omega_\theta$ and $ik u_\theta = \frac{im}{r}u_z - \omega_r$, 
it follows immediately from \eqref{eq:uzbound} that
\begin{equation}\label{eq:urt1}
   \|k u_r\|_{L^2}^2 + \|k u_\theta\|_{L^2}^2 \,\le\, C\Bigl(
   \|\omega_r\|_{L^2}^2 + \|\omega_\theta\|_{L^2}^2\Bigr)\,.
\end{equation}

On the other hand, we deduce from \eqref{eq:divu} and the last 
relation in \eqref{eq:BS} that
\begin{equation}\label{eq:urtsys}
  \partial_r u_r + \frac1r \bigl(u_r + im u_\theta) \,=\, -ik u_z\,, \qquad
  \partial_r u_\theta + \frac1r \bigl(u_\theta - im u_r) \,=\, \omega_z\,.
\end{equation}
We multiply the first equation by $r \partial_r \bar u_r$ and the second 
one by $r \partial_r \bar u_\theta$. Adding the resulting expressions, 
taking the real parts, and integrating over $\R_+$, we obtain the inequality 
\begin{equation}\label{eq:urt2}
  \|\partial_r u_r\|_{L^2}^2 + \|\partial_r u_\theta\|_{L^2}^2 \,\le\, 
  C\Bigl(\|k u_z\|_{L^2}^2 + \|\omega_z\|_{L^2}^2\Bigr) \,\le\, 
  C\Bigl( \|\omega_r\|_{L^2}^2 + \|\omega_\theta\|_{L^2}^2 + 
   \|\omega_z\|_{L^2} ^2\Bigr)\,.
\end{equation}
If $m = \pm 1$, this concludes the proof of \eqref{eq:unifBS}. 
Otherwise, we deduce from \eqref{eq:urtsys} that
\begin{equation}\label{eq:urtsys2}
\begin{split}
    \frac{m^2-1}{r}\,u_r \,&=\, \partial_r (u_r - im u_\theta) 
    + iku_z + im \omega_z\,, \\
   \frac{m^2-1}{r}\,u_\theta \,&=\, \partial_r (u_\theta + im u_r) 
    + km u_z - \omega_z\,.
\end{split}
\end{equation}
If $m = 0$ or $|m| \ge 2$, these relations allow us to estimate 
the $L^2$ norm of $u_r/r$ and $u_\theta/r$ in terms of 
quantities that are already controlled by \eqref{eq:uzbound}
or \eqref{eq:urt2}, and we arrive at \eqref{eq:unifBS}. 
\end{proof}

\subsection{Stability of Rankine's vortex}\label{subsec62}

We consider here in some detail the particular case of the Rankine
vortex \eqref{eq:Rankine}, which is of historical relevance. We do not
use the functional framework of Section~\ref{sec2} because, as is
clear from \eqref{eq:Bdef}, the linearization $L_{m,k}$ does not
define a bounded linear operator on $X_{m,k}$ if the vorticity profile
$W$ has a discontinuity. Instead we look for solutions of the
eigenvalue equation \eqref{eq:eigeq} where the velocity field $u$ (and
not the vorticity $\omega)$ belongs to $X_{m,k}$. We always assume that 
$m \neq 0$ and $k \neq 0$, the other cases being similar and in fact 
simpler. To avoid the essential spectrum, we also suppose that 
the spectral parameter $s \in \C$ satisfies $s \neq 0$ and $s + im \neq 0$. 

Following Kelvin's original approach \cite{Ke}, we eliminate the
radial velocity $u_r$ in the $2 \times 2$ system 
\eqref{eq:sys4}--\eqref{eq:sys5} to obtain a closed equation for the
vertical velocity $u_z$. In the inner region where $0 < r < 1$, we
have $\gamma(r) = \gamma := s +im$ and $\Phi(r) = W(r)^2 = 4$, so that
$u_z$ satisfies the Bessel equation
\begin{equation}\label{eq:Bessel}
  - \frac{1}{r}\partial_r\big(r \partial_r u_z\big) + 
  \Bigl(\beta^2 + \frac{m^2}{r^2}\Bigr)u_z \,=\, 0\,, 
  \qquad \hbox{where}\quad  \beta^2 \,=\, k^2\Bigl(1 
  + \frac{4}{\gamma^2}\Bigr)\,.
\end{equation}
Since $u_z$ is regular at the origin, it follows that $u_z(r) = A 
I_m(\beta r)$ for $0 < r < 1$, where $A \in \C$ and $I_m$ is 
the modified Bessel function of order $m$ \cite[Section~9.6]{AS}. 
In the outer region where $r > 1$, we have $W(r) = \Phi(r) = 0$, and 
system \eqref{eq:sys4}--\eqref{eq:sys5} reduces to the (somewhat 
simpler) Bessel equation
\begin{equation}\label{eq:Bessel2}
  - \frac{1}{r}\partial_r\big(r \partial_r u_z\big) + 
  \Bigl(k^2 + \frac{m^2}{r^2}\Bigr)u_z \,=\, 0\,. 
\end{equation}
As $u_z(r)$ decays to zero at infinity, we must have $u_z(r) = 
B K_m(kr)$ for some $B \in \C$, where $K_m$ is again a modified 
Bessel function. 

At the interface $r = 1$, both velocities $u_z$, $u_r$ are continuous,
as can be seen from \eqref{eq:sys2} and \eqref{eq:divu2}. Jump
conditions for the first order derivatives can be deduced from
system \eqref{eq:sys4}--\eqref{eq:sys5} and are found to be
\begin{align}\label{eq:matchr}
  \partial_r u_r(1_+) \,&=\, \partial_r u_r(1_-) -\frac{2im}{\gamma}
  \,u_r(1)\,, \\ \label{eq:matchz}
  \partial_r u_z(1_+) \,&=\, \frac{\gamma^2}{\gamma^2 + 4}
  \Bigl(\partial_r u_z(1_-) + \frac{2im}{\gamma}\,u_z(1)\Bigr)
  \,=\, ik u_r(1)\,.
\end{align}
In particular, as $u_z(r) =  A I_m(\beta r)$ for $r < 1$ and
$u_z(r) =  B K_m(k r)$ for $r > 1$, we must have
\begin{equation}\label{eq:ABrel}
  A I_m(\beta) \,=\, B K_m(k)\,, \qquad 
  \frac{A \gamma^2}{\gamma^2+4}\Bigl(\beta 
  I_m'(\beta)+ \frac{2im}{\gamma} I_m(\beta)\Bigr) 
  \,=\, B k K_m'(k)\,.
\end{equation}
This linear system has a nontrivial solution $(A,B)$ if and only if 
\begin{equation}\label{eq:dispersion}
  \frac{I_m'(\beta)}{\beta I_m(\beta)} + \frac{2im}{\gamma\beta^2}
  \,=\, \frac{K_m'(k)}{k K_m(k)}\,,
\end{equation}
where we recall that $\gamma = s+im \neq 0$ and $\beta^2 = 
k^2(1 + 4/\gamma^2)$. 

It was already observed by Kelvin that the dispersion relation
\eqref{eq:dispersion} is satisfied for a countable set of {\em purely
  imaginary} values of the spectral parameter $s$.  More precisely, if
we define $s = -imb$, so that $\gamma = im(1-b)$, equality
\eqref{eq:dispersion} holds for a decreasing sequence of values of $b$
accumulating at $1$, and also for an increasing sequence accumulating
at $1$, all solutions being contained in the interval $|b-1| \le 2/|m|$ 
\cite{Ke}. The linearized operator at Rankine's vortex thus has 
a countable family of purely imaginary eigenvalues (Kelvin modes). 
However, it is not easy to verify that the dispersion relation 
\eqref{eq:dispersion} has no solution when $s \notin i\R$, and 
there is no such claim in Kelvin's work\footnote{Except for 
an ambiguous sentence asserting, without any justification, that the 
eigenfunctions corresponding to purely imaginary eigenvalues should
form a complete family.} where only purely imaginary eigenvalues are 
considered. Thus, contrary to what is often asserted in the literature, 
stability of Rankine's vortex was not established by Kelvin, and 
we could not find any further reference where this point is 
clarified. 

Fortunately, it is quite easy to prove spectral stability of Rankine's
vortex following the approach of Section~\ref{subsec33}. Indeed,
taking into account the particular form of the vorticity profile
\eqref{eq:Rankine}, it is straightforward to verify that identity 
\eqref{eq:HG0} becomes
\begin{align}\nonumber
  &\int_0^\infty \Bigl(\cA(r)|\partial_r^* u_r|^2 + |u_r|^2\Bigr)r 
  \dd r \\ \label{eq:Rank1} 
  &\qquad\qquad + \int_0^1 \biggl\{-\frac{4k^2\cA(r)}{m^2 \ggamma^2} + 
  \frac{2r}{\ggamma}\partial_r\Bigl(\frac{1}{m^2+k^2 r^2}\Bigr)
  \biggr\} |u_r|^2 r \dd r \,=\, \frac{2\cA(1)}{\ggamma}\,|u_r(1)|^2\,,
\end{align}
see also Remark~\ref{rem:alter} below. Here $\ggamma = 1 - b - ia$, 
so that $\gamma = im \ggamma$. We now multiply both sides of
\eqref{eq:Rank1} by $\ggamma$ and take the imaginary parts. 
We arrive at the identity
\begin{equation}\label{eq:Rank2} 
  a \int_0^\infty \Bigl(\cA(r)|\partial_r^* u_r|^2 + |u_r|^2\Bigr)r 
  \dd r + a \int_0^1 \frac{4k^2}{m^2}\,\frac{\cA(r)}{
  (1-b)^2 + a^2}\,|u_r|^2 r \dd r \,=\, 0\,.
\end{equation}
If $a \neq 0$, it follows from \eqref{eq:Rank2} that $u_r \equiv 0$, 
hence the eigenvalue equation \eqref{eq:main} has no nontrivial 
solution if $s = m(a-ib) \notin i\R$. This proves that the linearized
operator at Rankine's vortex has no unstable eigenvalue. 

\begin{rem}\label{rem:alter}
Alternatively, one can obtain the relation \eqref{eq:Rank1} by 
restricting the eigenvalue equation \eqref{eq:main} to the open 
intervals $(0,1)$ and $(1,\infty)$, where the vorticity profile is 
smooth. On each interval, we multiply \eqref{eq:main} by $r\bar u_r$ 
and we integrate over $r$. If we add the resulting expressions and 
simplify the boundary terms (which result from partial integrations) 
using the matching condition \eqref{eq:matchr}, we arrive at 
\eqref{eq:Rank1}. 
\end{rem}

\subsection{Critical layers and their continuity properties}
\label{subsec63}

In this section we present the proof of Lemma~\ref{lem:approxsing}. 
We first rewrite Eq.~\eqref{eq:main2} for $u_n$ in the form
\begin{equation}\label{eq:odemod}
  u_n''(r) + \cP(r)u_n'(r) + \cQ_n(r)u_n(r)=0\,,
\end{equation}
where $\cP(r) = \cA'(r)/\cA(r)+ 1/r$ and 
\[
  \cQ_n(r) \,=\, \frac{k^2}{m^2}\,\frac{\Phi_n(r)}{\gamma_n(r)^2}
  - \frac{r}{\cA(r)\gamma_n(r)}\,\partial_r\Bigl(\frac{W_n(r)}{m^2+k^2r^2}
  \Bigr) + \frac{\cP(r)}{r} - \frac{2}{r^2}\,.
\]
Here, as in Section~\ref{subsec43}, we denote $\gamma_n(r) = 
\Omega_n(r)-b_n-ia_n$, where $\Omega_n$ is the angular velocity
associated with $W_n$ as in \eqref{eq:Omrep}. By assumption
$ii)$ and \eqref{eq:Omrep} we have $\Omega_n \to \Omega$ in $\cC^2$ on
compact subsets of $(0,\infty)$. In view of $iii)$, $\Omega$ is 
analytic in $\bbD(\rb,\rho)$ for some $\rho > 0$, and $\Omega_n$ 
converges uniformly to $\Omega$ on that disc as $n \to \infty$. 
Since $\Omega(\rb) = b$ and $\Omega'(\rb) < 0$, it follows from 
Hurwitz's theorem that, for sufficiently large $n$, there exists a 
unique $\rb_n \in \bbD(\rb,\rho)$ such that $\Omega_n(\rb_n) = b_n+ia_n$. 
Moreover $\Omega_n'(\rb_n)\neq 0$, $\Omega_n'(\rb) < 0$ and
\[
  \rb_n \,=\, \rb + \frac{b_n + i a_n - \Omega_n(\rb)}{\Omega_n'(\rb)}  
  + \cO\bigl(|b_n - \Omega_n(\rb)|^2 + a_n^2 \bigr)\,, \quad 
   \hbox{as } n \to \infty\,,
 \] 
so that $\rb_n \to \rb$ and  $\rb_n \in \{ z \in \C\,;\, {\rm Im}(z)<0\}$ 
when $n$ is sufficiently large. By construction, we also have
$\gamma_n(r) = \Omega_n'(\rb_n)(r-\rb_n) + \cO(|r-\rb_n|^2)$ as
$r \to \rb_n$.

Multiplying \eqref{eq:odemod} by $z^2$ and applying the change of variables 
$z = r-\rb_n$, $w_n(z) = u_n(\rb_n + z)$, we obtain the canonical form  
\begin{equation}\label{eq:odemodcomp}
  z^2 w_n''(z) + z P_n(z)w_n'(z) + Q_n(z)w_n(z)=0\,,
\end{equation}
where $P_n(z) = z\cP(\rb_n + z)$ and $Q_n(z) = z^2\cQ_n(\rb_n + z)$
are analytic inside the disc $\bbD(0,\rho/2)$, if $n$ is large enough
so that $|\rb_n-\rb| < \rho/2$. In this situation, the Frobenius
method \cite[Section~4.8]{CoLe}, which we briefly recall now, can be
used to construct solutions of \eqref{eq:odemodcomp} in
$\bbD(0,\rho/2)\setminus \R_-$ of the form
\begin{equation}\label{eq:sviluppo}
  w_n(z) = z^{d_n} v_n(z)\,, \quad  \hbox{where }\, v_n(z) \,=\, 
  \sum_{k=0}^\infty c_{n,k} z^k \,\hbox{ and } \,c_{n,0} \,=\, 1
  \hbox{ for all } n\,. 
\end{equation}
The coefficients $c_{n,k}$ for $k\ge 1$ are determined by substituting 
\eqref{eq:sviluppo} into \eqref{eq:odemodcomp} and collecting equal powers 
of $z$. If $P_n(z) = \sum_{k=0}^\infty p_{n,k}z^k$ and $Q_n(z) = \sum_{k=0}^\infty 
q_{n,k} z^k$, one obtains the recursion relations
\begin{equation}\label{eq:recu}
  c_{n,k} = \frac{-1}{f_n(d_n+k)}\sum_{j=0}^{k-1} c_{n,j}\left[ (d_n+j) p_{n,k-j} +
  q_{n,k-j}\right]\,, \qquad k\geq1\,, 
\end{equation}
where the {\em indicial function} $f_n : \R \to \R$ is defined by
\begin{equation}\label{eq:indif}
  f_n(d) \,=\, d(d-1)+ d p_{n,0}+q_{n,0}\,. 
\end{equation}
Assuming that $f_n(d_n+k) \neq 0$ for all $k \ge 1$, it is straightforward to verify that the formal series $w_n(z)$ defined by 
\eqref{eq:sviluppo}, \eqref{eq:recu} satisfies
\begin{equation}\label{eq:frobapprox}
  z^2 w_n''(z) + z P_n(z)w_n'(z) + Q_n(z)w_n(z) \,=\, f_n(d_n)z^{d_n}\,,
\end{equation}
hence is a (formal) solution of \eqref{eq:odemodcomp} provided $d_n$ is a 
root of the quadratic polynomial $f_n$. 

In our situation we have $p_{n,0} = 0$ and
$q_{n,0} = (k^2/m^2)J_n(\rb_n)$, so that the indicial equation
$f_n(d_n) = 0$ reduces to Eq.~\eqref{eq:determinant} as
$n \to \infty$. The roots $d_n^\pm$ of $f_n$ thus converge to the
explicit values $d_\pm$ described in Section~\ref{subsec34}, which are
such that $|d_+ - d_-| < 1$.  As a consequence, if $n$ is large enough
and $d_n = d_n^+$ or $d_n^-$, the denominator in \eqref{eq:recu} never
vanishes, and even satisfies $|f_n(d_n +k)| \ge c_0 k^2$ for all
$k \ge 1$, where $c_0 > 0$ is independent of $n$. This allows us to
solve the recursion relations \eqref{eq:recu} if $d_n = d_n^\pm$, and
it is then straightforward to verify that the series in
\eqref{eq:sviluppo} converges for all $z \in \bbD(0,\rho/2)$, and that
the sum $v_n^\pm$ is uniformly bounded on compact subsets of
$\bbD(0,\rho/2)$ if $n$ is sufficiently large.  We denote henceforth
by $w_n^\pm(z) = z^{d_n^\pm} v_n^\pm(z)$ the solution of
\eqref{eq:odemodcomp} given by \eqref{eq:sviluppo} with
$d_n = d_n^\pm$.

By assumption $iii)$, the quantities $p_{n,k}$, $q_{n,k}$ converge as
$n \to \infty$ to the Taylor coefficients of the functions $P(z)$,
$Q(z)$ associated with the limiting profile $W$ and the limiting value
of the spectral parameter. Using the recursion relation
\eqref{eq:recu}, where each coefficient $c_{n,k}^\pm$ is entirely
determined by a finite number of coefficients $p_{n,\ell}$,
$q_{n,\ell}$ (namely, those with $\ell < k$), we see that
\begin{equation}\label{eq:convcoefftaylor}
	c_{n,k}^\pm \,\to\, c_k^\pm \quad\hbox{as }\,n\to \infty\,, 
\end{equation}
where $c_k^\pm$ denote the coefficients of the Frobenius solution
$w^\pm(z) = z^{d_\pm}v^\pm(z)$ of the limiting equation
\eqref{eq:odemodcomp}, where $P_n,Q_n$ are replaced by $P,Q$. In view
of the uniform bounds mentioned above, this implies that $v_n^\pm(z)$
converges to $v^\pm(z)$ uniformly on compact subsets of
$\bbD(0,\rho/2)$ as $n \to \infty$. Note that since $P$ and $Q$ are
real valued on the real axis, the recurrence relation yields that the
coefficients $c_k^\pm$ are real too. The functions $V_\pm$ which appear 
in the formulas \eqref{eq:devsing1}--\eqref{eq:devsing3} are 
the only real analytic functions on $\Rp$ that satisfy
\[
  V_\pm(r)\,=\,v^\pm(r-\rb) \left(\frac{r-\rb}{b-\Omega(r)}\right)^{d_\pm}\,,
  \qquad r \in (\rb-\rho/2,\rb+\rho/2)\,.
\]
That $V_\pm$ are well defined and real analytic on the whole half-line
$\Rp$ follows from the representation \eqref{eq:devsing1} and the ODE
\eqref{eq:main2sing}.

Now, consider a sequence $(u_n)_{n\in\N}$ of solutions of
\eqref{eq:odemod} as in the statement of Lemma~\ref{lem:approxsing},
and assume first that $d_+\neq d_-$, so that $d_n^+\neq d_n^-$ when
$n$ is sufficiently large. Since \eqref{eq:odemod} is a second order
differential equation, there exist complex coefficients $\alpha_n^\pm$
such that\footnote{Note that $\bbD(\rb_n,\rho/2)$ contains the real
  interval $(\rb-\rho/4,\rb+\rho/4)$ for large values of $n$.}
\begin{equation}\label{eq:splitting}
  u_n(r) = \alpha_n^+ w_n^+(r-\rb_n) + \alpha_n^- w_n^-(r-\rb_n)\,, 
  \quad \hbox{for }\, r \in (\rb-\rho/4,\rb+\rho/4)\,.  
\end{equation}
By assumption $u_n(r)$ and $u_n'(r)$ have a limit as $n \to \infty$
for some $r \neq \rb$, and using elementary continuity properties for
solutions of nonsingular ODEs we deduce that convergence holds locally
uniformly for all $r > \rb$, or for all $r < \rb$. Since the functions
$w_n^\pm(r-\rb_n)$ converge uniformly to $w^\pm(r-\rb)$ in a
neighborhood of $\rb$, and since the limits $w^\pm(r-\rb)$ have
genuinely different behaviors as $r \to \rb$, this implies that the
coefficients $\alpha_n^\pm$ in \eqref{eq:splitting} have finite limits
$\alpha_\pm \in \C$ as $n \to \infty$. In particular we have
\[
  u_n(r) \,\to\, \alpha_+ (r-\rb)^{d_+}v^+(r-\rb) + \alpha_- (r-\rb)^{d_-} 
  v^-(r-\rb)\,,
\]
uniformly for $r \in (\rb,\rb+\rho/4)$, and (keeping in mind that 
${\rm Im}(\rb_n) < 0$)
\[
  u_n(r) \,\to\, \alpha_+ e^{i\pi d_+}(\rb-r)^{d_+}v^+(r-\rb) + 
  \alpha_- e^{i\pi d_-} (\rb-r)^{d_-} v^- (r-\rb)\,,
\]
uniformly for $r \in (\rb-\rho/4,\rb)$. Since outside the interval
$(\rb-\rho/4,\rb+\rho/4)$ the ODE \eqref{eq:odemod} is asymptotically
regular, this implies the desired conclusion, namely that
$u_n \to \alpha_+ \phi_+ + \alpha_- \phi_-$ where $\phi_\pm$ are as in
\eqref{eq:devsing1} or \eqref{eq:devsing2}.

We next consider the exceptional situation where $d_- = d_+$. Without
loss of generality we may assume that either $d_n^-=d_n^+$ for all
$n \in \N$, or $d_n^-\neq d_n^+$ for all $n \in \N$.  In the first
case we obtain from \eqref{eq:sviluppo} only one solution $w_n$ of
\eqref{eq:odemodcomp}, but we can construct a second solution by
differentiating \eqref{eq:frobapprox} with respect to the exponent
$d_n$, taking into account the fact that $f_n'(d_n)=0$ since
$d_n = d_n^\pm$ is a double root by assumption. The new solution has
the form
\[
  w_n^\sharp (z) \,=\, \log(z) w_n(z) + z^{d_n} \sum_{k=0}^\infty 
  \Big(\frac{\partial c_{n,k}}{\partial d_n} \Big) z^k\,, \qquad z \in 
  \bbD(0,\rho/2) \setminus \R_-\,,
\]
and its asymptotic behavior as $z \to 0$ is clearly different from
that of $w_n(z)$. This allows us to conclude the proof using the same
argument as above, and we obtain that $u_n \to \alpha_+ \phi_+ + 
\alpha_- \phi_-$ where $\phi_\pm$ are as in \eqref{eq:devsing3}.  On the 
other hand, when $d_n^- \neq d_n^+$ for all $n\in \N$ (but $d_n^- - d_n^+ \to 0$),
the decomposition \eqref{eq:splitting} is not appropriate, because in
that case we cannot prove that the coefficients $\alpha_n^+$ and
$\alpha_n^-$ are bounded, yet alone have limits as $n \to
\infty$. Instead, we write
\begin{equation}\label{eq:splittingbis}
  u_n(r) \,=\, \alpha_n \Big( w_n^+(r-\rb_n) + w_n^-(r-\rb_n)\Big) + 
  \alpha_n^\sharp \Big( \frac{w_n^+(r-\rb_n) - w_n^-(r-\rb_n)}{d_n^+-d_n^-}\Big)\,,   
\end{equation}
for $r \in (\rb-\rho/4,\rb+\rho/4)$, and this new decomposition has
the property that both coefficients $\alpha_n$ and $\alpha_n^\sharp$
necessarily have limits as $n \to \infty$.  We then finish the proof
along the same lines as above. \qed

\subsection{Approximation and interpolation in the class $\WW$}
\label{subsec64}

This section is devoted to the proof of Lemmas~\ref{lem:dense} and
\ref{lem:homotopie}. If $W$ is a vorticity profile that belongs to the
class $\WW$, in the sense of Definition~\ref{def:Hclass}, we denote by
$\Omega$ the corresponding angular velocity given by \eqref{eq:Omrep},
and by $J$ the function defined in \eqref{eq:Jdef}. The first
observation is that both $\Omega$ and $W$ can be expressed in terms 
of the auxiliary function $J$.

Indeed, let $\phi : \Rp \to \R$ be defined by $\phi(r) = 
\Omega(r)/\Omega'(r)$ for all $r > 0$. According to 
\eqref{eq:velovort}, \eqref{eq:Phidef}, \eqref{eq:Jdef} we 
have
\begin{equation}\label{eq:Jrep}
  J(r) \,=\, \frac{\Phi(r)}{\Omega'(r)^2} \,=\, 
  \frac{2r\Omega(r)}{\Omega'(r)} + \frac{4\Omega(r)^2}{\Omega'(r)^2}
  \,=\, 2 r \phi(r) + 4 \phi(r)^2\,,
\end{equation}
for all $r > 0$. Since $J(r) > 0$ and $\phi(r) < 0$ by assumption, 
we deduce that
\begin{equation}\label{eq:phirep}
  \phi(r) \,=\, \frac{\Omega(r)}{\Omega'(r)} \,=\, -\frac14 
  \Bigl(r + \sqrt{r^2 + 4 J(r)}\Bigr)\,, \qquad r > 0\,.  
\end{equation}
Integrating this differential equation and using the normalization 
condition $\Omega(0) = 1$, which follows from \eqref{eq:Omrep} 
since $W(0) = 2$, we obtain the representation formula
\begin{equation}\label{eq:Omegarep}
  \Omega(r) \,=\, \,\exp\Bigl(-\int_0^r \frac{4}{s+\sqrt{s^2+4J(s)}}
  \dd s\Bigr)\,, \qquad r > 0\,.
\end{equation}
As $W(r) = r\Omega'(r) + 2\Omega(r)$, we also have
\begin{equation}\label{eq:Wrep}
  W(r) \,=\, \Omega(r) \Bigl(2 + \frac{r}{\phi(r)}\Bigr) \,=\, 
  \Omega(r) \biggl(2 - \frac{4r}{r + \sqrt{r^2 + 4 J(r)}}\biggr)\,, 
  \qquad r > 0\,.
\end{equation}
Furthermore, if we differentiate \eqref{eq:Wrep} with respect to 
$r$ and observe that $\Omega' = \Omega/\phi$, we obtain
\[
  W' \,=\, \frac{\Omega}{\phi}\Bigl(2 + \frac{r}{\phi}\Bigr) +
  \Omega \Bigl(\frac{1}{\phi} - \frac{r\phi'}{\phi^2}\Bigr) \,=\, 
  \frac{\Omega}{\phi^2}\Bigl(3\phi + r - r\phi'\Bigr)\,.
\]
Thus, using the expression \eqref{eq:phirep} of $\phi$, we find after
elementary calculations
\begin{equation}\label{eq:Wprep}
  W'(r) \,=\, \frac{8\Omega(r)}{\bigl(r + \sqrt{r^2 + 4 J(r)}\bigr)^2}
  \biggl(r - \frac{r^2 + 6J(r)}{\sqrt{r^2 + 4J(r)}} + 
  \frac{r J'(r)}{\sqrt{r^2 + 4J(r)}}\biggr)\,.
\end{equation}
As $r\sqrt{r^2 + 4J} < r^2 + 6J$ when $J > 0$, this formula shows that 
the vorticity $W$ is strictly decreasing as soon as the auxiliary function 
$J$ satisfies $J(r) > 0$ and $J'(r) < 0$ for all $r > 0$. This observation
will be used later on. 

Since $W \in \WW$ by assumption, the angular velocity satisfies 
$\Omega'(r) \to 0$ as $r \to 0$, and in view of \eqref{eq:Jdef} 
or \eqref{eq:phirep} this implies that $J(r) \to \infty$ as 
$r \to 0$. It then follows from \eqref{eq:Wprep} that
\[
  W'(r) \,\sim\, \frac{rJ'(r)}{J(r)^{3/2}} - \frac{6}{J(r)^{1/2}}\,, 
  \qquad \hbox{as}\quad r \to 0\,,
\]
and since $W'(r)$ vanishes at the origin we deduce that
$rJ'(r) J(r)^{-3/2} \to 0$ as $r \to 0$. Concerning the behavior at
infinity, we observe that $\phi(r) = -2/r + \cO(1/r^3)$ as
$r \to \infty$, and in view of \eqref{eq:Omegarep} this implies that
$\Omega(r) = \Gamma/r^2 + \cO(1/r^4)$ as $r \to \infty$, for some
$\Gamma > 0$.  The expression \eqref{eq:Wrep} also shows that
$r^4 W(r) \to 2\Gamma J(\infty)$ as $r \to \infty$. Finally, 
one infers from \eqref{eq:Wprep} that
\[
  W'(r) \,\sim\, \bigg\{\frac{2\Gamma}{r^4} + \cO\Bigl(\frac{1}{r^6}
  \Bigr)\biggr\}\,\biggl\{\frac{rJ'(r)}{\sqrt{r^2 + 4J(r)}} 
  - \frac{4J(r)}{r} + \cO\Bigl(\frac{1}{r^3}\Bigr)\biggr\}\,, 
  \qquad \hbox{as}\quad r \to \infty\,,
\]
and since $rJ'(r) \to 0$ as $r\to \infty$ we also obtain $r^5W'(r) \to -8\Gamma
J(\infty)$ as $r \to \infty$.   

The properties of $J$ are more conveniently expressed in terms
of the new function 
\begin{equation}\label{eq:Qdef}
  Q(r) \,=\, \frac{1}{\sqrt{1+J(r)}}\,, \qquad r > 0\,.
\end{equation}

\begin{df}\label{def:Qclass}
We say that a $\cC^1$ function $Q : \Rp \to (0,1]$ belongs to the 
class $\QQ$ if \\[1mm]
\null~ i) $Q'(r) > 0$ for all $r > 0$; \\[1mm]
\null~ ii) $Q(r) \to 0$ and $rQ'(r) \to 0$ as $r \to 0$; \\[1mm]
\null~ iii) $rQ'(r) \to 0$ as $r \to \infty$.\\[1mm]
In particular $\QQ$ is convex, and if $Q \in \QQ$ then $\cN(Q) := 
\sup_{r > 0} r |Q'(r)| < \infty$. 
\end{df}

\begin{lem}\label{lem:Qclass}
A vorticity profile $W$ belongs to the class $\WW$ in the 
sense of Definition~\ref{def:Hclass} if and only if the 
function $Q$ defined by \eqref{eq:Qdef}, with $J$ as in 
\eqref{eq:Jdef}, belongs to the class $\QQ$. 
\end{lem}

\begin{proof}
If $W \in \WW$, we have just shown that the $\cC^1$ map $J : \Rp \to \R$ 
defined by \eqref{eq:Jdef} satisfies $J(r) > 0$ and $J'(r) < 0$ for
all $r > 0$, $J(r) \to \infty$ and $rJ'(r)J(r)^{-3/2} \to 0$ as $r \to 0$,
and $rJ'(r) \to 0$ as $r \to \infty$. These properties precisely
mean that the function $Q$ defined by \eqref{eq:Qdef} belongs to 
the class $\QQ$.  Conversely, if $Q \in \QQ$, we define $J = Q^{-2} - 1$ 
so that \eqref{eq:Qdef} holds, and we reconstruct the angular 
velocity $\Omega$ and the vorticity $W$ by the formulas
\eqref{eq:Omegarep}, \eqref{eq:Wrep}. The calculations above then 
show that $W \in \WW$. In particular, since $J'(r) < 0$, formula 
\eqref{eq:Wprep} shows that $W'(r) < 0$ for all $r > 0$. 
\end{proof}

The following result expresses the fact that the vorticity profile $W \in \WW$ 
depends continuously on the auxiliary function $Q \in \QQ$,
in appropriate topologies. 

\begin{lem}\label{lem:QLip}
Assume that $Q_1, Q_2 \in \QQ$, and take $\delta > 0$ small enough 
so that 
\begin{equation}\label{eq:deltabound}
  \delta \,\le\, \min\{Q_1(1),Q_2(1)\} \,\le\, \max\{Q_1(1),Q_2(1)\} 
  \,\le\, \sqrt{1 - \delta^2}\,.
\end{equation}
Then there exists a constant $C > 0$, depending only on $\delta$, 
such that, if $W_1, W_2$ denote the vorticity profiles associated
with $Q_1, Q_2$ as in Lemma~\ref{lem:Qclass}, the following 
estimates hold\:
\begin{align}\label{eq:Wdiff}
  \sup_{r > 0} (1+r^4) |W_1(r) - W_2(r)| \,&\le\, C \|Q_1-Q_2\|_{
  L^\infty(\Rp)}\,, \\ \label{eq:Wpdiff}
  \sup_{r > 0}(1+r^5) |W_1'(r) - W_2'(r)| \,&\le\, C \cN(Q_1-Q_2) + C\Bigl(
  1 + \cN(Q_2)\Bigr) \|Q_1-Q_2\|_{L^\infty(\Rp)}\,.
\end{align}
\end{lem}

\begin{proof}
Let $J_i(r) = Q_i(r)^{-2} - 1$ for $i = 1,2$. We first consider the quantity
\[
  \Theta(r) \,=\, \frac{1}{r+\sqrt{r^2+4J_1}} - \frac{1}{r+\sqrt{r^2+4J_2}} 
  \,=\, \frac{4(Q_1^2 - Q_2^2)}{(Q_1\Delta_2 + Q_2\Delta_1) (rQ_1 + \Delta_1) 
  (rQ_2 + \Delta_2)}\,,
\]
where we use the shorthand notation $\Delta_i(r) = Q_i(r)\sqrt{r^2+4J_i(r)} =
\sqrt{4 + (r^2-4) Q_i(r)^2}$ for $i = 1, 2$. We claim that 
\begin{equation}\label{eq:Thetabd}
  \sup_{r > 0}(1+r^3) |\Theta(r)| \,\le\, C \|Q_1-Q_2\|_{L^\infty(\Rp)}\,,
\end{equation}
for some constant $C > 0$ depending only on $\delta$. To prove that, 
we distinguish two cases\:

\noindent
i) In the region where $r \le 1$, we know from \eqref{eq:deltabound} that 
$Q_i(r)^2 \le Q_i(1)^2 \le 1-\delta^2$, and this implies that $\Delta_i(r) \ge 
2\sqrt{1-Q_i(r)^2} \ge 2\delta$. It follows that
\[
  |\Theta(r)| \,\le\, \frac{4(Q_1+Q_2)|Q_1-Q_2|}{8\delta^3
  (Q_1 + Q_2)} \,\le\, \frac{|Q_1 - Q_2|}{2 \delta^3}\,, \quad 
  0 < r \le 1\,.
\]
ii) When $r \ge 1$ we observe that $\Delta_i(r) \ge r Q_i \ge r\delta$ by 
\eqref{eq:deltabound}, so that 
\[
  |\Theta(r)| \,\le\, \frac{4(Q_1+Q_2)|Q_1-Q_2|}{8r^3 Q_1^2 Q_2^2}
  \,\le\, \frac{|Q_1 - Q_2|}{\delta^3 r^3}\,, \quad r \ge 1\,.
\]
Altogether, this proves \eqref{eq:Thetabd}. 

As an immediate consequence, we see that the angular velocities defined 
by \eqref{eq:Omegarep} satisfy the estimate $\|\Omega_1-\Omega_2
\|_{L^\infty(\Rp)} \le C \|Q_1-Q_2\|_{L^\infty(\Rp)}$. In fact, we have 
a stronger result\:
\begin{equation}\label{eq:Omdiff}
  \sup_{r > 0} (1+r^2) |\Omega_1(r) - \Omega_2(r)| \,\le\, C 
  \|Q_1-Q_2\|_{L^\infty(\Rp)}\,.
\end{equation}
Indeed, it follows from \eqref{eq:Omegarep} that
\begin{equation}\label{eq:r2Omega}
  r^2 \Omega_i(r) \,=\, \Omega_i(1)\,\exp \int_1^r 
  \Bigl(\frac{2}{s} - \frac{4}{s+\sqrt{s^2+4J_i(s)}}\Bigr)
  \dd s\,, \qquad r > 0\,, \quad i = 1,2\,.
\end{equation}
If we denote 
\[
  M_i(r) \,=\,\frac{1}{r} - \frac{2}{r+\sqrt{r^2+4J_i}} \,=\, 
  \frac{4 J_i}{r(r+\sqrt{r^2+4J_i})^2} \,=\, 
  \frac{4(1-Q_i^2)}{r(rQ_i + \Delta_i)^2} \,>\,0\,,
  \qquad r > 0\,,
\]
the same estimates as above show that 
\begin{equation}\label{eq:Kibound}
  \sup_{0 < r \le 1} r M_i(r) \,\le\, \frac{1}{\delta^2}\,, \qquad
  \hbox{and}\quad 
  \sup_{r \ge 1} r^3 M_i(r) \,\le\, \frac{1}{\delta^2}\,.
\end{equation}
This implies in particular that $r^2 \Omega_i(r) \le e^{1/\delta^2}$
for $r \ge 1$, hence $(1+r^2)\Omega_i(r) \le C$ for some constant
$C > 0$ depending only on $\delta$. In addition, using
\eqref{eq:Thetabd} and \eqref{eq:r2Omega}, we obtain
\[
  r^2|\Omega_1(r)-\Omega_2(r)| \,\le\, e^{1/\delta^2}\Bigl(
  |\Omega_1(1) - \Omega_2(1)| + 4\int_1^r |\Theta(s)|\dd s\Bigr)
  \,\le\, C \|Q_1-Q_2\|_{L^\infty(\Rp)}\,, 
\]
for all $r \ge 1$, and this concludes the proof of \eqref{eq:Omdiff}. 

On the other hand, in view of \eqref{eq:Wrep}, we have $W_i(r) = 2r 
\Omega_i(r) M_i(r)$ for $i = 1,2$, hence 
\[
  W_1(r) - W_2(r) \,=\, 2r(\Omega_1(r) - \Omega_2(r))M_1(r) 
  -4r \Omega_2(r)\Theta(r)\,, \qquad r > 0\,.
\]
Thus using estimates \eqref{eq:Thetabd}, \eqref{eq:Omdiff}, 
\eqref{eq:Kibound} we arrive at \eqref{eq:Wdiff}.

Finally, we deduce from \eqref{eq:Wprep} that $W_1'(r) - W_2'(r) = 
\Xi(r) - 16\Lambda(r)$, where $\Xi(r)$ collects all terms that 
do not involve the derivatives $J_1', J_2'$, and 
\[
  \Lambda(r) \,=\, \frac{\Omega_1 r Q_1'}{(rQ_1 + \Delta_1)^2
  \Delta_1} - \frac{\Omega_2 r Q_2'}{(rQ_2 + \Delta_2)^2
  \Delta_2} \,=\, \Lambda_1(r) + \Lambda_2(r)\,,
\]
where
\[
  \Lambda_1(r) \,=\, \frac{\Omega_1 r (Q_1'-Q_2')}{(rQ_1 + \Delta_1)^2
  \Delta_1}\,, \qquad 
  \Lambda_2(r) \,=\, \frac{\Omega_1 rQ_2'}{(rQ_1 + \Delta_1)^2
  \Delta_1} - \frac{\Omega_2 r Q_2'}{(rQ_2 + \Delta_2)^2\Delta_2}\,. 
\]
Proceeding exactly as above it is straightforward to verify that
\[
  \sup_{r > 0} (1+r^5) |\Xi(r)| \,\le\, C \|Q_1-Q_2\|_{L^\infty(\Rp)}\,,
\]
where $C > 0$ depends only on $\delta$. We thus concentrate on 
the new term $\Lambda$, which contains the derivatives $Q_1', Q_2'$.  
Again, considering separately the regions where $r \le 1$ and $r \ge 1$, 
and using the appropriate lower bound on $\Delta_1(r)$ in each 
region, we obtain 
\[
  |\Lambda_1(r)| \,\le\, C\Omega_1(r)\frac{r|Q_1'(r)-Q_2'(r)|}{
  1+r^3} \,\le\, C (1+r^2)\Omega_1(r) \frac{\cN(Q_1-Q_2)}{
  1+r^5}\,,
\]
where we recall that $(1+r)^2\Omega_1(r) \le C$ for some $C > 0$ 
depending only on $\delta$. The quantity $\Lambda_2$ is estimated 
along the same lines\:
\[
  |\Lambda_2(r)| \,\le\, C (1+r^2)|\Omega_1(r)-\Omega_2(r)|
  \,\frac{r|Q_2'(r)|}{1+r^5} + C (1+r^2)\Omega_2(r) 
  \,\frac{r|Q_2'(r)|}{1+r^5}\,|Q_1(r)-Q_2(r)|\,, 
\]
and using \eqref{eq:Omdiff} we find $(1+r^5)|\Lambda_2(r)| \le C 
\bigl(1 + \cN(Q_2)\bigr)\|Q_1-Q_2\|_{L^\infty(\Rp)}$. Combining 
these estimates we arrive at \eqref{eq:Wpdiff}. 
\end{proof}

\medskip\noindent{\bf Proof of Lemma~\ref{lem:dense}.}
Let $W \in \WW$, and denote by $Q \in \QQ$ the function defined by 
\eqref{eq:Qdef} with $J$ as in \eqref{eq:Jdef}. For any 
$\epsilon > 0$ we define
\begin{equation}\label{eq:Qepsdef}
  Q^{(\epsilon)}(r) \,=\, \frac{1}{\sqrt{\pi\epsilon}}
  \int_0^\infty \Bigl(e^{-(r-s)^2/\epsilon} - e^{-(r+s)^2/\epsilon}
  \Bigr)Q(s)\dd s\,, \qquad r > 0\,.
\end{equation}
In other words, $Q^{(\epsilon)}$ is the restriction to $\R_+$ of the
real-analytic odd function obtained by extending $Q$ to an odd
function $\bar Q : \R \to \R$ and applying to $\bar Q$ the heat
semigroup on $\R$ at time $t = \epsilon/4$. In particular,
$Q^{(\epsilon)}$ is real-analytic on $\Rpb$ for any $\epsilon > 0$.
Moreover, as the function $s \mapsto \bar Q$(s) is continuous on $\R$
and converges to finite limits as $s \to \pm\infty$, it is clear that
$Q^{(\epsilon)}$ converges uniformly to $Q$ on $\Rp$ as $\epsilon \to 0$. 
On the other hand, differentiating \eqref{eq:Qepsdef} with respect
to $r$, we obtain the relation
\begin{equation}\label{eq:Qepsdiff}
  rQ^{(\epsilon)}{}'(r) \,=\, \frac{1}{\sqrt{\pi\epsilon}}
  \int_0^\infty \Bigl(e^{-(r-s)^2/\epsilon} - e^{-(r+s)^2/\epsilon}
  \Bigr)sQ'(s)\dd s + R_\epsilon(r)\,, \qquad r > 0\,,
\end{equation}
where
\[
  R_\epsilon(r) \,=\, \frac{1}{\sqrt{\pi\epsilon}} \int_0^\infty 
  \Bigl((r-s)e^{-(r-s)^2/\epsilon} + (r+s)e^{-(r+s)^2/\epsilon}
  \Bigr)Q'(s)\dd s\,, \qquad r > 0\,.
\]
As $Q' \in L^1(\Rp)$ and $\cN(Q) < \infty$, it is straightforward to 
verify that $R_\epsilon$ converges uniformly to zero as $\epsilon \to 0$. 
Moreover, since $sQ'(s) \to 0$ as $s \to 0$ and $s \to \infty$, it 
is clear that the integral term in \eqref{eq:Qepsdiff} converges 
uniformly on $\Rp$ towards $rQ'(r)$ as $\epsilon \to 0$. 
Altogether, this shows that
\begin{equation}\label{eq:Qapprox}
  \lim_{\epsilon \to 0} \Bigl(\|Q^{(\epsilon)}-Q\|_{L^\infty(\Rp)} + 
  |\cN(Q^{(\epsilon)} - Q)|\Bigr) \,=\, 0\,.
\end{equation}
Now, let $W^{(\epsilon)} \in \WW$ be the vorticity profile associated 
with $Q^{(\epsilon)}$ as in Lemma~\ref{lem:Qclass}. By construction 
$W^{(\epsilon)}$ is real-analytic on $\Rpb$ for any $\epsilon > 0$,
and it follows from \eqref{eq:Qapprox} and Lemma~\ref{lem:QLip}
that $W^{(\epsilon)} \to W$ in the topology of $\cC^1_b(\Rpb)$ as 
$\epsilon \to 0$. This is the desired density result. \QED

\medskip\noindent{\bf Proof of Lemma~\ref{lem:homotopie}.}
Assume that $W_0, W_1 \in \WW$, and let $Q_1, Q_2 \in \QQ$ be the 
corresponding functions defined as in Lemma~\ref{lem:Qclass}. 
For any $t \in [0,1]$, we define $Q_t \in \QQ$ by the linear interpolation 
formula 
\begin{equation}\label{eq:Qinterp}
  Q_t(r) \,=\, (1-t)Q_0(r) + t Q_1(r)\,, \qquad r > 0\,,
\end{equation}
and we denote $W_t \in \WW$ the vorticity profile associated with $Q_t$. 
As 
\[
  \|Q_{t_1}-Q_{t_2}\|_{L^\infty(\Rp)} + \cN(Q_{t_1}-Q_{t_2}) \,=\, 
  |t_1 - t_2| \Bigl(\|Q_0-Q_1\|_{L^\infty(\Rp)} + \cN(Q_0-Q_1)\Bigr)\,,
\]
it follows from Lemma~\ref{lem:QLip} that $W_t$ is a Lipschitz
function of $t \in [0,1]$ in the topology of $\cC^1_b(\Rpb)$.  By
construction, if $W_0, W_1$ are real analytic on $\Rpb$, so is $W_t$
for any $t \in [0,1]$. Moreover, in that case, if $W_1''(0) \equiv -
8 Q_1'(0) < 0$, then $W_t''(0) \equiv -8 Q_t'(0) < 0$ for all $t \in 
(0,1]$, because $W_0''(0) \equiv -8 Q_0'(0) \le 0$. This concludes 
the proof. \QED

\subsection{Proof of Lemma~\ref{lem:kernel}.}
\label{subsec65}

The proof of estimate \eqref{eq:Knest} is lengthy but rather 
straightforward, and we just indicate here the main steps. 
Using classical estimates for the modified Bessel functions $K_\nu$   
\cite[Section 9.6]{AS}, we first observe that the approximate 
solution $\chi_n$ defined in \eqref{eq:chidef} satisfies
\begin{equation}\label{eq:chiest}
  \begin{split}
  |\chi_n(r)| \,&\approx\, \min(1,|r{-}r_n{+}ic_n|)^{\frac12-\nu_n}
  \,e^{-k(r-r_n)}\,, \\
  |\chi_n'(r)| \,&\approx\, \min(1,|r{-}r_n{+}ic_n|)^{-\frac12-\nu_n}
  \,e^{-k(r-r_n)}\,,
  \end{split}
\end{equation}
for all $r > 0$. Here $A \approx B$ means that the ratio $A/B$ is 
bounded from above and from below by some positive constants that 
are independent of $n$. A direct consequence of \eqref{eq:chiest} 
is\:

\begin{lem}\label{lem:chi}
There exists a constant $C > 0$ such that, for all $n \in \N$ 
and all $s > r > 0$, the following estimates hold\:
\begin{align}\label{eq:intest1}
  |\chi_n(s)|^2 \int_r^s \frac{1}{|\chi_n(t)|^2}\dd t \,&\le\, 
  C \min(1,|s{-}r_n{+}ic_n|)^{1-2\nu_n}\,, \\ \label{eq:intest2}
  |\chi_n(s)\chi_n'(s)| \int_r^s \frac{1}{|\chi_n(t)|^2}\dd t \,&\le\, 
  C \min(1,|s{-}r_n{+}ic_n|)^{-2\nu_n}\,.
\end{align}
\end{lem}

\begin{proof}
We only prove \eqref{eq:intest1}, the proof of \eqref{eq:intest2}
being similar. If $c_n \ge 1$, the first estimate in \eqref{eq:chiest}
simply means that $|\chi_n(r)| \approx e^{-k(r-r_n)}$, and 
\eqref{eq:intest1} follows immediately. Thus we assume henceforth
that $0 < c_n \le 1$, and for simplicity we set $s = r_n + \tau$, so 
that the proof of \eqref{eq:intest1} reduces to showing that
\begin{equation}\label{eq:intest3}  
  \cI_n(\tau) \,:=\, |\chi_n(r_n{+}\tau)|^2 \int_{-\infty}^\tau 
  \frac{1}{|\chi_n(r_n{+}t)|^2}\dd t \,\le\, C \min(1,|\tau{+}ic_n|)^{1-2\nu_n}\,,
  \quad \tau \in \R\,.
\end{equation}  
If $\tau \le -1$, we know from \eqref{eq:chiest} that $|\chi_n(r_n{+}\tau)| 
\le C\,e^{-k\tau}$ and $|\chi_n(r_n{+}t)| \ge C\,e^{-kt}$ for $t \le \tau$, 
hence $\cI_n(\tau) \le C$. If $|\tau| \le 1$, then $|\chi_n(r_n{+}\tau)| 
\le C |\tau{+}ic_n|)^{\frac12-\nu_n}$, and 
\[
  \cI_n(\tau) \,\le\, C\,|\tau{+}ic_n|^{1-2\nu_n} \biggl(\int_{-\infty}^{-1} 
  e^{2k t}\dd t + \int_{-1}^\tau |t{+}ic_n|^{2\nu_n-1}\dd t\biggr)
  \,\le\, C\,|\tau{+}ic_n|^{1-2\nu_n}\,.
\]
Finally, if $\tau \ge 1$, then
\[
  \cI_n(\tau) \,\le\, C\,e^{-2k\tau} \biggl(\int_{-\infty}^{-1} 
  e^{2k t}\dd t + \int_{-1}^1 |t{+}ic_n|^{2\nu_n-1}\dd t
  + \int_1^\tau e^{2kt}\dd t\biggr) \,\le\, C\,,
\]
and this completes the proof of \eqref{eq:intest1}. We observe that
the constant $C$ in the right-hand side is independent of $n$ 
because, as can be seen from \eqref{eq:nundef}, the exponent
$\nu_n$ is bounded away from zero as $n \to \infty$. 
\end{proof}

\begin{rem}\label{rem:nuzero}
In Section \ref{subsec49} we use the fact that if $\nu_n \to 0$ as $n \to \infty$, the conclusion of Lemma~\ref{lem:chi} 
remains valid up to a logarithmic correction. Estimates \eqref{eq:chiest} 
are not appropriate in that case, but one can use the fact that $K_{\nu_n}(z)$ 
is close to $K_0(z)$ for $n$ large, where $K_0(z) \approx -\log(z)$ 
as $z \to 0$. In particular, the integral 
\[
  \int_{r_n-1}^{r_n+1} \frac{1}{\chi_n(r)^2}\dd r \,=\, \frac{\pi}{2}
  \int_{-1}^1 \frac{1}{(t+ic_n)\,K_{\nu_n}(k(t+ic_n))^2}\dd t
\]
is uniformly bounded even if $\nu_n \to 0$ and $c_n \to 0$, because 
the function $t \mapsto |t|^{-1} |K_0(kt)|^{-2}$ is integrable 
over $[-1,1]$. We leave the details to the reader. 
\end{rem}

Returning to the proof of Lemma~\ref{lem:kernel}, we note that
$\cK_n(r,s) = \cK_n^{(1)}(r,s) - \cK_n^{(2)}(r,s)$, where 
\begin{align*}
  \cK_n^{(1)}(r,s) \,&=\, \cA(s)\chi_n(s)^2 \cR_n(s) \int_r^s 
  \frac{1}{\cA(t)\chi_n(t)^2}\dd t\,, \\
  \cK_n^{(2)}(r,s) \,&=\, \cA'(s)\chi_n(s)\chi_n'(s) \int_r^s 
  \frac{1}{\cA(t)\chi_n(t)^2}\dd t\,.
\end{align*}
In what follows, we assume that $s \ge r \ge \delta r_n$ for some 
$\delta > 0$, where $r_n = (\Gamma_n/b_n)^{1/2} \to \infty$ as 
$n \to \infty$. Since $\cA(s) \to 1/k^2$ and $\cA'(s) = \cO(s^{-3})$ 
as $s \to \infty$, it follows immediately from \eqref{eq:intest2}
that
\begin{equation}\label{eq:Kn2} 
  \sup_{r \ge \delta r_n} \int_{r}^\infty |\cK_n^{(2)}(r,s)| \dd s \,\le\, 
  C \int_{\delta r_n}^\infty s^{-3} \min(1,|s{-}r_n{+}ic_n|)^{-2\nu_n}\dd s
  \,\le\, \frac{C}{r_n^2} \,\le\, C b_n\,.
\end{equation}
It thus remains to estimate the kernel $\cK_n^{(1)}(r,s)$, which 
involves the remainder $\cR_n$ defined in \eqref{eq:Rndef}. We can 
further decompose $\cR_n = \cR_n^{(1)} + \cR_n^{(2)} + \cR_n^{(3)}$, where
\begin{align*}
  \cR_n^{(1)}(r) \,&=\, \frac{m^2+\frac34}{r^2} - \frac{1}{2r}\frac{
  \cA'(r)}{\cA(r)}\,, \qquad \cR_n^{(2)}(r) \,=\, -\frac{k^2}{m^2}\biggl(
  \frac{\Phi_n(r)}{\gamma_n(r)^2} - \frac{J_n(r_n)}{(r-r_n+ic_n)^2}\biggr)\,,\\
   \cR_n^{(3)}(r) \,&=\, \frac{r}{\cA(r)\gamma_n(r)}
  \partial_r\biggl(\frac{W_n(r)}{m^2+k^2 r^2}\biggr)\,.
\end{align*}
We concentrate here on the term $\cR_n^{(2)}$, which gives the main 
contribution to the integral \eqref{eq:Knest}. 

\begin{lem}\label{lem:R2}
For any $\delta > 0$, there exists a constant $C > 0$ such that
\begin{equation}\label{eq:R2est}
  |\cR_n^{(2)}(r)| \,\le\, \frac{C}{r\,|r{-}r_n{+}ic_n|}\,, \qquad 
  \hbox{for all } r \ge \delta r_n\,.
\end{equation}
\end{lem}

\begin{proof}
  We first assume that $|r-r_n| \le r_n/2$. Using the same notations
  as in \eqref{eq:gammaexp}, we have
\[
  \gamma_n(r) \,=\ \Omega_n(r)-\Omega_n(r_n)-ia_n \,=\, \Omega_n'(r_n)
  \bigl(r - r_n + i c_n\bigr) + \bigl(\Omega'(\xi)-\Omega'(r_n)\bigr)
  (r-r_n)\,,
\]
for some $\xi \in [r_n/2,3r_n/2]$. As $\Omega_n'(s) \approx s^{-3}$ 
and $\Omega_n''(s) \approx s^{-4}$ as $s \to \infty$ (see 
\eqref{eq:Jinfinipositif} and the first part of Section \ref{subsec64}), 
we have
\[
  |\Omega_n'(\xi)-\Omega_n'(r_n)| \,\le\, \frac{C}{r_n}\,|\Omega_n'(r_n)| 
  \,|\xi-r_n| \,\le\, \frac{C}{r_n}\,|\Omega_n'(r_n)|\,|r-r_n|\,,
\]
so that $\gamma_n(r) = \Omega_n'(r_n)(r - r_n + i c_n)[1 + \cO(
\epsilon(r))]$, where $\epsilon(r) = (r-r_n)/r_n$. Recalling that 
 $J_n(r_n) = \Phi_n(r_n)/\Omega_n'(r_n)^2$, we obtain the following
expression\:
\[
  \cR_n^{(2)}(r) \,=\, -\frac{k^2}{m^2}\,\frac{\Phi_n(r)[1 + \cO(
  \epsilon(r))] - \Phi_n(r_n)}{\Omega'(r_n)^2 (r - r_n + i c_n)^2}\,.
\]
Now, it also follows from \eqref{eq:Jinfinipositif} and the first part
of Section \ref{subsec64} that $\Phi_n(s) \approx s^{-6}$ and
$\Phi_n'(s) \approx s^{-7}$ as $s \to \infty$. This implies that
$|\Phi_n(r) - \Phi_n(r_n)| \le C \Phi_n(r_n) \epsilon(r)$, and we
deduce
\[
  |\cR_n^{(2)}(r)| \,\le\, C J_n(r_n) \,\frac{\epsilon(r)}{|r{-}r_n{+}i 
  c_n|^2} \,\le\, \frac{C}{r_n\,|r{-}r_n{+}ic_n|}\,, \qquad 
  \hbox{if} \quad \frac{|r-r_n|}{r_n} \,\le\, \frac12\,.
\]

The argument is simpler if $|r_n-r|\ge r_n/2$, because we can estimate
both terms in $\cR_n^{(2)}$ separately. Straightforward calculations
lead to the lower bound
\begin{equation}\label{eq:gammanlow}
  |\gamma_n(r)| \,=\, |\Omega_n(r) - \Omega_n(r_n) -i a_ n|
  \,\ge\, C |\Omega_n'(r)|\,|r-r_n + ic_n|\,,
\end{equation}
	whenever $r \ge \delta r_n$ (here $C$ depends on $\delta$), and this implies that
\[
  |\cR_n^{(2)}(r)| \,\le\, C\,\frac{\Phi_n(r)}{|\gamma_n(r)|^2} + C\,
  \frac{J_n(r_n)}{|r{-}r_n{+}ic_n|^2} \,\le\, \frac{C}{ 
  |r{-}r_n{+}ic_n|^2}\,, \qquad \hbox{if} \quad \frac{|r-r_n|}{r_n} 
  \,\ge\, \frac12\,.
\]
The proof of \eqref{eq:R2est} is thus complete. 
\end{proof}

It is now easy to conclude the proof of Lemma~\ref{lem:kernel}. 
The term $\cR_n^{(3)}$ in the remainder is estimated using the 
lower bound \eqref{eq:gammanlow}, which leads to
\begin{equation}\label{eq:R3est}
  |\cR_n^{(3)}(r)| \,\le\, \frac{C}{r^3\,|r{-}r_n{+}ic_n|}\,, \qquad 
  r \ge \delta r_n\,,
\end{equation} 
whereas $|\cR_n^{(1)}(r)| \le C/r^2$ for all $r > 0$. In view of 
\eqref{eq:intest1}, we thus have
\[
  |\cK_n^{(1)}(r,s)| \,\le\, C \min(1,|s{-}r_n{+}ic_n|)^{1-2\nu_n}
  \,\biggl(\frac{1}{s^2} + \frac{1}{s\,|s{-}r_n{+}ic_n|}\biggr)\,,
  \qquad s \ge r \ge \delta r_n\,.
\]
Integrating this bound, we arrive at
\begin{equation}\label{eq:Kn1} 
  \sup_{r \ge \delta r_n} \int_{r}^\infty |\cK_n^{(1)}(r,s)| \dd s \,\le\,
  \frac{C}{r_n}\Bigl(1 + \log_+(r_n)\Bigr) \,\le\, C b_n^{1/2}
  \Bigl(1 + \log_+\frac{1}{b_n}\Bigr)\,,
\end{equation}
and estimate \eqref{eq:Knest} follows immediately from 
\eqref{eq:Kn2}, \eqref{eq:Kn1}. \QED

\subsection{Proof of equality \eqref{eq:limBessel}.}
\label{subsec66}

Assume that $0 < \nu < 1/2$. Given any $\epsilon > 0$ and $a > 0$, we 
define
\begin{equation}\label{eq:Jnu1} 
  \cJ_\nu(a) \,=\, \int_{-\epsilon}^\epsilon \frac{-ax}{(a^2+x^2)^{3/2}}
  \,|K_\nu(x+ia)|^2\dd x\,,
\end{equation}
where $K_\nu$ is the modified Bessel function. Our goal is to compute 
the limit of $\cJ_\nu(a)$ as $a \to 0$. We recall \cite[Section~9.6]{AS} 
that
\[
  K_\nu(z) \,=\, \frac{\pi}{2}\,\frac{I_{-\nu}(z)-I_\nu(z)}{\sin(\nu \pi)}\,,
  \qquad \hbox{where}\quad I_\nu(z) \,=\, \frac{1}{\Gamma(1{+}\nu)}\,
  \frac{z^\nu}{2^\nu}\Bigl(1 + \cO(z^2)\Bigr) \quad \hbox{as } z \to 0\,.
\]
We thus have
\[
  K_\nu(z) \,=\, \frac{c_\nu}{z^\nu}\Bigl(1 - d_\nu z^{2\nu} + 
  \cO(z^2)\Bigr) \quad \hbox{as } z \to 0\,,
\]
where
\[
  c_\nu \,=\, \frac{\pi}{\sin(\nu\pi)}\,\frac{2^{\nu-1}}{\Gamma(1{-}\nu)}\,,
  \qquad d_\nu \,=\, \frac{1}{2^{2\nu}}\,\frac{\Gamma(1{-}\nu)}{\Gamma(1{+}\nu)}\,.
\]
It follows that
\begin{align}\nonumber
  |K_\nu(x+ia)|^2 \,&=\, \frac{c_\nu^2}{(a^2+x^2)^\nu}\Bigl(
  \big|1 - d_\nu(x+ia)^{2\nu}\bigr|^2 + \cO(x^2+a^2) \Bigr) \\ \label{eq:Jnu2} 
  \,&=\, \frac{c_\nu^2}{(a^2+x^2)^\nu}\Bigl(1 - 2d_\nu (a^2+x^2)^\nu
  \cos(2\nu\arg(x+ia)) + \cO((x^2+a^2)^{2\nu})\Bigr)\,
\end{align}
as $z = x + i a \to 0$. The leading term in \eqref{eq:Jnu2} is even and
therefore does not contribute to the integral \eqref{eq:Jnu1}, where it is multiplied by an odd function. The main 
contribution comes from the next term, so that 
\begin{align}\nonumber
  \lim_{a\to 0} \cJ_\nu(a) \,&=\, 2 d_\nu c_\nu^2 \lim_{a \to 0}
  \int_{-\epsilon}^\epsilon  \frac{ax}{(a^2+x^2)^{3/2}}\,
  \cos(2\nu\arg(x+ia))\dd x \\ \label{eq:Jnu3}
  \,&=\, 2 d_\nu c_\nu^2 \int_\R \frac{y}{(1+y^2)^{3/2}}\,
  \cos(2\nu\arg(y+i))\dd y\,,
\end{align}
where the second equality is obtained by setting $x = ay$.

It remains to perform the integral in \eqref{eq:Jnu3}. As 
$\arg(y+i) = \pi/2 - \arctan(y)$, we have
\[
  \cos(2\nu\arg(y+i)) \,=\, \cos(\nu\pi)\cos(2\nu\arctan(y)) + 
  \sin(\nu\pi)\sin(2\nu\arctan(y))\,.
\]
The first term does not contribute to the integral in 
\eqref{eq:Jnu3}, whereas setting $y = \tan(t)$ we find
\[
  \int_\R \frac{y}{(1+y^2)^{3/2}}\,\sin(2\nu\arctan(y))\dd y
  \,=\, 2 \int_0^{\pi/2} \sin(t)\sin(2\nu t)\dd t \,=\, 
  \frac{4\nu\cos(\nu \pi)}{1-4\nu^2}\,.
\]
Summarizing, we have shown that
\[
  \lim_{a \to 0} \cJ_\nu(a) \,=\,  2 d_\nu c_\nu^2 \sin(\nu\pi)
  \,\frac{4\nu\cos(\nu \pi)}{1-4\nu^2} \,=\, 
  \frac{2\pi\cos(\nu \pi)}{1-4\nu^2}\,.
\]

\subsection{Explicit calculations in some particular cases}
\label{subsec67}

We collect in this section a few results for the Kaufmann-Scully vortex 
\eqref{eq:KSvortex} and the Lamb-Oseen vortex \eqref{eq:LOvortex} which 
can be established by a direct calculation. 

\medskip\noindent
{\bf 1.} We first show that the vorticity profile $W$ of the 
Lamb-Oseen vortex satisfies assumption H2 in Section~\ref{subsec12}, 
hence belongs to the class $\WW$. Indeed, in that case, the function 
$J$ defined by \eqref{eq:Jdef} is given by
\[
  J(r) \,=\, \frac{r^4 \,e^{-r^2}(1-e^{-r^2})}{\bigl(1-(1+r^2)\,e^{-r^2}
  \bigr)^2}\,, \qquad r > 0\,,
\]
so that
\[
  J'(r) \,=\, \frac{2 r^3\,e^{-r^2}}{\bigl(1-(1+r^2)\,e^{-r^2}
  \bigr)^3}\,\Bigl(2 - r^2 - (4-r^2+r^4)\,e^{-r^2} + 2\,e^{-2 r^2}\Bigr)\,, 
  \qquad r > 0\,.
\]
We want to show that $J'(r) < 0$ for all $r > 0$. Setting $s = r^2$, 
we have to verify that 
\[
  s - 2 + (4-s+s^2)\,e^{-s} - 2\,e^{-2 s} \,>\, 0\,, \qquad s > 0\,,
\]
or equivalently
\[
  \frac{s}{1-e^{-s}} + \Bigl( \frac{s}{1-e^{-s}}\Bigr)^2 e^{-s} \,>\, 2\,,
  \qquad s > 0\,.
\]
Since $s > 1 - e^{-s}$, it is sufficient to show that 
\[
  \frac{s}{\tanh(s/2)} \,\equiv\, \frac{s}{1-e^{-s}} + \frac{s\,e^{-s}}{1-e^{-s}} 
  \,>\, 2\,, 
\]
which is indeed true because $\tanh(x) < x$ for all $x > 0$. 

\medskip\noindent
{\bf 2.} Next, for the Lamb-Oseen vortex, we establish the 
inequality $\cB(r) \ge 1 - 4/m^2$ when $a = 0$ and $b \le 0$, 
where $\cB$ is defined in \eqref{eq:ABdef}. Indeed, as $\ggamma(r) = 
\Omega(r) - b \ge \Omega(r)$, we have 
\begin{equation}\label{eq:BBlow1}
  \cB(r) \,\ge\, 1 - \frac{1}{m^2}\,\frac{k^2 r^2}{m^2 + k^2r^2}
  \,\frac{\Phi}{\Omega^2} + \frac{r}{\Omega}\,\frac{W'}{m^2 + k^2r^2}
  - \frac{2k^2 r^2}{\Omega}\,\frac{W}{(m^2 + k^2r^2)^2}\,.
\end{equation}
As $\Phi = 2\Omega W$ and $W'(r) = -2r W(r)$ for the Lamb-Oseen 
vortex, inequality \eqref{eq:BBlow1} can be written in the equivalent 
form
\begin{equation}\label{eq:BBlow2}
  \cB(r) \,\ge\, 1 - \frac{2W(r)}{m^2\Omega(r)}\biggl\{1 + 
  \frac{m^2 r^2}{m^2 + k^2r^2} - \frac{m^4}{(m^2 + k^2r^2)^2} \biggr\}\,.
\end{equation}
Denote $s = r^2 > 0$ and $\alpha = m^2/(m^2{+}k^2r^2) \in (0,1)$. 
As $W = 2\,e^{-s}$ and $\Omega = s^{-1}(1-e^{-s})$, it is straightforward 
to verify that \eqref{eq:BBlow2} implies the desired inequality 
 $\cB(r) \ge 1 - 4/m^2$ provided
\begin{equation}\label{eq:BBlow3}
  \frac{s\,e^{-s}}{1 - e^{-s}}\bigl(1 + \alpha s - \alpha^2\bigr) 
  \,\le\, 1\,, \qquad \hbox{or equivalently} \quad 
  e^s \,\ge\, 1 + s + \alpha s(s-\alpha)\,.
\end{equation}
But 
\[
  \sup_{\alpha \in (0,1)} \Bigl(1 + s + \alpha s(s-\alpha)\Bigr) 
  \,=\, \begin{cases} 1 + s + s^3/4 & \hbox{if } s \le 2\,,\\
  1 + s^2  & \hbox{if } s \ge 2\,,\end{cases}
\]
and we conclude that the last inequality in \eqref{eq:BBlow3} 
holds in all cases. 

\medskip\noindent
{\bf 3.} Finally, we establish the same inequality $\cB(r) \ge 1 - 4/m^2$ 
for the Kaufmann-Scully vortex. In that case $W'(r) = -4r W(r)\Omega(r)$ 
by \eqref{eq:KSvortex}, hence inequality \eqref{eq:BBlow1} takes the 
form
\begin{equation}\label{eq:BBlow4}
  \cB(r) \,\ge\, 1 - \frac{2W(r)}{m^2\Omega(r)}\biggl\{1 + 
  \frac{2r^2}{1+r^2}\,\frac{m^2}{m^2 + k^2r^2} - \frac{m^4}{(m^2 + k^2r^2)^2} 
  \biggr\}\,.
\end{equation}
Setting again  $\alpha = m^2/(m^2{+}k^2r^2)$ and using the fact that
$W = 2\Omega^2 = 2/(1+r^2)$ in the present case, we see that the desired 
inequality $\cB(r) \ge 1 - 4/m^2$ holds provided
\begin{equation}\label{eq:BBlow5}
  1 + \frac{2\alpha r^2}{1+r^2} - \alpha^2 \,\le\, 1+r^2\,, \qquad r > 0\,.
\end{equation}
The maximum of the left-hand side, considered as a function of $\alpha \in 
(0,1)$, is reached at the point $\alpha = r^2/(1+r^2)$, and the 
resulting inequality becomes $(1+r^2)^2 + r^4 \le (1+r^2)^3$, which is 
of course true. This concludes the proof. 


\end{document}